\documentclass[a4paper,12pt]{amsart}
\usepackage{amsmath}
\usepackage{amsthm}
\usepackage{amssymb, mathrsfs}
\usepackage{fullpage}
\usepackage{bbm}
\usepackage{mathtools}
\usepackage{pifont}
\usepackage{verbatim}

\usepackage{mathtools}
\DeclarePairedDelimiter\ceil{\lceil}{\rceil}
\DeclarePairedDelimiter\floor{\lfloor}{\rfloor}

\mathtoolsset{showonlyrefs}

\newtheorem{thm}{Theorem}
\newtheorem{rem}{Remark}
\newtheorem*{remark}{Remark}
\newtheorem{prop}{Proposition}
\newtheorem{lem}{Lemma}

\newcommand{\beql}[1]{\begin{equation}\label{#1}}
\newcommand{\eeq}{\end{equation}}

\newcommand{\sumstar}{\;\sideset{}{^*}\sum}
\newcommand{\sumb}{\sideset{}{^\flat}\sum}

\newcommand{\sumc}{\sideset{}{^\clubsuit}\sum}

\newcommand{\sumone}{\sideset{}{^1}\sum}

\newcommand{\sumsharp}{\sideset{}{^{\sharp}}\sum}

\newcommand{\sumd}{\sideset{}{^d}\sum}

\newcommand{\bx}{\mathbf{x}}
\newcommand{\fra}[1]{\mathfrak{#1}}
\newcommand{\fC}{\mathfrak{C}}

\newcommand{\fw}{\mathfrak{w}}
\newcommand{\fS}{\mathfrak{S}}
\newcommand{\cO}{\mathcal{O}}

\newcommand{\tRe}{\textup{Re }}
\newcommand{\tIm}{\textup{Im }}
\newcommand{\bfrac}[2]{\left(\frac{#1}{#2}\right)}

\newcommand{\V}{\mathcal V }
\newcommand{\E}{\mathcal E}

\newcommand{\B}{\mathcal B}
\newcommand{\Q}{\mathcal Q}
\newcommand{\C}{\mathcal C}
\newcommand{\A}{{\mathcal A}}
\newcommand{\R}{\mathcal R}
\newcommand{\Scal}{\mathcal S}
\newcommand{\T}{\mathcal T}
\newcommand{\U}{\mathcal U}
\newcommand{\I}{\mathcal I}

\newcommand{\D}{\mathcal D}
\newcommand{\N}{\mathcal N}

\newcommand{\Cf}{\mathfrak C}

\begin{document}
\title{Prime values of a sparse polynomial sequence}
\author{Xiannan Li}
\address{Mathematics Department\\
	138 Cardwell Hall\\
	Manhattan, KS 66506}
\email{xiannan@math.ksu.edu}


\begin{abstract} 
A distinguishing feature of certain intractable problems in prime number theory is the sparsity of the underlying sequence.  Motivated by the general problem of finding primes in sparse polynomial sequences, we give an estimate for the number of primes of the shape $x^3 + 2y^3$ where $y$ is small.  
\end{abstract}

\maketitle
\section{Introduction}
In this paper, we are interested in prime values of polynomials.  A simple and quite classical type of question asks whether a given polynomial $P$ takes on infinitely many prime values.  When $P$ is a linear polynomial in one variable, this problem was solved by Dirichlet, and the question for higher degree polynomials of one variable remain a deep open problem.  In particular, one of Landau's famous problems on primes asks for a proof that there are infinitely many primes of the form $a^2+1$, for integer $a$.

Relatively recently, remarkable results have appeared on prime values of polynomials of two variables.  Here, a classical result is that there are infinitely many primes of the form $a^2+b^2$.  Indeed, by a result of Fermat, primes of that form are essentially the same as primes of the form $4n+1$, so that this reduces to a special case of Dirichlet's theorem.  Interestingly, Fouvry and Iwaniec \cite{FouI} were able to understand primes of the form $a^2 + b^2$ where $b$ is also prime.  This was generalized very recently by Lam, Schindler and Xiao \cite{LSX}.  

Now define the exponential density of the sequence of values of the polynomial $P(a, b)$ to be 
\begin{equation}
\inf\{\lambda: \#\{(a, b) \in \mathbb{N}^2: P(a, b) \leq x\} \ll x^\lambda\}.
\end{equation}  For instance, the density of the sequences defined by $a^2+b^2$ and the aforementioned restricted form studied by Fouvry and Iwaniec \cite{FouI} are $1$, the same as the set of all natural numbers.   

It is much more challenging to prove a similar result when the
sequence given by $P(a, b)$ has density less than $1$.  The first
result in this direction was the breakthrough of Friedlander and
Iwaniec \cite{FI} on the prime values of $a^2+b^4$, which was followed by the result of Heath-Brown \cite{HB} on prime values of $a^3+2b^3$.  This was later generalized by Heath-Brown and Moroz for more arbitrary cubic forms in \cite{HBM}.  Heath-Brown and the author recently proved an analogous result on prime values of the form $a^2 +p^4$ where $p$ is prime \cite{HBL}.  Very recently, J. Maynard generalized Heath-Brown's result to similar restrictions of norms forms which are not too sparse in \cite{JM},

The sparsest sequence for which we have this type of result is that of $a^3+2b^3$, which has exponential density $2/3$.  In contrast the exponential density of the sequence from Landau's problem on $n^2+1$ has exponential density $1/2$.  It turns out that our current methods fail in numerous places once the exponential density drops below $2/3$.  The purpose of this work is to investigate a sequence with density somewhere between $1/2$ and $2/3$ and to illuminate some of the structural differences.

The proofs of these results broadly depend on two types of estimates.  The first, sometimes referred to as Type I estimates, gives information on the behavior of these sequences in arithmetic progressions on average.  The second, sometimes referred to as Type II estimates, involves bounds on certain bilinear sums attached to these sequences.  Achieving the latter  type of estimate is the most difficult part and is the ingredient which breaks the parity barrier.

With current methods, in order to understand such bilinear sums, it is crucial that these sequences are all special values of norm forms of some number field.  Given this, there are two main factors which affect the difficulty of the problem.  The first, already alluded to above, is that the problem tends to be more difficult the sparser the sequence.  The second is that for certain homogeneous polynomials, such as $a^3+2b^3$, estimating the bilinear sum involves a restriction of a variable to a one dimensional lattice, and this makes the problem more tractable.  This is an important structural advantage in Heath-Brown's work on $a^3+2b^3$, which is the sparsest such sequence for which we have such a result.  

Both the asymptotic sieve for primes from Friedlander and Iwaniec \cite{FI3} and Harman's alternative sieve \cite{Har} as used in Heath-Brown's work \cite{HB} fail to prove asymptotic estimates for sequences with exponential density strictly lower than $2/3$.  Nonetheless, we can still count primes in a sparser sequence.  Here, we do not ask for asymptotics, but rather estimates of the right order of magnitude.  To be precise, we prove the following result.

\begin{thm}\label{thm:1}
Let $X \ge 3$, $Y = X^{1-\gamma}$.  There exists an absolute constant $B_0>0$ such that for $\eta = \frac{1}{\log^{B_0} X} > 0$ and for all $0<\gamma < \frac{5}{67}$,
\begin{equation}
\#\{x^3+2y^3: x \in (X, X(1+\eta)], y\in (Y, Y(1+\eta)) \textup{ and } x^3+ 2y^3 \textup{ is prime}\}
\asymp \frac{\eta^2 XY}{\log X},
\end{equation} where the integers in $\{x^3+2y^3: x \in (X, X(1+\eta)], y\in (Y, Y(1+\eta))\}$ are counted with multiplicity, and $X$ is sufficiently large in terms of $B_0$.
\end{thm} 
Thus, the exponential density of our sequence is $2/3 - \gamma/3$ in contrast to Heath-Brown's work \cite{HB}, where the exponential density is $2/3$.  The bound $\gamma < 5/67$ can be improved with more attention to numerical optimization, but that will not be the focus of the current work.  The parameter $B_0$ is introduced for technical convenience; clearly Theorem \ref{thm:1} with larger values of $\eta$ is implied by our Theorem.

Our methods depart from the work of Heath-Brown \cite{HB} in two ways.  The first, which enables us to try to understand sparser sequences, is simply that we neglect certain difficult regions at the expense of sacrificing asymptotic information for lower bounds.  The second and more substantial change occurs in the treatment of the bilinear sums.  To start, we need to identify certain narrow regions of interest and treat them accordingly.  Then congruence problems restricted to narrow regions occurs here, and we anticipate that it occurs in problems involving other sparse sequences given by polynomials.  To be more precise, one may be interested in understanding sums of the type
\begin{equation}
\sum_{\substack{\beta_1 \in \C_1, \beta_2 \in C_2\\ \beta_1 \equiv \lambda \beta_2 \bmod D}} f(\beta_1)f(\beta_2),
\end{equation}on average over $\lambda \bmod D$, where $\C_1$ and $\C_2$ are cubes in $\mathbb{R}^n$ of side $S_0$, $f$ are somewhat arbitrary coefficients and $D$ is large compared to $S_0$.  It is reasonable to expect the number of points satisfying the congruence to be around $S_0^6/D^3$ on average over $\lambda \bmod D$ as long as $D < S_0^{2-\epsilon}$, and this can be proven when $n=1$ (see \cite{HBL}), but this is a challenge to understand for larger $n$ when $D > S_0$.

This also makes it necessary for us to understand results about arithmetic sequences in small regions to small moduli (see Lemma \ref{lem:SWbdd}).  In this direction, we are able to prove a result analogous to essentially the best primes in short intervals result, namely that there is the expected number of primes in intervals of the form $\left(x, x+x^{\frac{7}{12} + \epsilon}\right)$.  Since our sequence is not merely a sequence of prime ideals, we do not go through zero density estimates, but rather attack the problem directly via Heath-Brown's generalization of Vaughan's identity.  It turns out that the quality of this auxiliary result actually limits the quality of the main result in Theorem \ref{thm:1}.  \footnote{Specifically, in our treatment, we neglect two regions.  In one of the regions, the $7/12$ exponent limits the (logarithmic) width of the region to be less than $5/67$.  As noted before, it is still possible to squeeze out some numerical improvements by working on the other region, but we shall not focus on that here.}  
\\\\
{\bf Acknowledgement.} I would like to thank Professor Heath-Brown for many stimulating conversations.  I am also grateful to the anonymous referees for their careful reading of the paper and helpful editorial comments.  This work was partially supported by EPSRC grant EP/K021132X/1, a KSU Startup Grant and a Simons Foundation Collaboration Grant (524790).

\section{Notation and outline of the proof}\label{sec:notationoutline}
Here, we introduce notation and provide an outline of the proof.  Proofs of this form involve many technical estimates, and we will refer the reader to the previous work of Heath-Brown \cite{HB} where appropriate.  

For this paper, we will always let $\epsilon$ denote an arbitrary small positive number, which is not necessarily the same from line to line.  To be precise, a statement $P(\epsilon)$ should be interpreted as "for all sufficiently small $\epsilon>0$, $P(\epsilon)$ holds".

For this paper, we fix the number field $K = \mathbb{Q}(\sqrt[3]{2})$ and its ring of integers $\cO_K$.  We write $I$ and $J$ for integral ideals of $\cO_K$, and write $(x+y\sqrt[3]{2} + z\sqrt[3]{4})$ for the ideal generated by $x+y\sqrt[3]{2} + z\sqrt[3]{4} \in \cO_K$.  Further, let $N$ be the norm form from $K$ to $\mathbb{Q}$.  For future use, let us also define
\begin{equation}\label{eqn:eps0}
\epsilon_0 = 1+ \sqrt[3]{2} + \sqrt[3]{4}
\end{equation} 
and note that $\epsilon_0$ is the fundamental unit of $K$.

We will write $(x, y)$ to be the greatest common divisor of $x$ and $y$.  In the latter half of the paper, we will also use $(a, b, c)$ to denote an element in $\mathbb{R}^3$.  Note that $(a, b, c)$ will never denote the greatest common divisor of $a, b$ and $c$.  Further, we will never need to define an ideal generated by two or more elements.

Unfortunately, it is also common to write $(X, X+Y)$ to denote an open interval with endpoints $X$ and $X+Y$.  To avoid confusion, in this paper, we will use the less common notation $]X, X+Y[$ to denote the open interval with endpoints $X$ and $X+Y$ and similarly $]X, X+Y]$ to denote the half open interval excluding $X$ but including $X+Y$.


Since $K$ has class number one, we are able to pass from ideals to the elements which generate those ideals fairly easily.  Of course, it is possible for two distinct elements $x+y\sqrt[3]{2} + z\sqrt[3]{4}, x'+y'\sqrt[3]{2} + z'\sqrt[3]{4} \in \mathbb{Z}[\sqrt[3]{2}]$ to be associates and thus generate the same ideal.  However, as in Heath-Brown's work \cite{HB}, we will construct our sets to avoid this problem.  To be specific, we will be examining 
\begin{equation}
\A = \{(x+y\sqrt[3]{2}) : x \in ]X, X(1+\eta)], y\in ]Y, Y(1+\eta)], x, y \in \mathbb{Z}, (x, y) = 1\}
\end{equation}and
\begin{equation}
\B = \{J: N(J) \in ]3X^3, 3X^3(1+\eta)[\}.
\end{equation}
Recall $\eta = \frac{1}{\log^{B_0} X}$ and $X \ge 3$.  We will choose $B_0 \ge 1$ so that $\eta < 1$, and no two $x+y\sqrt[3]{2}$ occurring in the definition of $\A$ are associates, and so $A$ contains distinct ideals.  Specifically,  the fundamental unit in \eqref{eqn:eps0} satisfies $|\epsilon_0| = |1+ \sqrt[3]{2} + \sqrt[3]{4}| > 3$ so any unit $u$ with $|u|\neq 1$ satisfies either $|u|> 3$ or $|u| < 1/3$.  On the other hand, for $x \in ]X, X(1+\eta)], y\in ]Y, Y(1+\eta)]$, 
\begin{equation}\label{eqn:multiplyunit}
(X+Y\sqrt[3]{2}) < |x+y\sqrt[3]{2}| \leq (X+Y\sqrt[3]{2})(1+\eta)
\end{equation}
while $u(x+y\sqrt[3]{2})$ satisfies 
\begin{align*}
|u(x+y\sqrt[3]{2})| &< \frac 13 (X+Y\sqrt[3]{2})(1+\eta) \textup{ or}\\
|u(x+y\sqrt[3]{2})| &> 3 (X+Y\sqrt[3]{2}) 
\end{align*}and so cannot satisfy \eqref{eqn:multiplyunit} for $\eta < 1$.  

We further define $\pi(\A)$ and $\pi(\B)$ to be the number of prime ideals in $\A$ and $\B$ respectively.  Although the ideals in $\A$ have norm of somewhat different size as compared to the ideals in $\B$, they are still comparable sequences.  Our definition of $\B$ matches the definition of $\B^{(K)}$ in \cite{HB}, so this will be more convenient when referencing some preliminary results.

For convenience, we also define the sequence
\begin{equation}
\A^0 = \{a(n)\},
\end{equation}where
\begin{equation}
a(n) = \#\{x\in ]X, X(1+\eta)], y \in ]Y, Y(1+\eta)]: x, y \in \mathbb{Z}, (x, y) = 1, n = x^3 + 2y^3\}.
\end{equation}  In this paper, notation like $\{a(n)\}$ represents a sequence, despite the similarly to set notation.  Proving Theorem \ref{thm:1} is the same as proving 
$$\pi(\A^0) := \sum_p a(p) \asymp \frac{\eta^2 XY}{\log X}.
$$
Conjecturally, 
$$\pi(\A^0) \sim \sigma_0 \frac{\eta^2 XY}{3 \log X},
$$for
$$\sigma_0 = \prod_p \left(1 - \frac{\nu_p - 1}{p}\right),
$$where $\nu_p$ denotes the number of solutions of the congruence $x^3 \equiv 2 \bmod p$.

The primes in $\A^0$ correspond to prime ideals in $\A$.  To be specific, we first cite Lemma 3.1 of \cite{HB}.

\begin{lem}\label{lem:firstdegreeprimes}
	No prime ideal of degree greater than one can divide an element of $\A$, nor can a product of two distinct first degree prime ideals of the same norm. Thus if a square-free ideal $R$ divides an element of $\A$, then $N(R)$ must be square-free.	
\end{lem}

For instance, Lemma \ref{lem:firstdegreeprimes} implies that any prime ideal $(x+y\sqrt[3]{2})$ must be degree one with norm $x^3+2y^3$ being prime.  Of course, if $x^3 + 2y^3$ is prime, then the ideal $(x+y\sqrt[3]{2})$ must also be prime.  The reader should note that although $\pi(\A^0)$ appears to be counting primes with multiplicity $a(p)$, we always have $a(p) =0$ or $a(p) = 1$ by the same reasoning as the ideals in $\A$ being distinct.  Thus
\begin{equation}
\pi(\A) = \pi(\A^0).
\end{equation}
We further define
\begin{equation}
\B^0 = \{N(J)\in ]3X^3, 3X^3(1+\eta)[\},
\end{equation}  which is the integer analogue of $B$.  Then the number of primes in $\B^0$ is
\begin{equation}
\pi(\B^0) := \sum_{p \in \B^0} 1 \sim \frac{3\eta X^3}{3\log X}.
\end{equation}  Note that the primes in $\B^0$ correspond to first degree prime ideals in $\B$.  Further, the number of prime ideals in $\B$ which are not first degree is quite small since second degree prime ideals have norm $p^2 \ll X^3$ while inert prime ideals have norm $p^3 \ll X^3$ for some integer primes $p$.  Thus
\begin{equation}
\pi(\B) \sim \pi(\B^0).
\end{equation}

Let
\begin{equation}
\nu = \frac{\sigma_0 \eta Y}{3X^2},
\end{equation}so that conjecturally
$$\pi(\A) = \pi(\A^0) \sim \nu \pi(\B^0) \sim \nu \pi(\B).
$$Our main goal is to show that
$$\pi(A) - \nu \pi(\B)
$$is small.  To do this, we shall perform the same sieving procedure on $\A$ and $\B$, working over the field $K = \mathbb{Q} (\sqrt[3]{2})$.  For convenience, we fix $\C = \A$ or $\C = \B$ for the rest of this paper.

As usual, for any integral ideal $E$, we let
\begin{equation}
\C_E = \{I\in \C: E|I\},
\end{equation}and
\begin{equation}
S(\C, z) = \#\{I\in \C: P|I \Rightarrow N(P) \ge z \}.
\end{equation}

  Applying Buchstab's identity gives
\begin{align}
S(\C, 2X^{3/2}) &= S(\C, X^\delta) - \sum_{X^\delta \le N(P) < X^{1 - \tau/2}} S(\C_P, N(P)) - \sum_{X^{1 - \tau/2} \le N(P) < X^{1+\tau}} S(\C_P, N(P)) \\
&- \sum_{X^{1 + \tau} \le N(P) < X^{3/2(1-\tau)}} S(\C_P, N(P))- \sum_{X^{3/2- \tau} \le N(P) < 2X^{3/2}} S(\C_P, N(P))\\
&= S_1(\C) - S_2(\C) - S_3(\C) - S_4(\C) - S_5(\C),
\end{align}say.  We let 
\begin{equation}\label{eqn:deltadef}
\delta \asymp 1
\end{equation}be a small constant satisfying $\delta \le 1/6$ to be determined later.  The reader should think of $\delta$ as a very small fixed constant, and dependence on $\delta$ will not be explicitly stated in many estimates in the paper - in those cases, the value of $\delta$ is irrelevant to the analysis.  The exception to this is in Proposition \ref{prop:bound}, where the dependence is explicitly stated because taking $\delta$ sufficiently small there gives us desired bounds.

The parameter $\tau < 5/67$ can be taken to be any positive constant strictly greater than $\gamma$ appearing in Theorem \ref{thm:1}.  To fix ideas the reader may take
\begin{equation}\label{eqn:tau0}
\tau = \frac{\gamma + 5/67}{2}.
\end{equation}

In the above decomposition, $S_4$ is already in the right form for our Type II estimates in Proposition \ref{prop:bilinear2}, while $S_1, S_3$ and $S_5$ will be handled using standard sieve estimates.  However, $S_2$ requires further decomposition.  Specifically, let
\begin{align}
S^{(n)}(\C) &= \sum_{\substack{X^\delta \le N(P_n) < ... <N(P_1) < X^{1-\tau/2}\\ N(P_1...P_n)<X^{1+\tau}}} S(\C_{P_1...P_n}, N(P_n))\\
T^{(n)}(\C) &= \sum_{\substack{X^\delta \le N(P_n) < ... <N(P_1) < X^{1-\tau/2}\\ N(P_1...P_n)<X^{1+\tau}}} S(\C_{P_1...P_n}, X^\delta)\\
U^{(n)}(\C) &= \sum_{\substack{X^\delta \le N(P_{n+1}) < ... <N(P_1) < X^{1-\tau/2}\\ N(P_1...P_n)<X^{1+\tau} \le N(P_1...P_{n+1})}} S(\C_{P_1...P_{n+1}}, N(P_{n+1})).
\end{align}  Then we have that $S_2(\C) = S^{(1)}(\C)$ and that
$$S^{(n)}(\C) = T^{(n)}(\C) - U^{(n)}(\C) - S^{(n+1)}(\C).
$$By Lemma \ref{lem:firstdegreeprimes}, the prime ideals which appear in the decomposition above have distinct norms when $\C = \A$.  When $\C = \B$, the ideals need not have distinct norms, but of course are distinct as ideals.  We now have
\begin{equation}
S_2(\C) = \sum_{1\leq n\leq n_0} (-1)^{n+1} (T^{(n)}(\C)-  U^{(n)}(\C)),
\end{equation}where $n_0 \ll 1/\delta \ll 1$.  When $n\geq 4$, the conditions in the sum for $U^{(n)}$ imply
\begin{equation}
X^{1+\tau} \le N(P_1...P_{n+1}) \leq N(P_1...P_n)^{(n+1)/n} < X^{5(1+\tau)/4} \le X^{3/2(1 - \tau)},
\end{equation}the last inequality being equivalent to $\tau \le 1/11$, so that $U^{(n)}(\C)$ may be handled by Type II sums as in Proposition \ref{prop:bilinear2} for $n\ge 4$, while $T^{(n)}(\C)$ can be estimated asymptotically by the Fundamental Lemma.  We will need to further analyze $\U^{(n)}$ for $n=1, 2, 3$.  To do this, we let
\begin{align}
\U_1^{(1)}(\C) &= \sum_{\substack{X^\delta \le N(P_2)<N(P_1)< X^{1-\tau/2}\\ X^{1+\tau} \le N(P_1P_2) < X^{3/2 (1 - \tau)}}} S(\C_{P_1P_2}, N(P_2))\\
\U_2^{(1)}(\C) &= \sum_{\substack{X^\delta \le N(P_2)<N(P_1)< X^{1-\tau/2}\\  X^{3/2(1+ \tau)} \le N(P_1P_2)}} S(\C_{P_1P_2}, N(P_2))\\
\U_1^{(2)}(\C) &= \sum_{\substack{X^\delta \le N(P_3)<...<N(P_1)< X^{1-\tau/2}\\ N(P_1P_2) < X^{1+\tau} \le N(P_1P_2P_3) < X^{3/2(1 - \tau)}}} S(\C_{P_1P_2P_3}, N(P_3))\\
\U_2^{(2)}(\C) &= \sum_{\substack{X^\delta \le N(P_3)<...<N(P_1)< X^{1-\tau/2}\\ N(P_1P_2) < X^{1+\tau}\\ N(P_1P_2P_3) \ge X^{3/2(1 + \tau)}}} S(\C_{P_1P_2P_3}, N(P_3))\\
\U_1^{(3)}(\C) &= \sum_{\substack{X^\delta \le N(P_4)<...<N(P_1)< X^{1-\tau/2}\\ N(P_1...P_3) < X^{1+\tau} \le N(P_1...P_4) < X^{3/2 (1 - \tau)}}} S(\C_{P_1...P_4}, N(P_4))\\
\U_2^{(3)}(\C) &= \sum_{\substack{X^\delta \le N(P_4)<...<N(P_1)< X^{1-\tau/2}\\ N(P_1...P_3) < X^{1+\tau} \\ N(P_1...P_4) \ge X^{3/2 (1 + \tau)}}} S(\C_{P_1...P_4}, N(P_4))\\
S_6(\C) &= \sum_{\substack{X^\delta \le N(P_2)<N(P_1)< X^{1-\tau/2}\\ X^{3/2 (1 - \tau)} \le N(P_1P_2) < X^{3/2 (1 + \tau)}}} S(\C_{P_1P_2}, N(P_2))\\
S_7(\C)  &= \sum_{\substack{X^\delta \le N(P_3)<...<N(P_1)< X^{1-\tau/2}\\ N(P_1P_2) < X^{1+\tau} \\ X^{3/2(1 - \tau)} \le N(P_1P_2P_3) < X^{3/2(1 + \tau)}}} S(\C_{P_1P_2P_3}, N(P_3)),
\end{align}and
\begin{equation}
S_8(\C) = \sum_{\substack{X^\delta \le N(P_4)<...<N(P_1)< X^{1-\tau/2}\\ N(P_1...P_3) < X^{1+\tau}\\ X^{3/2(1-\tau)} < N(P_1...P_4)  < X^{3/2 (1 + \tau)}}} S(\C_{P_1...P_4}, N(P_4)).
\end{equation}
Then
\begin{equation}
\U^{(1)}(\C) = \U_1^{(1)}(\C) + \U_2^{(1)}(\C) + S_6(\C),
\end{equation}
\begin{equation}
\U^{(2)}(\C) = \U_1^{(2)}(\C) + \U_2^{(2)}(\C)+ S_7(\C),
\end{equation}and
\begin{equation}
\U^{(3)}(\C) = \U_1^{(3)}(\C) + \U_2^{(3)}(\C)+ S_8(\C).
\end{equation}

In the decomposition above, we will need to handle $S_6, S_7$ and $S_8$ directly either using the Prime Ideal Theorem for $\B$ or sieve bounds for $\A$, while the rest is now in an acceptable range to use bilinear sum (Type II) estimates.  Recall that we wish to show that $\pi(\A) - \nu \pi(\B)$ is small.  Thus, our main Theorem follows from the following two Propositions.

\begin{prop}\label{prop:bound}
We have that
\begin{align}
\sum_{n=1}^{n_0} |T^{(n)}(\A) - \nu T^{(n)}(\B)| &\ll \delta \frac{\eta^2XY}{\log X} \textup{  and}\\
|S_1(\A) - \nu S_1(\B)| &\ll \delta \frac{\eta^2 XY}{\log X}.
\end{align}
Moreover, 
\begin{align}
\nu S_j(\B) -S_j(\A) &\ge c_j  \sigma_0 \eta^2 \frac{XY}{\log X}
\end{align}for $j=3, 5, 7$ and
\begin{align}
S_j(\A) - \nu S_j(\B) &\ge c_j \sigma_0 \eta^2 \frac{XY}{\log X}
\end{align}for $j=6, 8$, where $c_3 = -0.187$, $c_5 = -0.172$, $c_6 = -0.088$, $c_7 = -0.124 $ and $c_8 = -0.037$.  In the above, the implied constants do not depend on $\delta$.
\end{prop}

\begin{prop}\label{prop:bilinear1}
For any constant $A>0$,
\begin{equation}
\mathfrak U(\A) - \nu \mathfrak U(\B) = o \bfrac{\eta^2 XY}{\log X},
\end{equation}for $\mathfrak U = S_4, \U_1^{(1)}, \U_2^{(1)}, \U_1^{(2)},\U_2^{(2)},$ and $ \U^{(n)}$ for all $n\geq 3.$
\end{prop}

We note $\pi(\C) \sim S(\C, 2X^{3/2})$ for $\C = \A, \B$.  By Proposition \ref{prop:bound}, there exists some constant $C_0$ such that
$$\sum_{n=1}^{n_0} |T^{(n)}(\A) - \nu T^{(n)}(\B)| + |S_1(\A) - \nu S_1(\B)| \le C_0 \delta \frac{\eta^2 XY}{\log X}.
$$

The quantities estimated in Proposition \ref{prop:bilinear1} are negligible compared to the size of $\pi(\A)$ and $\nu \pi(B)$.  Indeed, our decomposition for $S(\C, 2X^{3/2})$ and the two Propositions above tells us that
\begin{align}
\pi(\A) - \nu \pi(\B) 
&\sim \nu(S_3(\B) +S_5(\B)+ S_7(\B)) -  (S_3(\A) +S_5(\A)+ S_7(\A)) \\
&+ (S_6(\A)+S_8(\A)) - \nu(S_6(\B)+S_8(\B)) - C_0 \delta \frac{\eta^2 XY}{\log X}\\
&\geq (c_3+c_5+c_6+c_7+c_8 - \frac{C_0\delta}{\sigma_0}) \eta^2 \sigma_0 \frac{X Y}{\log X}.
\end{align}
Note that $c_3 + c_5+ c_6+ c_7+c_8 = -0.608 > -1$.  We need only choose $\delta$ small enough so that $\frac{C_0\delta}{\sigma_0} < 1-0.608$, and Theorem \ref{thm:1} follows.   In fact, letting $\delta \rightarrow 0$, we get the lower bound
$$\pi(\A)  \ge 0.392 (1-\epsilon) \eta^2 \sigma_0 \frac{X Y}{\log X},
$$for any constant $\epsilon = \epsilon(\delta)>0$ and $X$ sufficiently large in terms of $\epsilon$, from which the reader should surmise that no attempt at numerical optimization has been made.

Proposition \ref{prop:bound} uses an upper bound sieve and the Fundamental Lemma, for which we require Type I estimates while Proposition \ref{prop:bilinear1} may be reduced to Type II estimates.  The Type I estimates required are as follows.

\begin{lem} \label{lem:AtypeI}
Let $\rho_2(R)$ be the multiplicative function defined on powers of prime ideals by
\begin{equation}
\rho_2(P^e) = (1+N(P)^{-1})^{-1},
\end{equation}and extended to all integral ideals by multiplicity.  Let $\R$ be the set of ideals $R$ for which $N(R)$ is squarefree.  Then for any $A>0$, there exists a constant $c = c(A)$ such that
\begin{equation}\label{eqn:lemAtypeI}
\sum_{\substack{Q<N(R)\le 2Q\\ R \in\R}} \tau(R)^A \left|\#\A_R - \frac{6\eta^2 XY}{\pi^2 N(R)} \rho_2(R) \right| \ll (X\sqrt Y + \sqrt{XYQ} + Q) (\log QX)^{c(A)}.
\end{equation}
\end{lem}
Note that the right hand side of \eqref{eqn:lemAtypeI} is $\ll \frac{XY}{(\log X)^A}$ for any $A>0$ as long as $Q \ll (XY)^{1-\epsilon}$.  When applying sieve methods, we will find it convenient to pass to sieving over the rational integers.  The corresponding level of distribution result for $\A^0$ is below.
\begin{lem}\label{lem:A0typeI}
Let $\rho_0$ be the multiplicative function defined by
$$\rho_0(p^e) = \frac{\nu_p}{1+\frac 1p},
$$where $\nu_p$ is the number of first degree prime ideals above $p$.  Then for any $A>0$, there exists a constant $c = c(A)$ such that
\begin{equation}\label{eqn:lemAtypeI}
\sum_{Q<q\le 2Q} \mu(q)^2 \tau(q)^A \left|\#\A^0_q - \frac{6\eta^2 XY}{\pi^2 q} \rho_0(q) \right| \ll (X\sqrt Y + \sqrt{XYQ} + Q) (\log QX)^{c(A)}.
\end{equation}
\end{lem}

We also have the following level of distribution for $\B$.
\begin{lem}\label{lem:BtypeI}
For any $A>0$, there exists a constant $c(A)$ such that
\begin{equation}
\sum_{Q<N(R)\le 2Q} \tau(R)^A \left|\#\B_R - 3\gamma_0 \frac{\eta X^3}{N(R)}\right| \ll X^2 Q^{1/3} (\log Q)^{c(A)},
\end{equation}where 
\begin{equation}
\gamma_0 = \frac{\pi \log \epsilon_0}{\sqrt{27}},
\end{equation}is the residue of the Dedekind zeta function of $K$ at $1$ and where $\epsilon_0 = 1+ \sqrt[3]{2} + \sqrt[3]{4}$ is the fundamental unit of $K$.
\end{lem}
Similarly, the level of distribution result for $\B^0$ is below.

\begin{lem} \label{lem:B0type1}
Define the multiplicative function $\rho_1$ by
\begin{equation}
\rho_1(p^e) = p \left(1- \prod_{P|p} \left(1-\frac{1}{N(P)}\right)\right),
\end{equation}where $P$ runs over primes ideals in $K$.  Then for any $A>0$, there exists a constant $c(A)$ such that
\begin{equation}
\sum_{Q<q\le 2Q} \mu(q)^2\tau(q)^A \left|\#\B^0_q - 3\gamma_0 \rho_1(q) \frac{\eta X^3}{q}\right| \ll X^2 Q^{1/3} (\log Q)^{c(A)}.
\end{equation}
\end{lem}

Lemma \ref{lem:BtypeI} and Lemma \ref{lem:B0type1} is the same as Lemma 3.3 and Lemma 2.2 of \cite{HB}.  Lemmas \ref{lem:AtypeI} and \ref{lem:A0typeI} are very similar to Lemma 3.2 and 2.1 of \cite{HB}.  We shall prove Lemma \ref{lem:AtypeI} in Section \ref{sec:TypeIA}.  Passing from Lemma \ref{lem:AtypeI} to \ref{lem:A0typeI} is fairly straightforward, and we refer the reader to pg. 33 in \cite{HB} for details.

The first two bounds in Proposition \ref{prop:bound} are given by Lemma 3.5 in \cite{HB} which is proven by an application of the Fundamental Lemma; we refer the reader there for the proof.  \footnote{Note that Heath-Brown writes $T^{(0)}(\C)$ for $S_1(\C)$.}  The proof of the numerical bounds in Proposition \ref{prop:bound} is in Section \ref{sec:sievebounds} and the bulk of the paper is devoted to proving Proposition \ref{prop:bilinear1}.

\section{Type I estimate for $\A$} \label{sec:TypeIA}
Here, we prove Lemma \ref{lem:AtypeI}.  
Let
\begin{equation}
S(R; X, Y) = S(R) = \{x+y\sqrt[3]{2}: R|(x+y\sqrt[3]{2}), x \in ]X, X(1+\eta)], y\in ]Y, Y(1+\eta)]\}.
\end{equation}

To avoid excessive notation, we write $(R, x)$ to denote the greatest common divisor $(R, (x))$ for $R$ an ideal and $x$ an algebraic integer.  We begin with the following estimate.

\begin{lem}\label{lem:AtypeIlarge}
For $\R$ defined as in Lemma \ref{lem:AtypeI} and any $A>0$, there exists $c = c(A)$ such that
\begin{equation}
\sum_{\substack{Q<N(R)\le 2Q\\ R \in \R}} \tau(R)^A \left| S(R) - \frac{\eta^2 XY}{N(R)}\right| \ll (Q + X)(\log Q)^{c(A)}.
\end{equation}
\end{lem}
The proof of Lemma \ref{lem:AtypeIlarge} is entirely analogous to the proof of Lemma 5.1 of \cite{HB}.  Lemma 5.1 in \cite{HB} has the restriction $x, y\in ]X, X(1+\eta)]$ in place of our restriction $x \in ]X, X(1+\eta)], y\in ]Y, Y(1+\eta)]$, but this does not significantly affect the proof.  However, for technical convenience, Heath-Brown proved Lemma 5.1 (and we state our Lemma \ref{lem:AtypeIlarge}) for $R\in \R$ where recall that $\R$ is the set of ideals with squarefree norm.  Unfortunately, we will need a result like Lemma \ref{lem:AtypeIlarge} for $R$ a power of a prime ideal.  This is the focus of the next Lemma.

\begin{lem}\label{lem:AtypeIprimepower}
Fix $k \ge 1$ and an interval $I = [Q, Q+Q_0]$ for $Q_0 \le Q$.  Then there exists $r>0$ such that 
\begin{equation}
\sum_{\substack{N(P^k)\in I}} \left|S(P^k) - \frac{\eta^2 XY}{N(P^k)}\right| \ll   (Q + X)(\log Q)^{r},
\end{equation}
where $\sum_{\substack{N(P^k) \in I}}$ denotes a sum over degree one prime ideals $P$ with norm $N(P^k) \in I$.
\end{lem}
\begin{proof}
The proof is rather similar to the proof of Lemma 5.1 in \cite{HB}, except that some technical inconveniences were avoided in \cite{HB} due to the squarefree norm.  We have chosen to restrict our attention to powers of degree prime ideals in order to shorten the details of the proof.

As in the proof of Lemma 5.1 in \cite{HB}, we write
\begin{equation}\label{eqn:Aprimepower1}
S(P^k) -\frac{\eta^2 XY + O(X)}{N(P)^k} \ll  \sum_{\substack{|a|, |b| \le \frac{N(P^k)}{2}\\ (a, b) \neq (0, 0)}} \frac{|S_0(P^k, a, b)|}{N(P)^{2k}} \min\left\{X, \frac{N(P^k)}{|a|}\right\} \min\left\{X, \frac{N(P^k)}{|b|} \right\},
\end{equation}
where
\begin{equation}
S_0(P^k, a, b) = \sum_{\substack{u, v \bmod N(P^k) \\ P^k | u+v\sqrt[3]{2}}} e\bfrac{au+bv}{N(P)^k}.
\end{equation}
We refer the reader to (5.2) in \cite{HB} for this.  Note that $\sum_{\substack{N(P^k) \in I}} \frac{X}{N(P)^k} \ll X \log Q$, which is an acceptable error for our Lemma.  It remains to bound the contribution of the right hand side of \eqref{eqn:Aprimepower1}.  Here, we will assume that $a\neq 0$ and $b\neq 0$, the case when $a = 0$ or $b=0$ being similar but simpler (we refer the reader to pg. 30 of \cite{HB} for a treatment of this).  In doing so, we replace $\frac{1}{N(P)^{2k}}\min\left\{X, \frac{N(P^k)}{|a|}\right\} \min\left\{X, \frac{N(P^k)}{|b|} \right\}$ by $\frac{1}{|ab|}$. 

Since we only consider prime ideals $P$ of degree one, we write $N(P) = p$ for prime $p$.  We will deal with two cases.  First, let us consider the case $P| (b - a\sqrt[3]{2})$ or $K/\mathbb{Q}$ ramifies over $p$; we remind the reader that the latter occurs for only finitely many $p$.  In this case, we use the trivial bound $S_0(P^k, a, b) \ll N(P^k) = p^k \ll Q$.  Indeed, note that when $P$ is unramified, then $P^2 \nmid p$, so the condition $P^l|m$ for any integer $m$ implies that $p^l|m$.  This means that when $P^k| u+v \sqrt[3]{2}$ and $p^l|v$ for $l\le k$, then $p^l|u$ also.  Thus, since the condition $P^k| u+v \sqrt[3]{2}$ implies that $p^k | N(u+v \sqrt[3]{2})$, if $v$ is fixed, the number of choices for $u \bmod p^k$ is $O(1)$.  Thus, the contribution to the error term is bounded by
\begin{align*}
& \sum_{\substack{|a|, |b| \ll Q \\ ab \neq 0}} \frac{Q}{|ab|} \left(1+\sum_{\substack{N(P^k)\in I\\ P|a-b\sqrt[3]{2}}} 1 \right)\\
&\ll Q \sum_{\substack{|a|, |b| \ll Q \\ ab \neq 0}} \frac{1}{|ab|} \tau(a-b\sqrt[3]{2})\\
&\ll Q (\log Q)^r,
\end{align*}for some $r$.

Now we assume that $P \nmid (b-\sqrt[3]{2}a)$ which immediately implies that $p\nmid (a, b)$ and we assume that $K/\mathbb{Q}$ is unramified at $p$, so that $P^2 \nmid p$.  Then for any integer $t$ with $p \nmid t$, $tu$ and $tv$ runs over all the residues mod $p^k$ when $u$ and $v$ do, so  
\begin{align*}
S_0(P^k, a, b) = \sum_{\substack{u, v \bmod N(P^k) \\ P^k | tu+tv\sqrt[3]{2}}} e\bfrac{atu+btv}{p^k} = \sum_{\substack{u, v \bmod N(P^k) \\ P^k | u+v\sqrt[3]{2}}} e\bfrac{atu+btv}{p^k}.
\end{align*}Summing over $t \bmod p^k$ such that $p \nmid t$, we see that
\begin{align*}
(p^k - p^{k-1}) S_0(P^k, a, b) = \sum_{\substack{u, v \bmod N(P^k) \\ P^k | u+v\sqrt[3]{2}}} \sumstar_{t \bmod p^k} e\bfrac{atu+btv}{p^k},
\end{align*}where $\sumstar$ as usual denotes a sum over reduced residues.  The inner sum is
$$\sum_{t \bmod p^k} e\bfrac{atu+btv}{p^k} - \sum_{t \bmod p^{k-1}} e\bfrac{atu+btv}{p^{k-1}}.
$$It follows that
\begin{align*}
(p^k - p^{k-1}) S_0(P^k, a, b) = &p^k \#\{u, v \bmod p^k: P^k| (u+v \sqrt[3]{2}), p^k|au+bv\} \\
&- p^{k-1}\#\{u, v \bmod p^k: P^k| (u+v\sqrt[3]{2}), p^{k-1}|au+bv \}.
\end{align*}Since $p\nmid (a, b)$, at least one of $a$ or $b$ is invertible mod $p^k$, and so the condition $p^k|au+bv$ is equivalent to $u \equiv \lambda b \bmod p^k$ and $v \equiv -\lambda a \bmod p^k$ for some integer $0\le \lambda < p^k$.  Since $P^k|p^k$, the condition $P^k| (u+v \sqrt[3]{2})$ is equivalent to $P^k|(\lambda(b - a \sqrt[3]{2}))$, and since $P\nmid (b - a \sqrt[3]{2})$, $P^k | (\lambda)$, so $p^k|\lambda$, from which we conclude that $\lambda = 0$.  In this case $p^k \#\{u, v \bmod p^k: P^k| (u+v \sqrt[3]{2}), p^k|au+bv\} = p^k$.

Similarly, the conditions $p^{k-1}|au+bv$ and $P^k| (u+v\sqrt[3]{2})$ give that $u \equiv \lambda b \bmod p^{k-1}$ and $v \equiv -\lambda a \bmod p^{k-1}$ for some integer $0\le \lambda < p^{k-1}$.  Moreover $P^k|(u+v\sqrt[3]{2})$ implies that $P^k|p\lambda(b-a\sqrt[3]{2})$ and so $P^{k-1} |(\lambda)$ whence $\lambda = 0$.  Thus 
$$p^{k-1}\#\{u, v \bmod p^k: P^k| (u+v\sqrt[3]{2}), p^{k-1}|au+bv \} = p^{k-1} \#\{u, v \bmod p: P| (u+v\sqrt[3]{2})\} = p^k,$$ so that $S_0(P^k, a, b) = 0$ in this case.
\end{proof}

We now turn to the proof of Lemma \ref{lem:AtypeI}.  First, note that
\begin{align} \label{eqn:TypeI1}
\# \A_R &= \sum_d \mu(d) \# \{x\in ]X, X(1+\eta)], y\in ]Y, Y(1+\eta)]: d|(x, y), R|(x+y\sqrt[3] 2)\}\\
&= \sum_d \mu(d) \# \left \{x'\in ]X/d, X(1+\eta)/d], y'\in ]Y/d, Y(1+\eta)/d]: \frac{R}{(R, d)}|(x'+y'\sqrt[3] 2) \right \}\\
&= \sum_d \mu(d) S\left(\frac{R}{(R, d)}; \frac Xd, \frac Yd\right).
\end{align}

Recall that we want to show
\begin{equation}\label{eqn:lemAtypeIbdd}
\sum_{\substack{Q<N(R)\le 2Q\\ R \in\R}} \tau(R)^A \left|\#\A_R - \frac{6\eta^2 XY}{\pi^2 N(R)} \rho_2(R) \right|\ll (X\sqrt Y + \sqrt{XYQ} + Q) (\log QX)^{c(A)}.
\end{equation}
We will compare $\#\A_R$ with
\begin{align}\label{eqn:TypeImaint}
\sum_d \frac{ \mu(d) \eta^2 XY}{d^2} N\bfrac{R}{(R, d)}^{-1}
&= \frac{\eta^2 XY}{N(R)} \sum_d \frac{\mu(d)}{d^2} N((R, d)) \notag \\
&= \frac{\eta^2 XY}{N(R)} \prod_{p\nmid N(R)} \left(1-\frac{1}{p^2}\right) \prod_{p | N(R)} \left(1-\frac {p}{p^2}\right)\notag \\
&= \frac{\eta^2 XY}{N(R)} \frac{1}{\zeta(2)} \prod_{p | N(R)} \left(1+\frac 1p\right)^{-1}\notag \\
&= \frac{6\eta^2 XY}{\pi^2N(R)} \prod_{p|N(R)} \left(1+\frac 1p\right)^{-1}. 
\end{align}
We now split the sum over $d$ into two ranges $d<\Delta$ and $d\ge \Delta$, for some parameter $\Delta$ to be specified.  A standard calculation shows that the contribution of the large $d$ is small.  More precisely, the contribution of the terms in \eqref{eqn:lemAtypeIbdd} for which $d \ge \Delta$ arising from \eqref{eqn:TypeImaint} and summed over $R\in \R$ is bounded by
$$\sum_{\substack{Q<N(R)\le 2Q\\ R \in\R}} \tau(R)^A \frac{\eta^2 XY}{N(R)} \sum_{d\ge \Delta} \frac{|\mu(d)|}{d^2} N((R, d)) \ll \frac{\eta^2 XY (\log Q)^{c(A)}}{\Delta},
$$by a similar calculation to (5.5) of \cite{HB} and the contribution of the terms in which $d\ge \Delta$ from \eqref{eqn:TypeI1} is bounded by 
$$\sum_{\substack{Q<N(R)\le 2Q\\ R \in\R}} \tau(R)^A S\left(\frac{R}{(R, d)}; \frac Xd, \frac Yd\right) \ll \frac{\eta^2 XY (\log Q)^{c(A)}}{\Delta},
$$by a similar calculation to (5.6) of \cite{HB}.  Note that the two bounds above is indeed bounded by the right hand side of \eqref{eqn:lemAtypeIbdd} for $\Delta =  1 + \min \left(\sqrt Y, \sqrt{\frac{XY}{Q}}\right)$.

It suffices to show that
\begin{align*}
\sum_{\substack{Q<N(R)\le 2Q\\ R \in \R}} \tau(R)^A \sum_{d < \Delta} \left| S\left(\frac{R}{(R, d)}; \frac Xd, \frac Yd\right) - \frac{\eta^2XY}{d^2 N(R/(R, d))}\right| \ll (X\sqrt Y + \sqrt{XYQ} + Q) (\log QX)^{c(A)}.
\end{align*}
By Lemma \ref{lem:AtypeIlarge} and writing $I = (R, d)$ and $R = IT$, the above is 
\begin{align*}
&\ll \sum_{d < \Delta} \sum_{N(I)|d} \tau(I)^A \sum_{\substack{Q/N(I)<N(T)\le 2Q/N(I)\\ T \in \R}} \tau(T)^A \left|S(T; X/d, Y/d)  -  \frac{\eta^2XY}{d^2 N(T)}\right| \\
&\ll \sum_{d < \Delta} \sum_{N(I)|d} \tau(I)^A \left(\frac Xd + \frac{Q}{N(I)} \right) (\log Q)^{c(A)}\\
&\ll (X+Q) (\log XQ)^{c(A)} \Delta,
\end{align*}where the constant $c(A)$ is not necessarily the same from line to line.  Recalling $\Delta = 1 + \min \left(\sqrt Y, \sqrt{\frac{XY}{Q}}\right)$, we see that the above is
\begin{equation*}
(\log XQ)^{c(A)} (X+Q)\Delta
\ll (X\sqrt Y + \sqrt{XYQ} + Q) (\log QX)^{c(A)}, 
\end{equation*}as required for Lemma \ref{lem:AtypeI}.

\section{Sieve bounds} \label{sec:sievebounds}
We now prove Proposition \ref{prop:bound}.  The proof of the bounds for $T^{(n)}$ and $S_1$ is essentially the same as the proof of Lemma 3.5 in Section 6 of \cite{HB}.  One small difference to note is that our main term for $\A$ includes the factor $XY$ rather than $X^2$.  Also, our parameter $\delta \asymp \frac{1}{n_0}$ is different from that of Heath-Brown in that we choose $\delta \asymp 1$ which is larger than Heath-Brown's choice of $\tau = \frac{1}{(\log \log X)^{1/6}}$.  We do not need to modify Heath-Brown's proof however.  In particular, the sum over $n$ and other estimates introduces factors of $n_0 \asymp \frac{1}{\delta}$, but an application of the classical Fundamental Lemma of sieve theory (see e.g. Lemma 6.8 in \cite{FIO}) gives exponential savings of the form $\exp(-\frac{1}{\delta})$.  We refer the reader to \S 6 of Heath-Brown's work \cite{HB} for details.    

Similarly, for the proof of the bounds for $S_j(\C)$ for $j=5$ and $j=6$ we refer the reader to the proof of Lemma 3.6 in Section 7 of \cite{HB}.  It remains to deal with the bounds $S_j(\B)$ and $S_j(\A)$ for $j=3, 5, 6, 7,$ and $8$.  We shall evaluate $S_j(\B)$ precisely using the Prime Ideal Theorem, while $S_j(\A)$ shall be treated using an upper bound sieve.

\subsection{Computation of $S_j(\B)$}
Fundamentally, the computation of $S_j(\B)$ are standard arguments using the Prime Ideal Theorem.  However, the calculations are somewhat lengthy, and so for the sake of clarity, we first collect a few Lemmas which we will use freely without citation.  The first is the Prime Ideal Theorem.

\begin{lem}\label{lem:PIT}
There exists some constant $c>0$ such that 
\begin{equation}
\sum_{N(P) \le x} 1 = \int_2^x \frac{dt}{\log t} + O\left(x\exp(-c\sqrt{\log x})\right).
\end{equation}
\end{lem}

We will be estimating sums of ideals which are a product of a fixed number of prime ideals.  For this, we use the following variant of Lemma 4.10 of Heath-Brown in \cite{HB}.
\begin{lem}\label{lem:primesums}
	For $A\ge 2$ and $Y \ge A^n$, let $S \subset \mathbb{R}^n$ be a measurable set such that all $\bx = (x_1,...,x_n) \in S$ satisfy
	$$A\le x_i
	$$and
	$$\prod_{i=1}^n x_i \le Y.
	$$Moreover, we assume that for all $1\le j\le n$ and fixed $x_1,...,x_{j-1}, x_{j+1},...,x_n$, the set
	$$\{x_j: (x_1,...,x_j,...,x_n)\in S\}
	$$is a finite union of at most $C_0$ intervals.
	
	Further, let $f(t)$ be either $f(t) = \log t$ or $f(t) = 1$ for all $t$.  Then there exists  absolute constants $c_1$ and $c_2$ such that
	$$\sum_{(N(P_1),...,N(P_n)) \in S} \prod_{i=1}^n f(N(P_i)) = \int_S\prod_{i=1}^n \frac{f(t_i)}{\log t_i} dt_1...dt_n + O\left(nY(c_1+\log Y)^{n-1} \exp\left(-c_2 (\log A)^{1/2}\right)\right),
	$$where the implied constant $C_3$ depends only on $C_0$.  
	
\end{lem}
\begin{proof}
	We proceed by induction on $n$.  The case $n=1$ follows from the Prime Ideal Theorem.  Indeed, when $f(t) =1$, the case $n=1$ is identical to Lemma \ref{lem:PIT}.  When $f(t) = \log t$, the result follows by summation by parts.  Specifically we have that for a union of intervals $S \subset \mathbb{R} \cap [0, x]$,
	$$\sum_{N(P) \le x} f(N(P)) = \int_S \frac{f(t)dt}{\log t} + O(\mathfrak E)
	$$where the error
	$$\mathfrak E \ll \int_S \exp(-c\sqrt{\log t}) dt \ll x \exp(-c_2 (\log x)^{1/2}),
	$$for some constant $c_2 > 0$, and the implied constant $C_3$ depends only on $C_0$.
	
	Now suppose the result is true for $n-1$ for some $n\ge 2$.  Then for any $x_1$, let
	$$S'(x_1) = \{(x_2,...,x_n) \in \mathbb{R}^{n-1}: (x_1,x_2,...,x_n) \in S\}.
	$$Assuming the set above is non-empty, $A^{n-1} \le x_2...x_n \le Y/x_1$ and we may apply the induction hypothesis accordingly.  The reader should not be disturbed that $Y/x_1$ is not $Y$ - the result to be proven is not for fixed $Y$, but for all $Y \ge A^n$.  Thus, by the induction hypothesis applied to $S'(x_1)$,
	\begin{align}\label{eqn:int}
	&\sum_{(N(P_1),...,N(P_n)) \in S} \prod_{i=1}^n f(N(P_i))  \\
	&= \sum_{A \le N(P_1)\le Y} f(N(P_1))  \left(\int_{S'(N(P_1))}\prod_{i=2}^n \frac{f(t_i)}{\log t_i} dt_2...dt_n \right. \notag \\
	&\left.+ O\left((n-1)Y(c_1+\log Y)^{n-2} N(P_1)^{-1} \exp(-c_2 (\log A)^{1/2})\right) \right),
	\end{align}
	where we have used that $x_2x_3...x_n \le Y/N(P_1)$ for $(x_2,...,x_n) \in S'(N(P_1))$.
The error term from \eqref{eqn:int} above is
	\begin{align*}
	&\leq C_3 (n-1)Y(c_1+\log Y)^{n-2}  \exp(-c_2 (\log A)^{1/2}) \sum_{A \le N(P_1)\le Y}f_1(N(P_1))N(P_1)^{-1} \\
	&\leq C_3 (n-1)Y(c_1+\log Y)^{n-1}  \exp(-c_2 (\log A)^{1/2}),
	\end{align*}where $c_1$ needs to be chosen to be sufficiently large such that
	$$\sum_{A \le N(P_1)\le Y}f(N(P_1))N(P_1)^{-1} \leq c_1 + \log Y,
	$$for all $Y>1$.
	Now, let $S''$ be the projection of $S$ onto the last $n-1$ coordinates.  For fixed $(x_2,...,x_n)\in S''$, let $T = \{x_1: (x_1,x_2,...,x_n) \in S\}$, which by assumption is a finite union of intervals.  Then, applying the Prime Ideal Theorem again (or just the main result for $n=1$), the main term is
	\begin{align*}
	&\int_{S''} \int_{T} \prod_{i=1}^n \frac{f(t_i)}{\log t_i} dt_1...dt_n + O\left(  \int_{S''} \prod_{i=2}^n \frac{f(t_i)}{t_i \log t_i} Y \exp(-c_2 (\log A)^{1/2}) dt_2...dt_n \right)\\
	&= \int_{S} \prod_{i=1}^n \frac{f(t_i)}{\log t_i} dt_1...dt_n + O\left(  Y (c_1+\log Y)^{{n-1}} \exp(-c_2 (\log A)^{1/2}) \right),
	\end{align*}where the implied constant is still $C_3$ and where $c_1$ is chosen to be large enough so that
	$$\int_A^Y \frac{f(t)}{t \log t} dt \le c_1 + \log Y.
	$$
	
The sum of the two error terms is $\le C_3nY(c_1 + \log Y)^{n-1} \exp(-c_2 (\log A)^{1/2})$, as desired.
	
\end{proof}

Next we collect a commonly used calculation in the following Lemma.
\begin{lem}\label{lem:integral}
For any $n \in \mathbb{N}$, and $U, W \ge 1$ with $\log U \asymp \log W \asymp \log X$,
\begin{align}
\idotsint_{\substack{W \le t_1...t_n\le W(1+\eta)\\t_i \ge U \forall i}} \frac{dt_1...dt_n}{\log t_1...\log t_n} 
&= (n-1)! I + O_n\bfrac{\eta^2 W}{\log^2 X},
\end{align}where
\begin{align} \label{eqn:Iidentity}
I &= \idotsint_{\substack{t_i \le \bfrac{W}{t_1...t_{i-1}}^{\frac{1}{n-i+1}} \textup{ for all } 1 \le i\le n-1 \\ \frac{W}{t_1...t_{n-1}} \le t_n \le \frac{W}{t_1...t_{n-1}}(1+\eta)\\ U\le t_1 \le t_2\le...\le t_n}} \frac{dt_1...dt_n}{\log t_1...\log t_n} \notag  \\
&= \int_{U}^{W^{1/n}} \int_{t_1}^{\bfrac{W}{t_1}^{1/(n-1)}}...\int_{t_{n-2}}^{\bfrac{W}{t_1...t_{n-2}}^{1/2}}\int_{\frac{W}{t_1...t_{n-1}}}^{\frac{W}{t_1...t_{n-1}}(1+\eta)} \frac{dt_n...dt_1}{\log t_1...\log t_n} \notag \\
&= \eta W(1+o(1))\int_{U}^{W^{1/n}} \int_{t_1}^{\bfrac{W}{t_1}^{1/(n-1)}}...\int_{t_{n-2}}^{\bfrac{W}{t_1...t_{n-2}}^{1/2}}
\frac{dt_{n-1}...dt_1}{t_1...t_{n-1}\log t_1...\log t_{n-1} \log \bfrac{W}{t_1...t_{n-1}}},
\end{align}
where the $o(1)$ denotes a quantity tending to $0$ as $X\rightarrow \infty$.  In the above, expressions like $a_1...a_n$ denote the product $\prod_{i=1}^n a_i$.
\end{lem}
\begin{proof}
We first write
$$
\idotsint_{\substack{W \le t_1...t_n\le W(1+\eta)\\t_i \ge U \forall i}} \frac{dt_1...dt_n}{\log t_1...\log t_n} 
= (n-1)! \idotsint_{\substack{W \le t_1...t_n\le W(1+\eta)\\U \le t_1 \le...\le t_n }}  \frac{dt_1...dt_n}{\log t_1...\log t_n}.
$$

Since $t_i \le t_{i+1} \le...\le t_n$, $t_i^{n-i+1} \le t_i...t_n$, and $t_1...t_n \le W(1+\eta)$, so
$$t_i \le \bfrac{W(1+\eta)}{t_1...t_{i-1}}^{\frac{1}{n-i+1}}.
$$Now, we want to show that the contribution when $t_i > \bfrac{W}{t_1...t_{i-1}}^{\frac{1}{n-i+1}}$ for some $1\le i\le n-1$ is negligible.  Thus fix such an $i$ with $\bfrac{W}{t_1...t_{i-1}}^{\frac{1}{n-i+1}} <  t_i \le \bfrac{W(1+\eta)}{t_1...t_{i-1}}^{\frac{1}{n-i+1}} $.  

Since the integrand is positive, we may bound the contribution by enlarging the region of integration.  To be specific, 
\begin{align*}
&\idotsint_{\substack{W \le t_1...t_n\le W(1+\eta)\\U \le t_1 \le...\le t_n\\ \bfrac{W}{t_1...t_{i-1}}^{\frac{1}{n-i+1}} <  t_i \le \bfrac{W(1+\eta)}{t_1...t_{i-1}}^{\frac{1}{n-i+1}}    }}  \frac{dt_1...dt_n}{\log t_1...\log t_n}\\
&\le \idotsint_{\substack{U \le t_j \le W(1+\eta) \textup{ for all $j$}\\ \bfrac{W}{t_1...t_{i-1}}^{\frac{1}{n-i+1}} <  t_i \le \bfrac{W(1+\eta)}{t_1...t_{i-1}}^{\frac{1}{n-i+1}} \\ \bfrac{W}{t_1...t_{n-1}} \le t_n \le \bfrac{W(1+\eta)}{t_1...t_{n-1}} }}  \frac{dt_1...dt_n}{\log t_1...\log t_n}
\end{align*}

The contribution of the integral over $t_n$ is
$$\int_{\frac{W}{t_1...t_{n-1}}}^{\frac{W(1+\eta)}{t_1...t_{n-1}}} \frac{dt_n}{\log t_n} \ll \frac{\eta W}{t_1...t_{n-1} \log X}.
$$Let $V = \bfrac{W}{t_1...t_{i-1}}$ and note that the integral over $t_i$ is empty unless $\log V \asymp \log X$ since $t_i \ge U$ and $\log U \asymp \log X$ by assumption.  Then the integral over $t_i$ is 
$$\ll \log \frac{\log (V(1+\eta))}{\log V} \ll \log \left(1+\frac{\eta}{\log V}\right) \ll  \frac{\eta}{\log V} \ll \frac{\eta}{\log X},
$$
and for $j \neq i$, we bound the contribution of the integral over $t_j$ by
$$\int_{U \le t_j \le W(1+\eta)} \frac{dt_j}{t_j \log t_j} \ll \log \frac{\log W}{\log U} \ll 1,
$$
so that the total contribution is
$$\ll \frac{\eta^2 W}{\log^2 X},
$$as desired.  This proves the first line of \eqref{eqn:Iidentity}.

The second line of \eqref{eqn:Iidentity} follows from the observation that for $t_{n-1} \le \bfrac{W}{t_1...t_{n-2}}^{1/2}$, we  have $t_{n-1} \le \frac{W}{t_1...t_{n-2}t_{n-1}} \le t_n$, so the condition $t_{n-1} \le t_n$ is extraneous, while the third line follows from the fact that $\log t_n$ is essentially constant on the interval of integration.
\end{proof}

\subsubsection{Computation for $S_3$ and $S_5$}
Recall that
$$
S_3(\B) = \sum_{X^{1 - \tau/2} \le N(P) \le X^{1+\tau}} S(\B_P, N(P))
$$
and $\tau < 5/67$.  In the sum for $S_3(\B)$, we see that any ideal counted by the sum must have either two prime factors or three prime factors.  Thus, setting $Z = 3X^3$ for convenience, we have
\begin{align*}
S_3(\B) &= \sum_{X^{1-\tau/2}\le N(P) \le X^{1+\tau}} \left( \sum_{N(PP_1) \in (Z, Z(1+\eta)]} 1 + \sum_{\substack{N(P) \le N(P_1)< N(P_2)\\ N(PP_1P_2) \in (Z, Z(1+\eta)]}} 1 \right) \\
&= S_3^1(\B) + S_3^2(\B).
\end{align*}  Here, the contribution of the first term is
\begin{align}\label{eqn:S31B}
S_3^1(\B) 
&= \eta Z(1+o(1)) \int_{X^{1-\tau/2}}^{X^{1+\tau}} \frac{dt}{t\log t \log \bfrac{Z}{t}} \notag \\
&= \frac{\eta Z(1+o(1))}{\log X} \int_{1-\tau/2}^{1+\tau} \frac{du}{u(3-u)},
\end{align}by Lemmas \ref{lem:primesums} and \ref{lem:integral} and using the change of variables $u = \frac{\log t}{\log X}$.  

By Lemmas \ref{lem:primesums} and \ref{lem:integral}, the second term gives
\begin{align}\label{eqn:S32B}
S_3^2(\B) &= \eta Z (1+o(1)) \int_{X^{1-\tau/2}}^{Z^{1/3}} \int_{t_1 \le t_2 \le \sqrt{\frac{Z}{t_1}}} \frac{dt_2dt_1}{t_1 t_2 \log t_1 \log t_2 \log\bfrac{Z}{t_1t_2}} + O\bfrac{\eta^2 Z}{\log^2 X}\notag \\
&\sim \frac{\eta Z (1+o(1))}{\log X} \int_{1-\tau/2}^1 
\int_{v}^{\frac{3-v}{2}} \frac{dudv}{uv(3-v- u)},
\end{align}by a similar change of variables.
Recall that 
$$
S_5(\B) = \sum_{X^{3/2- \tau} \le N(P) \le 2X^{3/2}} S(\C_P, N(P)).
$$Again since $\tau < 5/67$, we see that any ideal counted by the sum must have exactly two prime factors, including $P$.  Using Lemmas \ref{lem:primesums} and \ref{lem:integral}, 
\begin{align}\label{eqn:S5B}
S_5(\B) &\sim \int_{X^{3/2-\tau}}^{2X^{3/2}} \int_{Z/t}^{X(1+\eta)t} \frac{du}{\log u} \frac{dt}{\log t} \notag \\
&\sim \eta Z \int_{X^{3/2-\tau}}^{2X^{3/2}} \frac{dt}{\log t \log \frac{Z}{t}} \notag \\
&\sim \frac{\eta Z}{\log X} \int_{3/2-\tau}^{3/2} \frac{du}{u(3-u)}.
\end{align}

\subsubsection{Computation for $S_6$}Recall that
$$S_6(\B) = \sum_{\substack{X^\delta \le N(P_2) < N(P_1) < X^{1-\tau/2}\\ X^{3/2(1-\tau)} \le N(P_1P_2) \le X^{3/2(1+\tau)} }} S(\B_{P_1P_2}, N(P_2)).
$$The conditions in the sum above imply that
$$N(P_2) \ge \frac{X^{3/2(1-\tau)}}{X^{1-\tau/2}} = X^{1/2 - \tau}.
$$If $MP_1P_2$ is counted in $S(\B_{P_1P_2}, N(P_2))$, we must have that $M$ is a product of at most three prime ideals, since otherwise $N(M) \ge N(P_2)^4 \ge X^{2 - 4\tau} > X^{3/2(1+\tau)}$ and $\tau < 5/67$.  We drop the condition $N(P_2) < N(P_1)$ resulting in a sum which is essentially $2S_6$.  The conditions $X^\delta \le t_2, t_1 < X^{1-\tau/2}$ and $X^{3/2(1-\tau)}\le t_1t_2 \le X^{3/2(1+\tau)}$ should be dissected into two regions giving 
$$2 S_6 \sim \left(\int_{X^{1/2 - \tau}}^{X^{1/2 + 2\tau}} \int_{\frac{X^{3/2(1-\tau)}}{t_2}}^{X^{1-\tau/2}} + \int_{X^{1/2 + 2\tau}}^{X^{1-\tau/2}} \int_{\frac{X^{3/2(1-\tau)}}{t_2}}^{\frac{X^{3/2(1+\tau)}}{t_2}} \right) I(t_1, t_2) \frac{dt_1dt_2}{\log t_1 \log t_2},
$$where 
$$I = I(t_1, t_2) = I_1 + I_2 + I_3,
$$for $I_j$ being the contribution of those turns with $j$ prime factors.  For convenience, write $v_i = \frac{\log t_i}{\log X}$, $t' = \min(t_1, t_2)$ and $v' = \min(v_1, v_2)$.  We have that
\begin{align}
I_1 = \int_{\frac{Z}{t_1t_2}}^{\frac{Z(1+\eta)}{t_1t_2}} \frac{du_1}{\log u_1} \sim \frac{\eta Z}{t_1t_2 \log \frac{Z}{t_1t_2}},
\end{align}
and
\begin{align}
I_2 &\sim 2 \int_{t'}^{\sqrt{\frac{Z}{t_1t_2}}} \int_{\frac{Z}{t_1t_2u_2}}^{\frac{Z(1+\eta)}{t_1t_2u_2}} \frac{du_1du_2}{\log u_1 \log u_2} \\
&\sim 2 \int_{t'}^{\sqrt{\frac{Z}{t_1t_2}}} \frac{\eta Z du_2}{t_1t_2 u_2\log u_2 \log\bfrac{Z}{t_1t_2u_2}}\\
&\sim 2\frac{\eta Z}{t_1 t_2 \log X} \int_{v'}^{\frac{3-v_1-v_2}{2}} \frac{dr}{r (3-v_1-v_2 - r)},
\end{align}and
\begin{align}
I_3 &\sim  6 \int_{t'}^{\bfrac{Z}{t_1t_2}^{1/3}} \int_{u_3}^{\bfrac{Z}{t_1t_2 u_3}^{1/2}} \int_{\frac{Z}{t_1t_2u_2u_3}}^{\frac{Z(1+\eta)}{t_1t_2u_2u_3}} \frac{du_1du_2 du_3}{\log u_1 \log u_2 \log u_3} \\
&\sim  \frac{6 \eta Z}{t_1 t_2 \log X} \int_{v'}^{\frac{3 - v_1 - v_2}{3}} \int_{r_3}^{\frac{3 - v_1 - v_2 - r_3}{2}} \frac{dr_2 dr_3}{r_2r_3 (3 - v_1 - v_2 - r_2 - r_3)}.
\end{align}
It then follows that
\begin{equation}\label{eqn:S6}
S_6 \sim \frac{\eta Z}{2 \log X} \int_{1/2 - \tau}^{1/2 + 2\tau} \int_{3/2(1-\tau) - v_2}^{1-\tau/2} + \int_{1/2 + 2\tau}^{1-\tau/2} \int_{3/2(1-\tau) - v_2}^{3/2(1+\tau)- v_2} K(v_1, v_2) \frac{dv_1dv_2}{v_1 v_2},
\end{equation}
where
\begin{align}
K(v_1, v_2) &= \frac{1}{3 - v_1 - v_2} + 2 \int_{v'}^{\frac{3-v_1-v_2}{2}} \frac{dr}{r (3-v_1-v_2 - r)} \\
&+ 6\int_{v'}^{\frac{3 - v_1 - v_2}{3}} \int_{r_2}^{\frac{3 - v_1 - v_2 - r_2}{2}} \frac{dr_1 dr_2}{r_1r_2 (3 - v_1 - v_2 - r_1 - r_2)}.
\end{align}

\subsubsection{Computation for $S_7$} 
Recall 
$$S_7 (\B) = \sum_{\substack{X^\delta \le N(P_3)<N(P_2)<N(P_1)< X^{1-\tau/2}\\ N(P_1P_2) < X^{1+\tau} \\X^{3/2(1-\tau)} < N(P_1P_2P_3) \le  X^{3/2(1 + \tau)}}} S(\B_{P_1P_2P_3}, N(P_3)).$$
In the sum for $S_7(\B)$, we see that $N(P_3) > X^{1/2 - 5\tau/2} < X^{1/3}$ which implies that any element counted by $S(\B_{P_1P_2P_3}, N(P_3))$ must be of the form $MP_1P_2P_3$ where $M$ has at most 4 prime factors.  This comes from noting that $5/3 > 3/2(1+\tau)$.  For ease of notation, let
\begin{align*}
R_7 &= \{(v_1, v_2, v_3) \in \mathbb{R}^3: 1/2 - 5\tau/2 \le v_3<v_2<v_1 < 1-\tau/2,\\ &v_1+v_2 <1+\tau, 3/2(1-\tau) < v_1+v_2+v_3 < 3/2 (1+\tau) \}.
\end{align*}

By a similar argument as for $S_6$,
\begin{align}\label{eqn:S7B}
S_7(\B) \sim \frac{\eta Z}{\log X} \int_{R_7} J(v_1 + v_2 + v_3, v_3) \frac{dv_1dv_2 dv_3}{v_1v_2v_3},
\end{align}where
\begin{align}\label{eqn:Jv}
J(v, v') &= \frac{1}{3 - v} + 2\int_{v'}^{\frac{3 - v}{2}} \frac{dr}{r(3 - v-r)} + 6\int_{v'}^{\frac{3 - v}{3}} \int_{r_2}^{\frac{3-v-r_2}{2}} \frac{dr_1dr_2}{r_1r_2(3-v-r_1-r_2)} \\
&+ 24\int_{v'}^{\frac{3 - v}{4}} \int_{r_3}^{\frac{3-v-r_3}{3}} \int_{r_2}^{\frac{3-v-r_3 - r_2}{2}} \frac{dr_1dr_2dr_3}{r_1r_2r_3(3-v-r_1-r_2-r_3)}.
\end{align}

\subsubsection{Computation for $S_8$} 
Recall that
$$S_8(\B) = \sum_{\substack{X^\delta \le N(P_4)<...<N(P_1)< X^{1-\tau/2}\\ N(P_1...P_3) < X^{1+\tau}\\ X^{3/2(1-\tau)} < N(P_1...P_4)  < X^{3/2 (1 + \tau)}}} S(\B_{P_1...P_4}, N(P_4)).
$$Here, similar to the situation with $S_7$, we also have that $N(P_4) > X^{1/2  -5\tau/2}$, and any ideal counted in the sum for $S_8$ must be of the form $MP_1...P_4$ where $M$ has at most $4$ prime factors.  We define
\begin{align}
R_8 &= \{(v_1, v_2, v_3, v_4) \in \mathbb{R}^4: 1/2 - 5\tau/2 \le v_4< v_3<v_2<v_1 < 1-\tau/2,\\ &v_1+v_2 +v_3 <1+\tau, 3/2(1-\tau) < v_1+v_2+v_3 + v_4 < 3/2 (1+\tau) \}.
\end{align}Then
\begin{align}\label{eqn:S8B}
S_8(\B) \sim \frac{\eta Z}{\log X} \int_{R_8} J(v_1+...+v_4, v_4) \frac{dv_1...dv_4}{v_1...v_4},
\end{align}where $J(v, v')$ is defined as in \eqref{eqn:Jv}.

\subsection{Bounds for $S_j(\A)$}
We shall apply lower and upper bound sieves to estimate $S_j(\A)$ for $j=3, 5, 6, 7, 8$.  Here, as in Heath-Brown's work \cite{HB}, we shall convert the sieving problem concerning ideals in $K$ into an analogous classical sieving problem over the integers.  To do this, recall that
\begin{equation}
\A^0 = \{a(n)\},
\end{equation}where
\begin{equation}
a(n) = \#\{x\in \: ]X, X(1+\eta)], y \in \:]Y, Y(1+\eta)]: x, y \in \mathbb{Z}, (x, y) = 1, n = x^3 + 2y^3\}.
\end{equation}

We note that
\begin{equation}
S(\A^0_{p_1...p_n}, z) = \sum_{N(P_i) = p_i} S(\A_{P_1...P_n}, z),
\end{equation}where $p_1,...,p_n$ denotes primes and $P_1,...,P_n$ denotes prime ideals.  This is (6.2) of \cite{HB}, the proof of which follows from Lemma \ref{lem:firstdegreeprimes}.

In what follows, we will apply the linear sieve to derive upper and lower bounds.  We refer the reader to Theorem 11.12 of Friedlander and Iwaniec's book \cite{FIO} for the details of this.  Here, let us define the usual Rosser-Iwaniec sieve coefficients $\lambda^\pm(d)$ supported on $d\le \D$, and write
\begin{equation}
\A^0_d = h(d) \mathscr A + R(d),
\end{equation}where 
\begin{equation}
h(d) = \frac{1}{d}  \prod_{p|d} \nu_p \left(1+\frac{1}{p}\right)^{-1},
\end{equation}and
\begin{equation}
\mathscr A = \frac{6\eta^2 XY}{\pi^2},
\end{equation}
and $R(d)$ is the remainder term.  Recall that
\[\lambda^{\pm}(d)=\mu(d)\mbox{ or $0$, for all }d\le \D\]
and
\[\sum_{d\mid n}\lambda^-(d)\le\sum_{d\mid n}\mu(d)\le
\sum_{d\mid n}\lambda^+(d)\]
for all positive integers $n|P(z)$ for some parameter $z$.  We first note that our multiplicative function $h(d)$ satisfies the linear sieve constraint
\begin{equation} \label{eqn:linsievecond}
\prod_{w\le p<z}\big(1-h(p)\big)^{-1}\le 
\frac{\log z}{\log w}\left(1+\frac{L}{\log w}\right)
\end{equation}
for $z\ge w\ge 2$, for some constant $L$.  Indeed, by definition of $\nu_p$,
$$\sum_{Y<p\le z} \frac{\nu_p-1}{p} = \sum_{Y<N(P)\le z} \frac{1}{N(P)} - \sum_{Y<p\le z} \frac{1}{p} + O(Y^{-1/3}),
$$where the $O(Y^{-1/3})$ accounts for the contribution of those prime ideals which are not degree one.  Further, the Prime Number Theorem and the Prime Ideal Theorem gives that the above is
\begin{equation}\label{eqn:vpbound}
\ll \frac{1}{\log^2 Y}.
\end{equation}
From the definition of $h$, we write
\begin{align}\label{eqn:hpprodcalc}
\prod_{Y<p\le z} \left(1-h(p)\right) = \prod_{Y<p\le z} \left(1-\frac{\nu_p-1}{p}\right) \left(1-\frac 1p\right) \left(1-\frac{1}{p^2} \right)^{-1},
\end{align}from which \eqref{eqn:linsievecond} follows.

Then, by Theorem 11.12 of
Friedlander and Iwaniec \cite{FI}, we have
\beql{lsub1} 
\sum_{d\mid P(z)}\lambda^+(d)h(d)\le 
\left\{F(s)+O_L\left((\log
\D)^{-1/6}\right)\right\}V(z,h)\;\;\;\;\;(s\ge 1)
\eeq
and
\[\sum_{d\mid P(z)}\lambda^-(d)h(d)\ge 
\left\{f(s)+O_L\left((\log
\D)^{-1/6}\right)\right\}V(z,h)\;\;\;\;\;(s\ge 2)\]
where $F(s)$ and $f(s)$ are the standard upper and lower bound
functions for the linear sieve, with $s=(\log \D)/(\log z)$, and
\[V(z,h)=\prod_{p<z}\left(1-h(p)\right).\]

Lemma \ref{lem:firstdegreeprimes} and Lemma \ref{lem:AtypeI} gives that for any $A>0$,
\begin{equation}\label{eqn:leveldistributionremainderbdd}
\sum_{d \le \D} \tau(d)^A |R(d)| \ll (XY)^{1-\epsilon},
\end{equation}provided that $\D \le (XY)^{1-\epsilon}$.  

We will need to apply the linear sieve to $\A_q^0$ rather than simply to $\A^0$, where $q$ is a product of at most four prime factors, each of which exceeds $X^\delta$.  To this end, we claim that
\begin{equation}
\sum_{q \le Q} \sum_{d \le \D} \mu(qd)^2 \tau(d)^A \left|\A_{qd}^0 - \frac{6\eta^2 XY}{\pi^2} h(q)h(d)\right|  \ll (XY)^{1-\epsilon},
\end{equation}as long as $\D Q < (XY)^{1-\epsilon}$.  By \eqref{eqn:leveldistributionremainderbdd}, it suffices to check that
\begin{equation}\label{eqn:lvldiff}
XY  \sum_{q \le  Q} \sum_{d \le \D} \mu(qd)^2 \tau(d)^A \left|h(q)h(d) - h(qd)\right|  =0.
\end{equation} The latter claim is obvious by multiplicativity of $h$ and the fact that $\mu(qd)^2=0$ unless $(q, d) = 1$.  In our applications of the linear sieve, we will always be examining quantities of the form $S(\A^0_q, p)$ where $p$ is the smallest prime factor of $q$; this is why the condition $(q, d) = 1$ above is acceptable.


The linear sieve as stated in Theorem 12.12 of \cite{FIO} then gives upper and lower bounds of the form
\begin{align}\label{eqn:sievebdd}
 &\sum_{q \in \Q} V(q, h) h(q) \left(f(s) + O\bfrac{1}{(\log \D)^{1/6}}\right) \mathscr A  + O(E) \notag \\
 &\le \sum_{q \in \Q} S(\A^0_q, q) \\
 &\le \sum_{q \in \Q} V(q, h) h(q) \mathscr A \left(F(s) + O\bfrac{1}{(\log \D)^{1/6}}\right) + O(E)  \notag
\end{align}where $\Q$ is a set of natural numbers larger than $X^\delta$ and  smaller than $(XY)^{1-\epsilon}$.  We remind the reader that $\delta \asymp 1$ as in \eqref{eqn:deltadef}.  For the above, we set $\D = (XY)^{1-\epsilon}/q\ge 1$ and $s = \frac{\log \D}{\log q}$.  Then we may take $E = (XY)^{1-\epsilon/2}$.  In our applications, $E$ will negligible compared to the main term.  Further, we note that by \eqref{eqn:hpprodcalc},
$$V(z, h) = \frac{\prod_{p< z} \left(1-\frac{v_p-1}{p}\right)}{\left(1-\frac{1}{p^2}\right)} \left(1-\frac{1}{p}\right)= \frac{\pi^2 e^{-\gamma} \sigma_0}{6 \log z}\left(1+O\bfrac{1}{\log z}\right),
$$where $\gamma$ is Euler's constant, by a classical estimate of Mertens.

We now write 
\begin{equation}
S_3(\A) = \sum_{X^{1-\tau/2}\le p \le X^{1+\tau}} S(\A^0_p, p).
\end{equation}Applying the bounds \eqref{eqn:sievebdd} gives
\begin{align}\label{eqn:S3A}
S_3(\A) 
&\le \sum_{X^{1-\tau/2}\le p \le X^{1+\tau}} h(p) V(p, h)F(s) \frac{6}{\pi^2} \eta^2 \sigma_0 XY (1+o(1)) + O((XY)^{1-\delta/2})\notag \\
&= \eta^2 \sigma_0 XY e^{-\gamma} (1+o(1)) \sum_{X^{1-\tau/2}\le p \le X^{1+\tau}} \frac{\nu_p}{(p+1)\log p} F(s)  + O((XY)^{1-\delta/2}) \notag \\
&= \frac{\eta^2 \sigma_0 XY e^{-\gamma}}{\log X} (1+o(1)) \int_{X^{1-\tau/2}}^{X^{1+\tau}} F\bfrac{\log XY/t}{\log t} \frac{dt}{(t+1)\log^2 t}  \notag \\
&= \frac{\eta^2 \sigma_0 XY e^{-\gamma}}{\log X} (1+o(1)) \int_{1-\tau/2}^{1+\tau} F\bfrac{2-\gamma-u}{u} \frac{du}{u^2}.
\end{align}
In the above calculation, we replaced $\nu_p$ by $1$, with negligible error since
$$\sum_{X^{1-\tau/2}\le p \le X^{1+\tau}} \frac{\nu_p-1}{(p+1) \log p} \ll \frac{1}{\log^3 X},
$$by \eqref{eqn:vpbound}, and this error may be absorbed into the $o(1)$ term.  Further we applied the Prime Number Theorem, used the change of variables $u = \frac{\log t}{\log X}$,and replaced $t+1$ by $t$ with negligible error.  Similarly
\begin{equation}\label{eqn:S5A}
S_5(\A) \le  \frac{\eta^2 \sigma_0 XY e^{-\gamma}}{\log X} (1+o(1)) \int_{3/2-\tau}^{3/2} F\bfrac{2-\gamma-u}{u} \frac{du}{u^2}.
\end{equation}
The calculations for $S_6, S_7$ and $S_8$ proceed in a similar fashion and give
\begin{align}\label{eqn:S7A}
S_6(\A) &\ge \frac{\eta^2 \sigma_0 XY e^{-\gamma}}{\log X} (1+o(1)) \int_{\substack{\frac 12 - \tau \le v_2<v_1 <1-\frac{\tau}{2}\\ \frac{3}{2}(1-\tau) \le v_1+v_2 \le \frac 32 (1+\tau)}} f\bfrac{2-\gamma-v_1-v_2}{v_2} \frac{dv_1dv_2}{v_1v_2^2}, \notag \\
S_7(\A) &\le \frac{\eta^2 \sigma_0 XY e^{-\gamma}}{\log X} (1+o(1)) \int_{(v_1, v_2, v_3) \in R_7} F\bfrac{2-\gamma-v_1-v_2-v_3}{v_3} \frac{dv_1dv_2dv_3}{v_1v_2v_3^2} \notag  \textup{ and}\\
S_8(\A) &\ge \frac{\eta^2 \sigma_0 XY e^{-\gamma}}{\log X} (1+o(1)) \int_{(v_1, v_2, v_3, v_4) \in R_8} f\bfrac{2-\gamma-v_1-v_2-v_3-v_4}{v_4} \frac{dv_1dv_2dv_3dv_4}{v_1v_2v_3v_4^2}.\\
\end{align}These also involve an application of the sieve bounds \eqref{eqn:sievebdd}, replacement of $\nu_p$ by $1$, the Prime Number theorem, and a change of variables of the form $v_i = \frac{\log t_i}{\log X}$.  It should be noted that the bounds for $S_6(\A)$ and $S_8(\A)$ are included more for the sake of completeness than utility, as the lower bounds are essentially of the form $S_6(\A) \ge 0$ and $S_8(\A) \ge 0$.

\subsection{Comparison between $S_j(\A)$ and $S_j(\B)$}
We do not attempt to precisely evaluate the integrals involved above, as that is not the focus of this work.  Instead, we will focus on gaining appropriate upper bounds.

It will be convenient for us to be able to apply sieve bounds to $\B$ as well, in order to compare our bounds for $S_j(\A)$ with our formulas for $S_j(\B)$.  To do this, recall that
\begin{equation}
\B^{0} = \{N(J) \in (3X^3, 3X^3(1+\eta))\}.
\end{equation}Now Heath-Brown's (6.4) in \cite{HB} gives
\begin{equation}\label{eqn:B0B}
S(\B^0_q, z) = \sum_{N(Q) = q} S(\B_Q, z) + O\left( \frac{\tau(q)^7}{q} X^3 z^{-1/2} (\log X)^c\right),
\end{equation}where $q = p_1...p_n$ is squarefree with each $p_i \ge z$.  We refer the reader to \cite{HB} for the standard proof of this.  We will prove the following Lemma.

\begin{lem}\label{lem:BQsquarefree}
Let $a, b>0$.  Let $S$ a finite set of ideals $Q$ satisfying $N(Q)\le X^{3-b}$ and such that for any prime ideal factor $P$ of $Q$, $P$ satisfies $N(P) \ge X^{a}$.  Then there exists some absolute constant $c>0$ such that
\begin{equation}
\sum_{Q \in S} S(\B_Q, z) = \sumone_{Q \in S} S(\B_Q, z) + O\left((X^{3-a} + X^{3-b/3}) (\log X)^c  \right)
\end{equation}where $\sumone$ denotes a sum over those $N(Q)$ which are squarefree.
\end{lem}
\begin{proof}
For fixed $q$, the number of ideals $Q$ satisfying $q = N(Q)$ is bounded by $\tau(q)^3$.  Thus applying Lemma \ref{lem:BtypeI}
\begin{align*}
\sum_{Q \in S} S(\B_Q, z) - \sumone_{Q \in S} S(\B_Q, z)
&\le \sum_{p\ge X^{a}} \sum_{\substack{q \le X^{3-b}\\p^2|q}} 
\sum_{N(Q) = q} \# \B_Q \\
&\ll X^3 \sum_{p\ge X^{a}} \sum_{\substack{q \le X^3\\p^2|q}} 
\frac{\tau(q)^3}{q} + X^{3-b/3}(\log X )^c\\
&\ll X^{3-a} (\log X )^c +X^{3-b/3}(\log X )^c,
\end{align*}for some constant $c>0$.  
\end{proof}
Now let 
$$S^0 = \{N(Q): Q \in S, N(Q)\textup{ squarefree}\},
$$and suppose that $N(Q) \le X^{3/2(1+\tau)}$ for all $Q\in S$, which will be the case when we apply this estimate for $S_7(\B)$ and $S_8(\B)$.  Then applying \eqref{eqn:B0B} and Lemma \ref{lem:BQsquarefree} with $z\ge X^\delta$ possibly depending on $Q$ and notation as in Lemma \ref{lem:BQsquarefree}, 
\begin{equation}\label{eqn:Bidealtoint}
\sum_{Q\in S} S(\B_Q, z) = \sum_{q\in S^0} S(\B^0_q, z) + O(X^{3-\delta/2}).
\end{equation}  Here we recall that $\delta \asymp 1$ so that $X^{3-\delta} (\log X)^c\ll X^{3-\delta/2}$.

Let 
$$R_j' = \{(t_1, t_2,...,t_{j-4}) \in \mathbb{R}^{j-4}: (\log t_1/\log X,...,\log t_{j-4}/\log X) \in R_j\}
$$for $j=7, 8$.
This allows us to write
\begin{align*}
S_7(\B) &\sim \sum_{(N(P_1), N(P_2), N(P_3)) \in R_7'}S(\B_{P_1P_2P_3}, N(P_3))\\
&\sim \sum_{(p_1,p_2,p_3)\in R_7'} S(\B^0_{p_1p_2p_3}, p_3),
\end{align*}by \eqref{eqn:Bidealtoint} and recalling that $N(P_1P_2P_3) \le X^{3/2(1+\tau)} = X^{3-3/2(1-\tau)}$ in the sum above from the definition of $R_7$.  Note that our previous arguments show that
$$S(\B^0_{p_1p_2p_3}, p_3) \sim \frac{\eta Z}{\log X}\frac{J(v_1+v_2+v_3, v_3)}{p_1p_2p_3},
$$where $v_i = \log p_i/\log X$.  Now, we would like to apply the linear sieve \eqref{eqn:sievebdd} to $S(\B^0_{p_1p_2p_3}, p_3)$ for comparison purposes - this is not necessary but will make the subsequent numerical computations easier.  We first need to verify that
\begin{equation}\label{eqn:linsievecondB}
\prod_{w\le p<z}\big(1-\frac{\rho_1(p)}{p}\big)^{-1}\le 
\frac{\log z}{\log w}\left(1+\frac{L}{\log w}\right),
\end{equation}where recall $\rho_1$ is defined as in Lemma \ref{lem:B0type1}.  In order to do this, we write
\begin{align*}
\prod_{w\le p<z}\left(1-\frac{\rho_1(p)}{p}\right)
&= \prod_{w\le p<z}\prod_{P|p} \left(1-\frac{1}{N(P)}\right)\\
&= \prod_{w\le p<z} \left(1-\frac 1p\right)  \prod_{w\le p<z} \prod_{P|p} \left(1-\frac{1}{N(P)}\right) \left(1-\frac 1p\right)^{-1}. 
\end{align*}
On the other hand, for all $\sigma \ge 1$, 
\begin{align}\label{eqn:rho1est}
\sum_{w\le p < z} \frac{1}{p^\sigma} - \sum_{w\le p < z} \sum_{P|p} \frac{1}{N(P)^\sigma} 
&= \sum_{w\le p < z} \frac{1}{p^\sigma} - \sum_{w\le N(P) < z} \frac{1}{N(P)^\sigma} + O(w^{-1/2})\\
&\ll \frac{1}{\log^2 w},
\end{align}by the Prime Number Theorem and the Prime Ideal Theorem.  The condition \eqref{eqn:linsievecondB} follows directly from here.

Further, we note that for all $\sigma > 1$,
\begin{align}
\prod_{p<z} \prod_{P|p} \left(1-\frac{1}{N(P)^\sigma}\right)
&= \prod_{p<z} \left(1-\frac 1p\right)  \prod_{p<z} \prod_{P|p} \left(1-\frac{1}{N(P)^\sigma}\right)\left(1-\frac 1p\right)^{-1}\\
&= \frac{\zeta(\sigma)}{\zeta_K(\sigma)}\prod_{p<z} \left(1-\frac 1p\right)  \prod_{p>z} \prod_{P|p} \left(1-\frac{1}{N(P)^\sigma}\right)\left(1-\frac 1p\right)^{-1}\\
&= \frac{\zeta(\sigma)}{\zeta_K(\sigma)} \prod_{p<z} \left(1-\frac 1p\right)  \left(1+ O\bfrac{1}{\log^2 z} \right),
\end{align}by \eqref{eqn:rho1est} .  Letting $\sigma \rightarrow 1$ gives that
\begin{equation}
\prod_{p<z} \left(1 - \frac{\rho_1(p)}{p} \right) \sim \frac{e^{-\gamma}}{\gamma_0 \log z} 
\end{equation}where recall that $\gamma_0$ is the residue of $\zeta_K(s)$ at $s=1$ and we have used Mertens' classical estimate again.  Finally, we note that the appropriate level of distribution result holds for $\B^0_q$ in our applications as in Lemma \ref{lem:B0type1}; this may be easily verified in the same way as for $\A^0_q$.

Now applying the linear sieve to $S(\B^0_{p_1p_2p_3}, p_3)$ and writing $v_i = \frac{\log p_i}{\log X}$ gives that
\begin{align*}
S(\B^0_{p_1p_2p_3}, p_3) 
&\le \frac{e^{-\gamma }\eta Z}{p_1p_2p_3 \log p_3}F\bfrac{3-v_1-v_2-v_3}{v_3} (1+o(1)) \\
&\le \frac{e^{-\gamma }\eta Z}{p_1p_2p_3 \log p_3}F\bfrac{2-\gamma - v_1-v_2-v_3}{v_3}(1+o(1)),
\end{align*}from which it follows that
\begin{equation}
\frac{\eta Z}{\log X} \frac{J(v_1+v_2+v_3, v_3)}{p_1p_2p_3}
\le \frac{e^{-\gamma }\eta Z}{p_1p_2p_3 \log p_3}F\bfrac{2-\gamma - v_1-v_2-v_3}{v_3}(1+o(1)),
\end{equation}and upon substituting $\log p_3 = v_3 \log X$, we get
\begin{equation}\label{eqn:sievecompare}
J(v_1+v_2+v_3, v_3)\le \frac{e^{-\gamma }}{v_3}F\bfrac{2-\gamma - v_1-v_2-v_3}{v_3}(1+o(1)).
\end{equation}
Thus
\begin{align*}
\nu S_7(\B) - S_7(\A) &\ge \nu \frac{\eta Z}{\log X} \int_{R_7} J(v_1 + v_2 + v_3, v_3) \frac{dv_1dv_2 dv_3}{v_1v_2v_3}\\
&-  \frac{\eta^2 \sigma_0 XY e^{-\gamma}}{\log X} (1+o(1)) \int_{(v_1, v_2, v_3) \in R_7} F\bfrac{2-\gamma-v_1-v_2-v_3}{v_3} \frac{dv_1dv_2dv_3}{v_1v_2v_3^2}\\
&= \frac{\sigma_0 \eta^2 XY (1+o(1))}{\log X} \int_{R_7} \frac{J(v_1+v_2+v_3, v_3)}{v_1v_2v_3} \\
&- e^{-\gamma} F\bfrac{2-\gamma-v_1-v_2-v_3}{v_3} (v_1v_2v_3^2)^{-1} dv_1dv_2dv_3
\end{align*}
By \eqref{eqn:sievecompare}, the integrand above is nonpositive (or more precisely, the integrand is $\le o(1)$).  Let $\tilde{R_7} \subset \mathbb{R}^3$ be any set satisfying $R_7 \subset \tilde{R_7}$.  Then the above tells us that
\begin{align}
\nu S_7(\B) - S_7(\A) 
&\ge \frac{\sigma_0 \eta^2 XY (1+o(1))}{\log X} \int_{\tilde{R_7}} \frac{J(v_1+v_2+v_3, v_3)}{v_1v_2v_3} \\
&- e^{-\gamma} F\bfrac{2-\gamma-v_1-v_2-v_3}{v_3} (v_1v_2v_3^2)^{-1} dv_1dv_2dv_3.
\end{align}
By the same process, 
\begin{align}
S_8(\A) - \nu S_8(\B) 
&\ge \frac{\sigma_0 \eta^2 XY (1+o(1))}{\log X} \int_{\tilde{R_8}} e^{-\gamma} f\bfrac{2-\gamma-v_1-v_2-v_3-v_4}{v_4} (v_1v_2v_3v_4^2)^{-1}\\
&- \frac{J(v_1+v_2+v_3+v_4, v_4)}{v_1v_2v_3v_4}  dv_1dv_2dv_3dv_4
\end{align}for any set $\tilde {R_8}$ such that $R_8 \subset \tilde{R_8} \subset \mathbb{R}^4$.  In the above, we similarly see that the integrand is $\leq o(1)$ by applying the lower bound arising from the linear sieve.  To be precise,  
we have that
\begin{align*}
&\frac{e^{-\gamma}\eta Z}{p_1...p_4\log p_4} f\bfrac{2-\gamma-v_1-v_2-v_3-v_4}{v_4} \\
&\leq S(\B^0_{p_1p_2p_3p_4}, p_4) (1+o(1)) = \frac{\eta Z}{\log X} \frac{J(v_1+...+v_4, v_4)}{p_1...p_4},
\end{align*}again with the identification $v_i = \frac{\log p_i}{\log X}$.  

\subsection{Numerical estimates}
For convenience, we let
\begin{equation}
\mathcal X = \frac{\sigma_0 \eta^2 XY(1+o(1))}{\log X}.
\end{equation}
In numerically bounding various integrals, we may substitute $\frac{5}{67}$ for $\tau$ since $\tau < \frac{5}{67}$.  Then, computing the integrals in \eqref{eqn:S31B}, \eqref{eqn:S32B} and \eqref{eqn:S3A} numerically in Maple, we get 
\begin{align*}
\nu S_3(\B) - S_3(\A) \ge -0.187 \mathcal{X}
\end{align*}

Similarly, computing the integrals in \eqref{eqn:S5B}, \eqref{eqn:S5A}, we have
\begin{equation}
\nu S_5(\B) - S_5(\A) \ge -0.172 \mathcal{X}.
\end{equation}

For $S_6$, we use the trivial bound $S_6(\A) \ge 0$ and get
\begin{align*}
S_6(\A)  - \nu S_6(\B) \ge -\nu S_6(\B).
\end{align*}Moreover, we use \eqref{eqn:S6}.  The first integral in \eqref{eqn:S6} may be evaluated precisely.  For the second integral, we use that
\begin{align*}
K(v_1, v_2) &\le  \frac{1}{3 - v_1 - v_2} + 2 \int_{v'}^{\frac{3-v_1-v_2}{2}} \frac{dr}{r (3-v_1-v_2 - r)} \\
&+ \frac{18}{3-v_1-v_2} \int_{v'}^{\frac{3 - v_1 - v_2}{3}} \log \bfrac{3-v_1-v_2-r_2}{2r_2} \frac{dr_2}{r_2}.
\end{align*}Typing the integrals into Maple again results in
\begin{align*}
S_6(\A)  - \nu S_6(\B) \ge -0.088 \mathcal X.
\end{align*}

For $S_7(\B)$, we use \eqref{eqn:S7B} and the lower bound
\begin{align*}
J(v_1+v_2+v_3, v3) &\ge \frac{1}{3-v_1-v_2-v_3} + \frac{2}{3-v_1-v_2-2v_3} \log \bfrac{3-v_1-v_2-v_3}{2v_3} \\
&+ \frac{6}{3-v_1-v_2-3v_3} \left(-\log v_2 - 2\log (3-v_1-v_2-2v_3) \right.\\
&\left. - 3\log 3 + 3\log (3-v_1-v_2-v_3) + 2\log 2\right).
\end{align*}This is derived by writing 
$$\int_{v_3}^{\frac{3-v}{2}} \frac{dr}{r (3-v-r)} \ge \int_{v_3}^{\frac{3-v}{2}} \frac{dr}{r (3-v-v_3)}, 
$$and
$$\int_{v_3}^{\frac{3-v}{3}} \int_{r_2}^{\frac{3-v-r_2}{2}} \frac{dr_1dr_2}{r_1r_2 (3-v-r_1-r_2)} \ge \int_{v_3}^{\frac{3-v}{3}} \int_{r_2}^{\frac{3-v-r_2}{2}} \bfrac{2}{3-v-r_2} \frac{dr_1dr_2}{r_2 (3-v-2v_3)}.
$$Using this, and the bounds $1/2 - \frac{5\tau }{2}\le v_3\le \frac {1}{2} (1+\tau), v_3\le v_2 \le \frac{1}{2}(1+\tau), \frac{3}{2} (1-\tau) - v_2-v_3\le v_1 \le \frac{3}{2} (1+\tau) - v_2-v_3$ for $R_7$, valid since replacing $R_7$ by a region $R_7 \subset \tilde R_7$ suffices, we type the integrals resulting from here and \eqref{eqn:S7A} into Maple to get that
$$\nu S_7(\B)- S_7(\A) \ge (0.114 - 0.238) \mathcal{X}  = -0.124 \mathcal{X}.
$$
In the case of $S_8$, we again use the trivial bound $S_8(\A) \ge 0$, and write
$$S_8(\A)  - \nu S_8(\B) \ge - \nu S_8(\B).
$$ In deriving an upper bound for $S_8(\B)$ we use \eqref{eqn:S8B} and replace $R_8$ by the larger set $\tilde R_8$ defined by
\begin{align}
\tilde{R}_8 &:= \{(v_1, v_2, v_3, v_4)\in \mathbb{R}^4: 1/2 - \frac{5\tau}{2} \le v_3, v_4 < \frac{1+\tau}{3}, \\
&1/2 - \frac{5 \tau}{2} \le v_2 \le \frac 14 +\frac{7\tau}{4}, 3/2(1-\tau) - (1+\tau) \le v_1 \le  3/2(1+\tau) - 3(1/2 - \frac{5\tau}{4}) \}.
\end{align}
Further we use that
\begin{align}
J(v, v') &\le \frac{1}{3-v} + \frac{3}{v'(3-v-v')}\frac{3-v-2v'}{2} + \frac{(3-v-3v') (3-v-3v')}{v'^2 (3-v-2v')}\\
&+ 2 \frac{(3-5v')(2-v-4v')}{v'^3(3-v-3v')}.
\end{align}This is derived by replacing each integrand by its absolute maximum over the interval of integration.  Maple calculations using these bounds then give us that
\begin{equation}
S_8(\A)  - \nu S_8(\B) \ge -\nu S_8(\B) \ge -0.037 \mathcal X.
\end{equation}

\section{Proposition 2: initial treatment of bilinear sums}

We first reduce the proof of Proposition 2 to a bound on certain bilinear sums.  This reduction follows many of the same steps as in Section 3 of \cite{HB}, and we provide the framework here for the reader's convenience.  

First, we fix some notation.  We set
\begin{equation}\label{eqn:xi}
\xi = \frac{1}{\log \log X},
\end{equation} and recall 
\begin{equation}\label{eqn:delta}
\delta \asymp 1
\end{equation}is a fixed constant.
Further, let
\begin{equation}\label{eqn:L}
L = X^{\xi}.
\end{equation}

For an integer $n\geq 0$, fix ${\bf m} = (m_1,...,m_{n+1}) \in \mathbb{N}^{n+1}$, and intervals $J(m_i) = [X^{m_i \xi}, X^{(m_i+1)\xi}[$ satisfying
\begin{equation}\label{eqn:mcond}
m_1>m_2>...m_{n+1} \geq \frac{\delta}{\xi},
\end{equation}so that the intervals $J(m_i)$ are disjoint.  The first goal is to express the sums $S_4(\C), \U_1^{(1)}(\C), \U_1^{(2)}(\C), \U_1^{(3)}(\C), \U_2^{(1)}(\C), \U_2^{(2)}(\C), \U_2^{(3)}(\C), U^{(n)}(\C)$ with $n\ge 4$ as a combination of sums of the form
\begin{equation}\label{eqn:bilinearform0}
\sum_{R} c_R \sum_{S: RS \in \C} d_S,
\end{equation}
where $C_R$ is either $0$ or $1$ and is supported on ideals $R\in \R$ such that any prime factor $P|R$ satisfies $N(P) \ge X^\delta$.  Further $d_S$ is supported on the range $X^{1+\tau} \le N(S) \le X^{3/2 (1- \tau)}$ and on those $S$ of the form
$$S = \prod_{i=1}^{n+1} P_i$$ 
where $N(P_i) \in J(m_i)$ are distinct prime ideals.  When $S$ is of this form, we define
\begin{equation}\label{eqn:dS}
d_S = \prod_{i=1}^{n+1} \frac{\log (N(P_i))}{m_i \xi \log X},
\end{equation}  and set $d_S = 0$ otherwise.

Here, we simply note what our choice of $S$ is in each case, so that it will be clear $d_S$ is supported on the range $X^{1+\tau} \le N(S) \le X^{3/2 (1- \tau)}$.  In the definition of $S_4$, $\U_1^{(1)}$, $\U_1^{(2)}$ and $\U_1^{(3)}$, $S = P$, $S = P_1P_2$, $S = P_1P_2P_3$ and $S = P_1...P_4$ respectively.  For $U^{(n)}$ with $n\ge 4$, we take $S = P_1...P_{n+1}$ which we have already noted has norm in the above range for $n \geq 4$.  In the case of $\U_2^{(1)}$, we may write an element of $\C_{P_1P_2}$ as $P_1P_2Q$ and the conditions on $P_1$ and $P_2$ imply that $X^{1+\tau} \leq N(Q) \leq X^{3/2 (1-\tau)}$ so we may take $S = Q$ in this case.  A similar argument holds for $\U_2^{(2)}$ and $\U_2^{(3)}$.

To see that each of the quantities in Proposition \ref{prop:bilinear1} may be reduced to studying combinations of sums of the form
\begin{equation}\label{eqn:tildeUC}
U(\C) := \sum_R c_R \sum_{S: RS \in \C} d_S
\end{equation}for $\C = \A$ or $\C = \B$  as claimed, we refer the reader to \S 3 of \cite{HB} (specifically pages 14 to 17).  Our situation is completely covered by his treatment.  Our definition of $\delta$ being a fixed constant should be considered as making the situation simpler.  The analogous quantity in Heath-Brown's work \cite{HB} is $\tau$ and is of size $\frac{1}{(\log \log X)^{1/6}}$.  Lemma 3.7 of \cite{HB} immediately implies that the total error incurred by this replacement is 
\begin{align*}
\ll 
\begin{cases}
\frac{1}{\log \log X} \frac{\eta^2 XY}{\log X} \textup{ for } \C = \A \\
\frac{1}{\log \log X} \frac{\eta X^3}{\log X} \textup{ for } \C = \B,
\end{cases}
\end{align*}and these errors suffice for Proposition \ref{prop:bilinear1}.

For completeness, we note that in the definition of $d_S$ above, $n\ll \frac{1}{\delta} \asymp 1$, while $m_i \ll \frac{1}{\xi}$, so that the number of sums of the form \eqref{eqn:tildeUC} needed is bounded by $(\log \log X)^c$ for some constant $c>0$.  In the sequel, we will be performing replacements that will produce error terms which save a large power of $\log X$ over the main term and this will be sufficient for Proposition \ref{prop:bilinear1}.

\subsection{Replacement of $d_S$ by $h_S$}
For $\zeta_K(s)$ the Dedekind zeta function of $K = \mathbb{Q}(\sqrt[3]{2})$, we define $\Lambda_K(T)$ for integral ideals $T$ by 
$$-\frac{\zeta_K'}{\zeta_K}(s) = \sum_{T} \frac{\Lambda_K(T)}{N(T)^s}.
$$ 

Fix $n$ and $m_1,m_2,...,m_n$ as in the definition of $d_S$.  We first show that we may replace $d_S$ with
\begin{equation}\label{eqn:hbeta}
h_S = \prod_{i=1}^{n+1} \frac{\Lambda_K(T_i) \frak{W}_i(N(T_i))}{m_i \xi \log X},
\end{equation}where $S = T_1...T_{n+1}$ for integral ideals $T_i$, $0\le \frak{W}_i(x)\le 1$ is a smooth function which is $1$ on $J(m_i) = [X^{m_i \xi}, X^{(m_i+1)\xi}]$, is supported on $[X^{m_i \xi}(1-\mathcal{\iota}), X^{(m_i+1)\xi}(1+\mathfrak \iota)]$, where

\begin{equation}\label{eqn:iota}
\iota = \exp(- (\log X)^\epsilon).
\end{equation}  The only two important properties of $\iota$ for us are that for any $A>0$,
\begin{equation}\label{eqn:iotaproperty}
exp\left(-\frac{1}{A} (\log X)^{1/3 - \epsilon}\right) \ll  \iota \ll \frac{1}{\log^A X}.
\end{equation}  The second inequality will be useful shortly, while the first is useful in \S \ref{sec:errorsmallcube}.  Note that there is no explicit support condition on $h_S$, but our definition of $h_S$ implies that $h_S = 0$ unless $S = T_1...T_{n+1}$ for integral ideals $T_i$ satisfying $N(T_i) \in [X^{m_i \xi}(1-\mathcal{\iota}), X^{(m_i+1)\xi}(1+\mathfrak \iota)]$, by the support condition on $\frak{W}_i(x)$.  Moreover, the coefficient $\Lambda_K(T_i)$ is supported on $T_i$ a power of a prime ideal.

We further demand that
\begin{equation}\label{eqn:frakWbdd}
\frak{W}_i^{(k)}(x) \ll_k (\iota x)^{-k},
\end{equation}for all $k \ge 0$.  

Although we do not explicitly demand that the ideals $T_1,...,T_{n+1}$ are distinct, the intervals $J(m_i)$ are disjoint so the ideals $T_1,...,T_{n+1}$ are distinct in what follows.  The form $h_S$ is more convenient for our purposes later in the paper.  We perform the replacement now because it is easier to do so here. 

We record the trivial bounds below which we will use without explanation later.

\begin{lem}\label{lem:trivialbddcoef}
With notation as above, we have that
\begin{align*}
n &\ll 1 \\
d_S &\ll 1\\
h_S &\ll 1.
\end{align*}
\end{lem}
\begin{proof}
We have that $n\ll \frac{1}{\delta} \asymp 1$ by \eqref{eqn:deltadef}.  Further, by the definition of $d_S$,
$$d_S \ll \prod_{i=1}^{n+1} \frac{m_i+1}{m_i} \le 2^{n+1} \ll 1,
$$and similarly for $h_S$.
\end{proof}

We now prove the next Lemma.

\begin{lem}\label{lem:hbetadbetareplace}
With $d_S$ and $h_S$ as above we have that for any $A>0$,

\begin{equation}\label{eqn:dShSA}
\sum_R c_R \sum_{S: RS \in \A} (d_S - h_S) \ll_A \frac{XY}{\log^A X}
\end{equation}and
\begin{equation}\label{eqn:dShSB}
\sum_R c_R \sum_{S: RS \in \B} (d_S - h_S) \ll_A \frac{X^3}{\log^A X}
\end{equation}

\end{lem}
\begin{proof}
We concentrate on proving \eqref{eqn:dShSA}, the proof for \eqref{eqn:dShSB} begin similar but simpler.  For clarity, we proceed in two steps.  

We define
\begin{equation}\label{eqn:hbetaprime}
h_S' = \prod_{i=1}^{n+1}\frac{\Lambda_K(T_i)}{m_i \xi \log X},
\end{equation}supported on $S = T_1...T_{n+1}$ where $T_i$ are integral ideals with $N(T_i) \in J(m_i)$.  We now show that
\begin{align}\label{eqn:dbhberror1}
\sum_R c_R \sum_{S: RS \in \A} |d_S - h_S'| \ll  XY (\log X)^{-A},
\end{align}for any $A$. 

Indeed, by definition, $\Lambda_K(T_i) = 0$ unless $T_i$ is a power of a prime ideal and moreover
$$\prod_{i=1}^{n+1} \frac{\log (N(P_i))}{m_i \xi \log X} = \prod_{i=1}^{n+1} \frac{\Lambda_K(T_i)}{m_i \xi \log X} 
$$if $P_i = T_i$ for all $i$.  Hence, we need only bound the contribution of those terms $\beta = T_1...T_{n+1}$ where at least one of $T_i = P^r$ for some prime ideal $P$ and $r\ge 2$.  The number of choices for $i$ is $n+1$.  Noting that $|d_{\beta} - h_\beta '|\ll n \ll 1$, we see that the quantity in \eqref{eqn:dbhberror1} is bounded by
\begin{align*}
\sum_{1<r\ll \log X} \sum_{N(P)^r \in J(m)} \sum_{\substack{RS \in \A \\ P^r|S}} 1 \ll \sum_{1<r\ll \log X} \sum_{N(P)^r \in J(m)} \frac{XY}{N(P^r)} + X^{3/2(1-\delta)} \log^c X,
\end{align*}for some $c>0$ where $J(m) = [X^{m \xi}, X^{(m+1)\xi}[$ for some $m \in \mathbb{N}$ and we have used Lemma \ref{lem:AtypeIprimepower} with the bound $N(P^r) \le N(S) \le X^{3/2(1-\tau)}$.  By similar arguments as before $n X^{3/2(1-\delta)} \log^r X \ll_A XY (\log X)^{-A}$ for any $A>0$.  To bound the first term, we cover $J(m)$ by $\ll \log X$ intervals of the form $\mathcal I = [N_0, 2N_0]$ for $N_0 \gg X^\delta$, and so this contribution is bounded by

\begin{align*}
\log X \sum_{1<r\ll \log X} \sum_{N(P)^r \in I} \frac{XY}{N(P^r)}
\ll \log X \sum_{1<r\ll \log X} N_0^{1/r} \frac{XY}{N_0}
\ll \log^2 X \frac{XY}{X^{\frac{\delta}{2}}},
\end{align*}upon noting that $N_0^{1/r - 1} \leq N_0^{-1/2} \ll X^{-\delta/2}$.  The last line is $\ll_A XY(\log X)^{-A}$ for any $A>0$ by \eqref{eqn:xi} and $n \ll \frac{1}{\xi}$ as before.  


Now it remains to show that 

\begin{align}\label{eqn:dbhberror2}
\sum_R c_R \sum_{S: RS \in \A} |h_S' - h_S| \ll  XY (\log X)^{-A},
\end{align}for any $A$. 
We note that by definition $\left|h_S' -h_S \right| \neq 0$ only when there is some ideal $T_i = P_i^k$ for some prime ideal $P_i$ with $k\ge 1$ such that $T_i|S$ and 
$$N(T_i) \in [X^{m_i \xi}(1-\mathcal{\iota}),X^{m_i \xi}[ \;\; \bigcup \;\; ]X^{(m_i+1)\xi}, X^{(m_i+1)\xi}(1+\mathfrak \iota)],$$ 
for some $1\le i\le n+1$.  For less cumbersome notation, we fix $i$, write $T = T_i = P^k$, and without loss of generality assume $N(T) \in \;]X^{(m_i+1)\xi}, X^{(m_i+1)\xi}(1+\mathfrak \iota)]$.

There are $n+1 \asymp 1$ choices for $i$, and so the left side of \eqref{eqn:dbhberror2} is 

\begin{align*}
&\ll \sum_{1\le k\ll \log X} \sum_{N(P^k) \in ]X^{(m_i+1)\xi}, X^{(m_i+1)\xi}(1+\mathfrak \iota)]} \;\sum_{\substack{P^k|S \\ RS \in \A}} 1 \\
&\ll \log X \sum_{N(T) \in ]X^{(m_i+1)\xi}, X^{(m_i+1)\xi}(1+\mathfrak \iota)]}  \frac{XY}{N(T)} +X^{3/2(1-\delta)} \log^c X,
\end{align*} for some $c>0$ by Lemma \ref{lem:AtypeIprimepower} again.

The final quantity above is bounded by
$$\iota XY \log X  \ll \frac{XY}{\log^A X},
$$for any $A>0$ by \eqref{eqn:iotaproperty}.  

The Lemma follows from \eqref{eqn:dbhberror1} and \eqref{eqn:dbhberror2}.
\end{proof}

Thus, it suffices to show that 
$$\sum_R c_R \sum_{S: RS \in \A} h_S  - \nu  \sum_R c_R \sum_{S: RS \in \B} h_S \ll \frac{XY}{\log^A X}$$ for any $A>0$.
In order to avoid dealing with bilinear sums involving both sequences at the same time, we extract a main term from $h_S$.  To be precise, let 
\begin{equation}
e_S = \frac{w'(N(S))}{\prod_{i=1}^{n+1} (m_i \xi \log X)} \sum_{J|S: N(J) < L} \mu(J)\log \frac{L}{N(J)},
\end{equation}where recall $L = X^{\delta/2}$ and 
$$w(t) = \int_{\substack{x\in \mathbb{R}^{n+1}\\  \prod x_i \le t\\ x_i \in J(m_i) }} 1 \; dx_1...dx_{n+1}.
$$  We set $f_S = h_S - e_S$.  Here, $e_S$ is constructed to behave in the same way as $d_S$ and $h_S$ in arithmetic progressions.  To be precise, we have Lemma 3.8 from \cite{HB} below.

\begin{lem}\label{lem:fbetaAP}
Let $V \ge 1$.  Let $C \subset \mathbb R^3$ be a cube of side $S_0 \geq L^2$ and edges parallel to the coordinate axes.  Suppose that for every vector $(x, y, z)\in C$ we have $x, y, z \ll V^{1/3}$ and 
$$x^3 + 2y^3 + 4z^3 - 6xyz \gg V.
$$For each $\beta = a+ b\sqrt[3]{2}+c\sqrt[3]{4} \in K$, let $\hat \beta = (a, b, c)$.  Then for any constant $A_0>0$ and any integer $\alpha \in \mathbb{Z}[\sqrt[3]{2}]$ we have that there exists a constant $c_0>0$ such that
\begin{equation}
\sum_{\substack{\beta \equiv \alpha \bmod q\\ \hat \beta \in C}} f_{(\beta)} \ll V\exp\left(-c_0 \sqrt{\log L}\right), 
\end{equation}uniformly for $q \leq (\log X)^{A_0}$.
\end{lem}

Actually Heath-Brown proved Lemma \ref{lem:fbetaAP} for $f_S = d_S - e_S$, but the proof for our case follows the same way.  Note that Lemma \ref{lem:fbetaAP} is only meaningful when the side of the cube $S_0$ is close to $V^{1/3}$.  Later on, we shall prove a similar result for much smaller cubes in Lemma \ref{lem:SWbdd}.

Writing $h_S = e_S + f_S$, it is relatively straightforward to handle the contribution of $e_S$.  The reader should note that $e_S$ involves a very short sum of Mobius functions, and this quantity may be understood by standard methods.  Specifically, by a small modification of Lemma 3.9 in \cite{HB} (again, the modification is to replace $X^2$ by $XY$ in the upper bounds in Lemma 3.9 and to verify that our values of $\delta$ and $L$ do not cause any issues), the contribution of $e_S$ to the quantity $U(\A) - \nu U(\B)$ is
$$\ll \eta^{1/2} (\log X)^c \frac{\eta^2 XY}{\log X},
$$  for some constant absolute $c>0$, and this is negligible upon taking $B_0 \ge 2c+1$ say.  We recall for clarity that $\eta = (\log X)^{-B_0}$ and that the only condition that we need $B_0$ to satisfy is $B_0 \ge \max(1, 2c+1) = 2c+1$.

Now it suffices to consider the contribution of $f_S$.  In particular, it suffices to prove the following Proposition.

\begin{prop} \label{prop:bilinear2}
Suppose that $f_S$ is defined as above.  Then for any $A>0$, 
\begin{equation}\label{eqn:bilinear1}
\sum_{R} c_R \sum_{\substack{V < S\le 2V \\ RS \in \A}} f_S \ll_A \frac{XY}{\log^A X},
\end{equation}for $X^{1+\tau}\ll V \ll X^{3/2(1-\tau)}$, $\tau$ as in \eqref{eqn:tau0} and where $c_R$ satisfies $c_R \ll 1$.  The implied constant depends only on $A$.
\end{prop}

\section{Sketch of the proof of Proposition \ref{prop:bilinear2}}
Proposition \ref{prop:bilinear2} contains the main new features of this paper.  The proof requires a good amount of work and the details are somewhat intricate, so we provide a sketch of the proof here to give the reader a rough road map for the rest of the paper.  When studying the sum 
$$\sum_{R} c_R \sum_{\substack{V < S\le 2V \\ RS \in \A}} f_S,$$ after some technicalities, we replace $R$ and $S$ by their generators $\alpha$ and $\beta$.  This may be done in a unique way, by demanding that  
\begin{align}
N(\beta)^{1/3} \epsilon_0^{-1/2} < \beta \le N(\beta)^{1/3} \epsilon_0^{1/2},
\end{align} where $\epsilon_0 = 1 + \sqrt[3]{2} + \sqrt[3]{4}$ is a fundamental unit of $K$. 

Using a small abuse of notation and writing $f_\beta = f_{(\beta)}$, an application of Cauchy-Schwarz then reduces Proposition \ref{prop:bilinear2} to proving the upper bound 
$$\sum_{\substack{V < N(\beta_1), N(\beta_2) \le 2V\\ \beta_1 \neq \beta_2}} f_{\beta_1}f_{\beta_2} \sumstar_{\alpha} 1 \ll \frac{Y^2V}{X(\log X)^C},
$$where $\sumstar_{\alpha}$ is restricted to a sum over $\alpha$ satisfying a number of conditions depending on $\beta_1$ and $\beta_2$.   To be more precise, for any element $o = x+\sqrt[3]{2} y + \sqrt[3]{4} z \in \cO_K$, we write $\hat{o} = (x, y, z)$.  Then let
\begin{align}
L_1(\alpha) &= (c, b, a), \notag \\
L_2(\alpha) &= (b, a, 2c),\textup{ and} \notag \\
L_3(\alpha) &= (a, 2c, 2b).
\end{align}The conditions inherent in $\sumstar_\alpha$ are
\begin{align}\label{eqn:L1L2cond2prevcopy}
L_1(\alpha) \cdot \hat \beta_i &= 0, \notag \\
L_2(\alpha) \cdot \hat \beta_i &= y_i\in (Y, Y(1+\eta)], \notag \\
L_3(\alpha) \cdot \hat \beta_i &= x_i\in (X, X(1+\eta)], \\
\end{align}for $i=1, 2$ and where $\epsilon_0 = 1 + \sqrt[3]{2} + \sqrt[3]{4}$ is a fundamental unit as before.  These are inherited from the condition that $\alpha \beta_i = x_i + \sqrt[3]{2} y_i + \sqrt[3]{4}\cdot 0 \in \A$.

The first line of \eqref{eqn:L1L2cond2prevcopy} implies that if $\alpha$ exists, then it is essentially uniquely determined by $\beta_1$ and $\beta_2$.  To be precise, letting $\gamma_i = (w_i, v_i, u_i) = L_1(\hat \beta_i)$, some work quickly yields that for $\gamma_1 \neq \gamma_2$,
\begin{equation}\label{eqn:alphag1timesg2prevcopy}
\hat \alpha = \pm \Delta(\beta_1, \beta_2)^{-1} \gamma_1 \times \gamma_2,
\end{equation}where $\times$ denotes the usual cross product, and $\Delta(\beta_1, \beta_2)$ is the greatest common divisor of the coordinates of $\gamma_1\times \gamma_2$.

Actually, for most $\beta_1$ and $\beta_2$, there is no $\alpha$ satisfying all the conditions in \eqref{eqn:L1L2cond2prevcopy}.  In fact the conditions above imply that $(\hat \beta_1, \hat \beta_2)$ must reside in a narrow region and the first task is to determine this region.  Some analysis using the second line of \eqref{eqn:L1L2cond2prevcopy} gives that
\begin{equation}\label{eqn:beta1beta2prevcopy}
\|\gamma_1 \times \gamma_2\| \ll \frac{Y}{X} V^{2/3},
\end{equation}and
\begin{equation}\label{eqn:Dbddprecopy}
\Delta(\beta_1, \beta_2) \ll \frac{Y}{X}\frac{V}{X}.
\end{equation}
The proof of this is completed in Lemma \ref{lem:deltageom} and the lines that immediately follow.  It is then immediate that for fixed $\gamma_1$, $\gamma_2$ is restricted to be in a cylinder of height $V^{1/3}$ and radius bounded by $\frac{Y}{X} V^{1/3}$.  Some more careful analysis yields a stronger result in Lemma \ref{lem:beta2region}.  In the generic case where $\|\gamma_1 \times L_2(\hat \beta_1)\| \asymp V^{2/3}$, Lemma \ref{lem:beta2region} implies that for fixed $\beta_1$, $\hat \beta_2$ is in a rectangular prism of dimensions $V^{1/3}, \frac{Y}{X} V^{1/3}$ and $\bfrac{Y}{X}^2 V^{1/3}$, which saves us an additional factor of $\frac{Y}{X}$.  Here, we are neglecting a significant technicality, which is that one must also estimate what occurs when $\|\gamma_1 \times L_2(\hat \beta_1)\|$ is smaller than the usual $V^{2/3}$ - in this case, there is a less strong restriction on $\beta_2$, but the region for $\beta_1$ is forced to be smaller.

Using this, and some standard estimates, it is then possible to show that
$$ 
\sum_{\substack{V < N(\beta_1), N(\beta_2) \le 2V\\ \beta_1 \neq \beta_2}} f_{\beta_1}f_{\beta_2} \sumstar_{\alpha} 1 \ll \frac{Y^2V}{X} (\log X)^a,
$$for some $a>0$.  This is insufficient for our purposes, being off by a factor of a large power of $(\log X)$.  However, so far, we have used no information about $f_{\beta_i}$ (aside from its size) and this bound represents a rough count of the number of summands.  

The analysis above implies that we should only consider pairs $\beta_1, \beta_2$ with $(\hat \beta_1, \hat \beta_2) \in \mathfrak R$ for some narrow region $\mathfrak R \in \mathbb{R}^6$.  By standard arguments, the condition $D = \Delta(\beta_1, \beta_2)$ may be replaced by a condition like $D|\gamma_1\times\gamma_2$, which in turn may be rewritten as 
$$\hat \beta_1 \equiv \lambda \hat \beta_2  \bmod D,
$$for some $\lambda \bmod D$.  Morally, we now need to study a sum like
$$\sum_{D \ll \frac{Y}{X}\frac{V}{X}}  \sum_{\lambda \bmod D} \sum_{\substack{(\hat \beta_1, \hat \beta_2) \in \mathfrak R \\ \hat \beta_1 \equiv \lambda \hat \beta_2 \bmod D}} f_{\beta_1}f_{\beta_2}.
$$

The local conditions $(\hat \beta_1, \hat \beta_2) \in \mathfrak R$ and $\hat \beta_1 \equiv \lambda \hat \beta_2 \bmod D$ are examined separately.  First, we restrict $\beta_i \in \C_i$, for small cubes $\C_i$.  The purpose is to replace the conditions \eqref{eqn:L1L2cond2prevcopy} on $\beta_i$ by conditions on $\C_i$ instead.  This begins in \S \ref{sec:reductiontocubes}, where the proof is reduced to the proof of Propositions \ref{prop:Cf1cubes} and \ref{prop:Cf2cubes}.

Very roughly speaking, these two Propositions deal with the cases when $\C_1 \times \C_2 \subset \mathfrak R$ and when $\C_1 \times \C_2 \not \subset \mathfrak R$.  The case $\C_1 \times \C_2 \not \subset \mathfrak R$ is conceptually easier.  Here, of course, we assume $\C_1 \times \C_2 \cap \mathfrak R \neq \varnothing$, so $\C_1\times \C_2$ intersects the boundary of $\mathfrak R$.  The main idea behind this estimate is to show that there are relatively few $\C_1 \times \C_2$ which intersect the boundary of a "nice" region like $\mathfrak R$ compared to the number of $\C_1 \times \C_2$ completely contained within the interior of $\mathfrak R$.  This is accomplished in \S \ref{sec:Cf2cubes}, the main input of which is Lemma \ref{lem:boundarycounting}.

In the case when $\C_1 \times \C_2 \subset \mathfrak R$, we fix $\C_1, \C_2$ and replace all the conditions on $\beta_i$ by $\beta_i \in \C_i$, and are left to study sums of the form
$$\sum_{D \ll \frac{Y}{X}\frac{V}{X}}  \sum_{\lambda \bmod D} \sum_{\substack{\hat \beta_i \in \C_i \\ \hat \beta_1 \equiv \lambda \hat \beta_2 \bmod D}} f_{\beta_1}f_{\beta_2}.
$$
In \S \ref{sec:Cf1cubes}, we express the congruence condition $\hat \beta_1 \equiv \lambda \hat \beta_2 \bmod D$ using additive characters $\bmod D$.  The reader will not be surprised to learn that the proof naturally splits into two cases.  The case when the modulus is not too small may be handled by the large sieve type estimate in Lemma \ref{lem:largesieveexpbdd}.  The small modulus case needs a Siegel Walfisz type estimate as in Lemma \ref{lem:SWbdd}.  

The Siegel Walfisz type estimate we need is unavailable in the literature, and so the rest of the paper is devoted to establishing this estimate for our coefficients $f_\beta$.  This begins in \S \ref{sec:reductiontogrossencharacters}.  First, the condition $\hat \beta \in \C$ is expressed using Hecke Grossencharacters.  The main term is then extracted as in Proposition \ref{prop:Mcomparison}, so that the result is reduced to a bound on a sum over Grossencharacters as in Proposition \ref{prop:errorsmallcube}.

The proof of Proposition \ref{prop:errorsmallcube} is completed in \S \ref{sec:errorsmallcube}.  Our Lemma \ref{lem:SWbdd} is essentially as good as could be reasonably expected, being the analogue of the result about primes in short intervals of the form $(x, x+x^{7/12 +\epsilon})$ due to Huxley \cite{Hux}.  This requires a bit of extra care.

In \S \ref{subsec:largesieve}, we state (and briefly prove) the analogues of now standard results in the case of the zeta function over $\mathbb{Q}$, namely a large sieve bound due to Duke \cite{Du} in Lemma \ref{lem:dukelargesieve} and some large value estimates analogous to those of Montgomery \cite{Mont} and Huxley \cite{Hux} in Lemmas \ref{lem:mont} and \ref{lem:huxmont} respectively.  

Since we are dealing with a sequence that could possibly be supported on a product of many prime ideals, we avoid going through the zero density route, and instead apply a combinatorial decomposition of $\Lambda_K$ known as Heath-Brown's identity \cite{HBidentity}.  Unfortunately, there is a lengthy separation of variables in \S \ref{subsec:decomp}.  On a first reading, the reader can safely pretend that any smooth function with sufficiently small derivatives may be ignored with negligible error (via Mellin inversion).  

After reduction to a discrete set as \S \ref{subsec:reductiontodiscrete}, it then suffices to examine a product of Dirichlet polynomials.   For a Dirichlet polynomial $F$ of length $\fw$, there are two cases.  The first is that the coefficients of $F$ involve an analogue of the Mobius function $\mu_K$.  In this case, the application of Heath-Brown's identity insures that $\fw$ is not too large.  To be precise, we construct our decomposition to force $\fw \le T_0^{9/5}$.

The reader should recall that in applications of the large sieve, optimal results occur when the length of the sum is of comparable size to the number of harmonics.  Our family of Grossencharacters along with harmonics of the form $N(I)^{it}$ forms a family of around size $T_0^3$ harmonics.  Thus, we apply our large values estimates to $F^g$ where 
$$T_0^{12/5} \le \frak w^g \le T_0^{18/5}
$$to gain good results in \S \ref{subsec:largevalueapplication} (noting that then the length of $F^g$ is close to $T_0^3$).

When $\fw > T_0^{9/5}$, the coefficients of $F$ are friendlier, involving only smooth functions.  Here, one uses a good bound on the fourth moment of certain $L$-functions to gain some additional advantage, which is completed in \S \ref{subsubsec:T0big}.  Finally, estimates which essentially follow from an improved zero free region due to Coleman \cite{Cole} are compiled in \S \ref{subsec:zerofree}.

\section{Proof of Proposition \ref{prop:bilinear2}: determining the narrow region}
The proof of Proposition \ref{prop:bilinear2} is the main part of our treatment of the bilinear sum and is the most original part of this work.  For any $\alpha = a + b\sqrt[3]{2} + c\sqrt[3]{4}$, we let $\hat \alpha = (a, b, c)$.  We call $\alpha$, the ideal $(\alpha)$ and the vector $\hat \alpha$ primitive if $\alpha$ is not divisible by an integer prime.  Note that elements of $\A$ are by definition primitive, and as a consequence both $R$ and $S$ are primitive for $RS \in \A$.  We thus assume that $c_R$ and $f_S$ are supported on primitive ideals for the rest of this discussion.  Having done that, we first proceed to remove the condition that elements of $\A$ are primitive.  For notational convenience, let
\begin{equation}\label{eqn:Ahat}
\hat \A = \{(x+y\sqrt[3]{2}) : x \in ]X, X(1+\eta)], y\in ]Y, Y(1+\eta)]\}.
\end{equation}Then
\begin{align*}
\sum_{R} c_R \sum_{\substack{V < S\le 2V \\ RS \in \A}} f_S
&= \sum_{(x+y\sqrt[3]{2}) \in \hat A} \sum_{d|(x, y)} \mu(d) \sum_{R}  \sum_{\substack{V < S\le 2V \\ RS =  (x+y\sqrt[3]{2})}} c_Rf_S\\
&= \sum_{(x+y\sqrt[3]{2}) \in \hat A} \sum_{R}  \sum_{\substack{V < S\le 2V \\ RS = (x+y\sqrt[3]{2})}} c_Rf_S + \E_0,
\end{align*}where $\E_0$ is the contribution of those $d$ with $d>1$. Since $(d)|RS$ and both $R$ and $S$ are free of prime ideal factors smaller than $X^{\delta}$, we must have that there is a prime ideal $P$ with $N(P) > X^{\delta}$ such that $P|(d)$.  Thus we must have $d> X^{\delta/3}$.  Thus, using the trivial bound $f_S \ll \tau(S) \log X$,
\begin{align*}
\E_0 &\ll \log X \sum_{X^{\delta/3} < d \ll X} \sum_{\substack{(x + y\sqrt[3]{2}) \in \hat A\\ d|(x, y)}} \sum_{S|(x + y\sqrt[3]{2})} \tau(S)\\
&\ll \log X \sum_{X^{\delta/3} < d \ll X} \sum_{\substack{(x + y\sqrt[3]{2}) \in \hat A\\ d|(x, y)}} \tau(x + y\sqrt[3]{2})^2\\
&\ll \log X \sum_{X^{\delta/3} < d \ll X} \tau(d)^2 \sum_{x \ll X/d, y\ll Y/d} \tau(x + y\sqrt[3]{2})^2\\
&\ll XY \log^{C_1} X \sum_{X^{\delta/3} < d \ll X} \frac{\tau(d)^2}{d^2} \\
&\ll X^{1-\delta/3}Y\log^{C_2} X,
\end{align*}for some constants $C_1, C_2>0$ where we have handled the divisor sums using Lemmas 4.7 and 4.2 of \cite{HB}.  This bound is acceptable for Proposition 3.

Thus it remains to study
\begin{equation}
\T_1 := \sum_{R} c_R \sum_{\substack{V < S\le 2V \\ RS \in \hat \A}} f_S .
\end{equation}
It is now convenient to replace $R$ and $S$ by their generators $\alpha$ and $\beta$, say, and write, by an abuse of notation $f_\beta = f_{(\beta)}$.  In doing so, we choose the associate $\beta>0$ satisfying 
\begin{align}\label{eqn:betadef}
N(\beta)^{1/3} \epsilon_0^{-1/2} < \beta \le N(\beta)^{1/3} \epsilon_0^{1/2},
\end{align} where recall $\epsilon_0 = 1 + \sqrt[3]{2} + \sqrt[3]{4}$ is a fundamental unit of $K$.  Note that all positive generators of $(\beta)$ must be of the form $\epsilon_0^n \beta$ for some $n\in \mathbb{Z}$, which all have the same norm.  It is clear that a $\beta$ satisfying  \eqref{eqn:betadef} exists, since multiplying any positive generator by an appropriate power of $\epsilon_0$ gives a generator satisfying \eqref{eqn:betadef}.  Moreover, \eqref{eqn:betadef} uniquely fixes our choice of $\beta$.  Note that fixing $x, y$ and $\beta$ determines the choice of $\alpha$ when we demand that $\alpha \beta = x+y\sqrt[3]{2}$.  Thus, we may write
\begin{equation}\label{eqn:T1}
\T_1 := \sum_{\alpha} c_{\alpha} \sum_{\substack{V < N(\beta)\le 2V \\ (\alpha\beta) \in \hat \A}} f_\beta,
\end{equation}for $\beta$ satisfying \eqref{eqn:betadef}.

It will be convenient later to have control over the components of $\alpha$ and $\beta$ as well, for which we record the following Lemma.

\begin{lem}\label{lem:abcomponents}
	Suppose $a, b, c\in \mathbb{Z}$ and $\lambda = a + b2^{1/3} + c4^{1/3} \asymp Z^{1/3}$ and $N(\lambda) \asymp Z$.  Then $a, b, c \ll Z^{1/3}$
\end{lem}
\begin{proof}
	We may write
	$$Z \asymp N(\lambda) = \lambda \lambda' \lambda''
	$$where for $\xi$ a primitive third root of unity, $\lambda' = a + b \xi 2^{1/3} + c\xi^2 4^{1/3}$ and $\lambda''= a + b\xi^2 2^{1/3} + c \xi 4^{1/3}$ are conjugates of $\lambda$.  Note that $|\lambda'| = |\lambda''|$ (indeed, they are complex conjugates), so $|\lambda| \asymp |\lambda'| \asymp |\lambda''| \asymp Z^{1/3}$.  Solving for $a, b$ and $c$ in terms of $\lambda, \lambda'$ and $\lambda''$ completes the proof.
\end{proof}
Now, let
\begin{equation}\label{eqn:T2}
\T_2 = \sum_{\substack{V < N(\beta_1), N(\beta_2) \le 2V\\ \beta_1 \neq \beta_2}} f_{\beta_1}f_{\beta_2} \sumstar_{\alpha} 1,
\end{equation}where $\sumstar_\alpha$ denotes a sum over algebraic integers $\alpha$ satisfying $\alpha \beta_i = x_i + y_i \sqrt[3]{2}$ for some $x_i \in ]X, X(1+\eta)]$ and $y_i \in ]Y, Y(1+\eta)]$, and the sum over $\beta_i$ satisfies \eqref{eqn:betadef}.  We first apply Cauchy - Scharwz to the bilinear sum $\T_1$ to see that 
\begin{equation}
\T_1 \ll \bfrac{X^3}{V}^{1/2} \left( \T_2 + O(XY (\log X)^c)\right)^{1/2},
\end{equation}for some $c>0$ where the $O(XY(\log X)^c)$ term arises from the diagonal term
$$\sum_{V < N(\beta)\le 2V} |f_{\beta}|^2 \sumstar_{\alpha} 1 \ll XY(\log X)^c,
$$again with similar restrictions on $\alpha$ and $\beta$ as in \eqref{eqn:T2}.  Here, we have used that $|f_\beta| \leq |e_\beta| + |h_\beta| \ll \tau(\beta)\log X$, keeping in mind that $n\ll \frac{1}{\delta} \asymp 1$.  

This gives an acceptable total contribution of $\frac{X^2\sqrt{Y}}{V^{1/2}}(\log X)^c \ll XY X^{\frac{\gamma - \tau}{2}}(\log X)^c \ll XY (\log X)^{-C}$ for any $C>0$ recalling that $Y = X^{1-\gamma}$ and $V \ge X^{1+\tau}$ for fixed $\gamma$ and $\tau$ with $\gamma < \tau$.  

It thus suffices to show that
\begin{equation}\label{eqn:T2bdd}
\T_2 \ll \frac{Y^2V}{X(\log X)^C},
\end{equation}
for any $C>0$.

We now let
\begin{equation}
\alpha = a + b\sqrt[3]{2} + c \sqrt[3]{4},
\end{equation}and
\begin{equation}
\beta_i = u_i + v_i \sqrt[3]{2} + w_i \sqrt[3]{4}.
\end{equation}Recall the notation $\hat \alpha = (a, b, c) \in \mathbb{Z}^3$ and similarly for $\hat \beta_i$.  For notational convenience, we let
\begin{align}\label{eqn:L1L2cond}
L_1(\alpha) &= (c, b, a), \notag \\
L_2(\alpha) &= (b, a, 2c),\textup{ and} \notag \\
L_3(\alpha) &= (a, 2c, 2b).
\end{align}We have the conditions
\begin{align}\label{eqn:L1L2cond2}
L_1(\alpha) \cdot \hat \beta_i &= 0, \notag \\
L_2(\alpha) \cdot \hat \beta_i &= y_i\in (Y, Y(1+\eta)], \notag \\
L_3(\alpha) \cdot \hat \beta_i &= x_i\in (X, X(1+\eta)], \textup{ and}  \notag \\
N(\beta_i)^{1/3} \epsilon_0^{-1/2} &< \beta_i \le N(\beta_i)^{1/3} \epsilon_0^{1/2},
\end{align}for $i=1, 2$ and where $\epsilon_0 = 1 + \sqrt[3]{2} + \sqrt[3]{4}$ is a fundamental unit as before.  Now, a preliminary estimate shows that for most $\beta_1, \beta_2$, there does not exist $\alpha$ satisfying the above.  In fact, as we will see, the conditions above further imply that $(\hat \beta_1, \hat \beta_2)$ is in a narrow region.  

The first line of \eqref{eqn:L1L2cond2} implies that $aw_i + bv_i + cu_i = 0$ for $i=1, 2$.  These two equations imply that $\hat \alpha$ is on a line, when $\beta_1$ and $\beta_2$ are fixed.  Letting $\gamma_i = (w_i, v_i, u_i) = L_1(\hat \beta_i)$, we see that since $\hat \alpha$ is primitive and $\gamma_1 \neq \gamma_2$,
\begin{equation}\label{eqn:alphag1timesg2}
\hat \alpha = \pm \Delta(\beta_1, \beta_2)^{-1} \gamma_1 \times \gamma_2,
\end{equation}where $\times$ denotes the usual cross product, and $\Delta(\beta_1, \beta_2)$ is the greatest common divisor of the coordinates of $\gamma_1 \times \gamma_2$.  Thus $\alpha$ is determined up to sign by $\beta_1$ and $\beta_2$.  The condition that $L_2(\alpha)\cdot \hat \beta_i > 0$ also shows there can only be one choice for $\alpha$.

Note \eqref{eqn:betadef} and Lemma \ref{lem:abcomponents} implies that
\begin{equation}\label{eqn:betanormbdd}
\|\hat \beta_i\| \asymp V^{1/3},
\end{equation}where $\|\cdot \|$ denotes the usual Euclidean norm.  Moreover, since $N(\alpha) \asymp X^3/V$, Lemma \ref{lem:abcomponents} also implies that
\begin{equation}\label{eqn:alphanormbdd}
\|\hat \alpha\| \asymp \frac{X}{V^{1/3}}.
\end{equation}
Indeed, Lemma \ref{lem:abcomponents} implies that $a, b, c \ll \frac{X}{V^{1/3}}$ for $\hat \alpha = (a, b, c)$ and this implies \eqref{eqn:alphanormbdd} given $N(\alpha) \asymp X^3/V$.
We now note that $L_1(\alpha)$ and $L_2(\alpha)$ cannot point in the same direction.  The following Lemma makes this precise.

\begin{lem}\label{lem:L1L2ind}
With notation as above, we have that
\begin{equation}
\|L_1(\alpha) \times L_2(\alpha)\| \gg \|L_1(\alpha)\|\|L_2(\alpha)\| \gg \frac{X^2}{V^{2/3}}.
\end{equation}
\end{lem}

\begin{proof}
Note that \eqref{eqn:alphanormbdd} implies that 
$$\|L_i(\alpha)\| \asymp \frac{X}{V^{1/3}}$$ for $i = 1, 2, 3$.  Define the linear operator $L_\alpha: \mathbb{R}^3 \rightarrow \mathbb{R}^3$ by 
$$L_\alpha(\hat \beta') = \widehat{\alpha \beta'},
$$for all $\hat \beta' \in K$.  Then the matrix representation of $L_\alpha$ has rows $L_3(\alpha), L_2(\alpha)$ and $L_1(\alpha)$, and one of the definitions of norm gives that $N(\alpha) = \det(L_\alpha)$.  Thus,
\begin{align*}
\frac{X^3}{V} \asymp \det(L_\alpha) \leq \|L_3(\alpha)\| \|L_1(\alpha) \times L_2(\alpha)\|,
\end{align*}from which the stated bound follows.

\end{proof}
Now we write 
\begin{equation}\label{eqn:betainspan}
\hat \beta_i = c_i L_1(\alpha)\times L_2(\alpha) + \delta_i
\end{equation}
where $\delta_i$ is in the span of $\{L_1(\alpha), L_2(\alpha)\}$.  This can be done in a unique way since $\{L_1(\alpha), L_2(\alpha), L_1(\alpha)\times L_2(\alpha)\}$ comprises a linearly independent set by Lemma \ref{lem:L1L2ind}.  The reader may think of \eqref{eqn:betainspan} as defining $\delta_i$ in terms of $\beta_1$ and $\beta_2$, keeping in mind that $\alpha$ is uniquely determined by $\beta_1$ and $\beta_2$ upon recalling \eqref{eqn:alphag1timesg2} and the following discussion.  Then the second line in \eqref{eqn:L1L2cond2} implies
\begin{equation}
L_j(\alpha)\cdot \delta_i = L_j(\alpha)\cdot \hat \beta_i \ll Y
\end{equation}for $j = 1, 2$.  This implies that $\|\delta_i\|$ must be small.  To be precise, we  have the following Lemma.

\begin{lem}\label{lem:deltageom}
	Let $\delta$ be in the span of $\{L_1(\alpha), L_2(\alpha)\}$ such that
	$$L_j(\alpha) \cdot \delta \ll Y,
	$$for $j=1, 2$.  Then for $\mathfrak A = \|L_1(\alpha)\| \asymp \|L_2(\alpha)\| \asymp \|\hat \alpha\|$,
	\begin{equation}
	\|\delta\| \ll \frac{Y}{\mathfrak A} \ll \frac{YV^{1/3}}{X}.
	\end{equation}
\end{lem}
\begin{proof}
	Write $\delta = u_1L_1(\alpha) + u_2 L_2(\alpha)$. By Lemma \ref{lem:L1L2ind}, we have that there exists some $\theta$ with $|\theta| < 1$ such that
	$$L_1(\alpha) \cdot L_2(\alpha) = \theta \|L_1(\alpha)\|\|L_2(\alpha)\|.
	$$
Indeed, Lemma \ref{lem:L1L2ind} implies that there exists some fixed constant $c'>0$ such that $1-|\theta| > c'$ for all $\alpha$ satisfying \eqref{eqn:L1L2cond2} independently of $\alpha$.	
	Suppose without loss of generality that $|u_1| \|L_1(\alpha)\| \ge |u_2| \|L_2(\alpha)\|.$  Then
	\begin{align*}
	|u_1| \|L_1(\alpha)\|^2 
	&\ge (1-|\theta|)|u_1| \|L_1(\alpha)\|^2 + |\theta||u_2|\|L_1(\alpha)\|\|L_2(\alpha)\|\\
	&\ge (1-|\theta|)|u_1| \|L_1(\alpha)\|^2 + |u_2L_1(\alpha)\cdot L_2(\alpha)\|,
	\end{align*}so that
	\begin{align*}
	Y &\gg |\delta \cdot L_1(\alpha)|\\
	&\geq |u_1| \|L_1(\alpha)\|^2 - |u_2L_1(\alpha)\cdot L_2(\alpha)\|\\
	&\geq (1-|\theta|)|u_1| \|L_1(\alpha)\|^2,
	\end{align*}from which it follows that
	\begin{equation}
	\|\delta\| \ll |u_1| \|L_1(\alpha)\| \ll \frac{Y}{\|L_1(\alpha)\|} \asymp \frac{Y}{\mathfrak A}, 
	\end{equation}as desired.
\end{proof}
Applying our Lemma \ref{lem:deltageom} to \eqref{eqn:betainspan} implies that
\begin{equation}\label{eqn:beta1condbeta2}
\hat \beta_1 = \frac{c_1}{c_2} \hat \beta_2 + O\bfrac{Y}{\mathfrak A},
\end{equation}upon noting that $c_1 \asymp c_2$, so that $\hat \beta_1$ points in roughly the same direction $\hat \beta_2$.  This implies that
\begin{equation}\label{eqn:beta1beta2}
\|\gamma_1 \times \gamma_2\| \ll \frac{V^{1/3}Y}{\mathfrak A} \ll \frac{Y}{X} V^{2/3}.
\end{equation}From \eqref{eqn:alphag1timesg2} and \eqref{eqn:beta1beta2}
\begin{equation}\label{eqn:Dbdd}
\Delta(\beta_1, \beta_2) \ll \frac{V^{1/3}Y}{\mathfrak A^2} \ll \frac{Y}{X}\frac{V}{X}.
\end{equation}
The conditions \eqref{eqn:beta1beta2} and \eqref{eqn:Dbdd} restrict our $(\hat{\beta_1}, \hat{\beta_2})$ to a very narrow region in $\mathbb{R}^6$, and is indicative of the new features on this work.  However, it turns out that these conditions are not enough by themselves.

We will need to partition and discard certain inconvenient parts of our sum. We first quote two Lemmas.  The first is Lemma 4.5 from \cite{HB}.
\begin{lem}\label{lem:cubedivisorsum}
Let $\C = (a_1, a_1+S_0] \times (a_2, a_2+S_0] \times (a_3, a_3+S_0]$ be a cube of side $S_0$, and suppose that $\max_i |a_i| \le S_0^A$ for some positive constant $A$. For any $\beta =x+y\sqrt[3]{2} + z \sqrt[3]{4} \in \cO_K$, write $\hat \beta = (x,y,z)$. Then there is a constant $c(A)$ such that
$$\sum_{\hat \beta \in \C} \tau(\beta)^2 \ll S_0^3(\log S_0)^{c(A)}.
$$
\end{lem}

The second is the closely related Lemma 11.1 from \cite{HB}.

\begin{lem}\label{lem:cubebdd}
Let $\C_1, \C_2$ be cubes of side $S_0$, not necessarily containing the origin.  Suppose that $\C_1$ and $\C_2$ are included in a sphere, centered on the origin, of radius $S_0^A$ for some $A>0$.  Then if the vectors $\hat \beta_i$ are restricted to be primitive, we will have
\begin{equation}\label{eqn:cubebddtau}
\sum_{\substack{\hat \beta_i \in \C_i\\ D|\gamma_1 \times \gamma_2}} \tau(\beta_1)^2 \ll \frac{S_0^6}{D^2} (\log S_0)^{c(A)}, 
\end{equation}for some constant $c(A)$, providing that $D\ll S_0$.
This further implies that
\begin{equation}\label{eqn:cubebddf}
\sum_{\substack{\hat \beta_i \in \C_i\\ D|\gamma_1 \times \gamma_2}} |f_{\beta_1}f_{\beta_2}| \ll \frac{S_0^6}{D^2} (\log S_0)^{c(A)},
\end{equation}and
\begin{equation}\label{eqn:cubebddfDDelta}
\sum_{\substack{\hat \beta_i \in \C_i\\ D = \Delta(\beta_1, \beta_2)}} |f_{\beta_1}f_{\beta_2}| \ll \frac{S_0^6}{D^2} (\log S_0)^{c(A)},
\end{equation}
for some constant $c(A)$.
\end{lem}
\begin{proof}
The first part of the Lemma above is exactly Lemma 11.1 from \cite{HB}.  The second bound stated in \eqref{eqn:cubebddf} follows from \eqref{eqn:cubebddtau} by noting that $|f_{\beta_1}f_{\beta_2}| \ll \tau(\beta_1)\tau(\beta_2) \log^2X \ll (\tau(\beta_1)^2 + \tau(\beta_2)^2)\log^2X$.  Note that a slight abuse of notation has occurred; the $c(A)$ appearing in equations \eqref{eqn:cubebddtau} and \eqref{eqn:cubebddf} are not necessarily the same.

Finally, \eqref{eqn:cubebddfDDelta} follows from \eqref{eqn:cubebddf} since $\Delta(\beta_1, \beta_2)|\gamma_1\times \gamma_2$ so that the condition $D = \Delta(\beta_1, \beta_2)$ implies that $D|\gamma_1 \times \gamma_2$.
\end{proof}

Using Lemma \ref{lem:cubebdd} and the bounds \eqref{eqn:beta1beta2} and \eqref{eqn:Dbdd} would lead to the trivial bound of 
\begin{equation}\label{eqn:trivialT2}
\T_2 \ll YV \log^c X,
\end{equation} for some constant $c$.  Since we state the bound \eqref{eqn:trivialT2} for motivating reasons only, we provide a simplified justification.  The reader may look at Lemma \ref{lem:T2discard} for the details of a similar proof.  Here, we assume that $\Delta(\beta_1, \beta_2) \asymp \frac{Y}{X} \frac{V}{X}$, which is essentially the largest size for $\Delta$ by \eqref{eqn:Dbdd}.  Then we partition the sums over $\beta_1$ and $\beta_2$ into a sum over cubes $\C_1$, $\C_2$ with side $S_0 = \frac{Y}{X} V^{1/3}$.  We will need $V/S_0^3$ cubes $\C_1$ to cover the entire range for $\beta_1$.  By \eqref{eqn:beta1beta2}, it is immediate that for fixed $\gamma_1$, $\gamma_2$ is restricted to be in a cylinder of height $V^{1/3}$ and radius bounded by $\frac{Y}{X} V^{1/3}$.  Thus, if we fix a cube $\C_1$ with $\beta_1 \in \C_1$, the above discussion implies that the number of cubes $\C_2$ with $\beta_2 \in \C_2$ is bounded by $\bfrac{Y}{X}^2 V/S_0^3$.  Then, our estimate for $\T_2$ is
\begin{align}\label{eqn:blah}
&\sum_{D \asymp \frac{Y}{X} \frac{V}{X}} \sum_{\substack{V< N(\beta_1), N(\beta_2) \le 2V \\ \beta_1 \neq \beta_2\\ D = \Delta(\beta_1, \beta_2)}} |f_{\beta_1}f_{\beta_2} |\notag \\
&\ll \sum_{D \asymp \frac{Y}{X} \frac{V}{X}} \sum_{\C_1, \C_2} \sum_{\substack{\beta_1 \in \C_1, \beta_2 \in \C_2 \\\beta_1 \neq \beta_2\\ D| \gamma_1 \times \gamma_2}} \left(\tau(\beta_1)^2 + \tau(\beta_2)^2\right)\notag \\
&\ll \sum_{D \asymp \frac{Y}{X} \frac{V}{X}} \sum_{\C_1, \C_2} \frac{S_0^6}{D^2} (\log S_0)^{c(A)},
\end{align}by Lemma \ref{lem:cubebdd}, where $\sum_{\C_1, \C_2}$ indicates a sum over cubes $\C_1$ and $\C_2$ such that there exists $\beta_i \in \C_i$ satisfying \eqref{eqn:beta1beta2}, and of course $N(\beta_i) \asymp V$.  Thus, the quantity in \eqref{eqn:blah} is
\begin{align*}
\ll  \sum_{D \asymp \frac{Y}{X} \frac{V}{X}} \bfrac{Y}{X}^2 V^2 S_0^{-6} \frac{S_0^6}{D^2} (\log S_0)^{c(A)}
\ll YV (\log X)^c,
\end{align*}for some constant $c$.

This estimate is essentially missing a factor of $\frac{Y}{X}$ as compared to the expected true size, assuming no cancellation occurs in the sum.  Thus, although the fact that $(\hat \beta_1, \hat \beta_2)$ is in a narrow region is reflected in \eqref{eqn:beta1beta2}, \eqref{eqn:beta1beta2} does not completely capture the thinness of the region.

We recover the missing factor from the additional condition
\begin{equation}
L_2(\gamma_1\times \gamma_2) \cdot \hat \beta_1 \ll \Delta Y,
\end{equation}which results from the second line of \eqref{eqn:L1L2cond2}.  By the definition of $L_2$, this is equivalent to
\begin{equation}
\gamma_2 \cdot (\gamma_1 \times L_2(\hat \beta_1)) \ll \Delta Y.
\end{equation}If we were to write $\gamma_2 = u+ c(\gamma_1\times L_2(\hat \beta_1))$ where $u\cdot \gamma_1\times L_2(\hat \beta_1) = 0$, then we see that 
\begin{equation}\label{eqn:gamma2cond2}
c\|\gamma_1\times L_2(\hat \beta_1)\|^2 = \gamma_2 \cdot \gamma_1\times L_2(\hat \beta_1) \ll \Delta Y.
\end{equation}We also have from \eqref{eqn:alphag1timesg2} and \eqref{eqn:beta1condbeta2} that
\begin{equation}\label{eqn:gamma2cond1}
\gamma_2 = c_0 \gamma_1 + O\bfrac{\Delta Y}{\|\gamma_1 \times \gamma_2\|},
\end{equation}for some constant $c_0$.  This leads to the following Lemma.

\begin{lem}\label{lem:beta2region}
For fixed $\gamma_1$ and fixed $\Delta = \Delta(\beta_1, \beta_2)$, $\gamma_2$ is restricted to be in a rectangular box with sides bounded by $V^{1/3}, \frac{\Delta Y}{\|\gamma_1 \times \gamma_2\|} \asymp \frac{Y}{X} V^{1/3}$ and $\frac{\Delta Y}{\|\gamma_1\times L_2(\hat \beta_1)\|}$ where the side of length $V^{1/3}$ is parallel to $\gamma_1$.	
\end{lem}
\begin{proof}
The equation \eqref{eqn:gamma2cond1} implies that if we write $\gamma_2$ as
\begin{equation}\label{eqn:gamma2lincombo}
\gamma_2 = c_0\gamma_1 + c(\gamma_1 \times L_2(\hat \beta_1)) + \gamma',
\end{equation}
where $\gamma'$ is orthogonal to both $\gamma_1$ and $\gamma_1 \times L_2(\hat \beta_1)$, then $\|\gamma'\| \ll \frac{Y}{\mathfrak A} \asymp \frac{\Delta Y}{\|\gamma_1\times \gamma_2\|}$ by \eqref{eqn:alphag1timesg2}.  Moreover, \eqref{eqn:gamma2cond2} implies that
$$c\|\gamma_1\times L_2(\hat \beta_1)\| \ll \frac{\Delta Y}{\|\gamma_1\times L_2(\hat \beta_1)\|}.
$$Noting that \eqref{eqn:gamma2lincombo} expresses $\gamma_2$ as a linear combination of orthogonal vectors, the Lemma follows.
\end{proof}

\begin{rem}
Recall that by \eqref{eqn:beta1beta2}, it is immediate that for fixed $\gamma_1$, $\gamma_2$ is restricted to be in a cylinder of height $V^{1/3}$ and radius bounded by $\frac{Y}{X} V^{1/3}$.  The restriction on the last dimension arising from Lemma \ref{lem:beta2region} is only stronger than this when $\gamma_1 \times L_2(\hat \beta_1)$ is not too small.  In the generic case where $\|\gamma_1 \times L_2(\hat \beta_1)\| \asymp V^{2/3}$, we see that $\frac{DY}{\|\gamma_1\times L_2(\hat \beta_1)\|} \ll \bfrac{Y}{X}^2 V^{1/3}$ by \eqref{eqn:Dbdd}, which saves us an additional factor of $\frac{Y}{X}$.  
\end{rem}

While it is possible for $\gamma_1 \times L_2(\hat{(\beta_1)})$ to have small norm, we shall show that this occurs only for a small number of $\beta_1$.  The Lemma below first shows that this constitutes a strong condition on $\gamma_1$.
\begin{lem}\label{lem:gamma1region}
	Suppose that $|w_1| \gg V^{1/3}$.  Then, the condition $|\gamma_1 \times L_2(\hat \beta_1)| \ll \N$ is equivalent to
	\begin{align}\label{eqn:u1v1}
	u_1 &= 2^{2/3}w_1 \left(1 + O\bfrac{\N}{V^{2/3}}\right), \textup{ and} \notag \\
	v_1 &= 2^{1/3} w_1 \left(1 + O\bfrac{\N}{V^{2/3}}\right).
	\end{align}
Note that \eqref{eqn:u1v1} is symmetric in $u_1, v_1, w_1$ in the sense that if $|u_1| \gg V^{1/3}$ we may express both $v_1$ and $w_1$ in terms of $u_1$ in the same manner.  Similarly so if $|v_1| \gg V^{1/3}$.
\end{lem}
\begin{proof}
	Since 
	\begin{equation}
	\gamma_1 \times L_2(\hat \beta_1) = (2v_1w_1 - u_1^2, u_1v_1 - 2w_1^2, w_1u_1 - v_1^2),
	\end{equation}
	the condition $|\gamma_1 \times L_2(\hat \beta_1)| \ll \N$ is equivalent to
	\begin{align}\label{eqn:gamma1L2beta11}
	u_1^2 &= 2v_1w_1 + O(\N) \notag \\
	u_1v_1 &= 2w_1^2 + O(\N), \textup{ and} \notag \\
	v_1^2 &= u_1w_1 + O(\N).
	\end{align}
Now, if $\N \asymp V^{2/3}$, the Lemma is trivial since $|w_1| = C_0 V^{1/3}$ for some $C_0 >0$.  Supposing that the implied constant in $O(\N)$ appearing above in \eqref{eqn:gamma1L2beta11} is $C_1>0$, we now assume that $\N < C_0^2 V^{2/3}/C_1$.  

Then, the second line in \eqref{eqn:gamma1L2beta11} immediately implies that $|u_1v_1| \ge |2C_0^2 V^{2/3} -  C_1 C_0^2 V^{2/3}/C_1| \gg V^{2/3}$ and so $u_1 \asymp v_1 \asymp V^{1/3}$.  Then \eqref{eqn:u1v1} follows easily from \eqref{eqn:gamma1L2beta11}.
\end{proof}

We now discard the part of $\T_2$ with  $|\gamma_1 \times L_2(\hat \beta_1)| \ll \bfrac{Y}{X}^{1/2} V^{2/3} (\log X)^{-B}$ for some parameter $B$ to be determined.  Specifically, we have the following Lemma.
\begin{lem}\label{lem:T2discard}
We have that
\begin{equation}
\sum_{\substack{V<N(\beta_1), N(\beta_2)\le 2V\\ |\gamma_1 \times L_2(\hat \beta_1)| \ll \bfrac{Y}{X}^{1/2} V^{2/3}(\log X)^{-B} \\ \beta_1 \neq \beta_2}} f_{\beta_1} f_{\beta_2} \sumstar_\alpha 1 \ll \bfrac{Y}{X} YV (\log X)^{-B+c(A)},
\end{equation}where $c(A)$ is a constant depending only on $A$ and as usual $\beta_i, \alpha \in \cO_K$ satisfy \eqref{eqn:L1L2cond}.
\end{lem}
\begin{proof}

 Indeed, the condition  $|\gamma_1 \times L_2(\hat \beta_1)| \ll \bfrac{Y}{X}^{1/2} V^{2/3} (\log X)^{-B}$ along with \eqref{eqn:u1v1} implies that $\gamma_1$ is inside a cylinder with height $V^{1/3}$ and radius $\bfrac{Y}{X}^{1/2}V^{1/3}(\log X)^{-B}$.  We partition our sum over $\beta_1$ and $\beta_2$ in $\T_2$ into cubes $\C_1$ and $\C_2$ of side $S_0 = \frac{DX}{V^{2/3}}$.  We require $\ll \bfrac{Y}{X} \frac{V}{S_0^3}(\log X)^{-2B}$ cubes $\C_1$ to cover our cylinder and $\ll \frac{(DX)^2}{VS_0^3}$ cubes $\C_2$ to cover our region for $\beta_2$.  Since $D = o(S_0)$, we may apply Lemma \ref{lem:cubebdd} to see that
\begin{align*}
\sum_{D} \sum_{\substack{V< N(\beta_1), N(\beta_2) \le 2V \\ \beta_1 \neq \beta_2\\ D = \Delta(\beta_1, \beta_2)}} |f_{\beta_1}f_{\beta_2} |
&\ll \sum_D \sum_{\C_1, \C_2} \sum_{\substack{\beta_1 \in \C_1, \beta_2 \in \C_2 \\\beta_1 \neq \beta_2\\ D| \gamma_1 \times \gamma_2}} \left(\tau(\beta_1)^2 + \tau(\beta_2)^2\right)\\
&\ll \sum_D \bfrac{Y}{X} X^2 (\log S_0)^{c(A)-2B}\\
&\ll \bfrac{Y}{X} \frac{Y}{X}\frac{V}{X} X^2 (\log S_0)^{c(A)-2B}\\
&\ll \bfrac{Y}{X} YV (\log S_0)^{c(A)-2B},
\end{align*}by our bound for $D$ in \eqref{eqn:Dbdd}.  
\end{proof}

If we set $2B = c(A)+C$, then the bound in the above Lemma suffices for the bound required in \eqref{eqn:T2bdd}.  Now we would like to partition the sum over $\beta_1$ over regions in which $\gamma_1 \times L_2(\hat \beta_1)$ is approximately constant.  The following Lemma helps us achieve that.
\begin{lem}\label{lem:T2dissect}
Let $\C_1$ be a cube of side $S_0 = o\left(\frac{Y}{X}V^{1/3}\right)$.  Let 
\begin{equation}
\N = \N(\C_1) = \max \{ |\gamma_1 \times L_2(\hat \beta_1)|: \hat \beta_1 \in \C_1\}.
\end{equation}
Then either $|\gamma_1 \times L_2(\hat \beta_1)| \ll \frac{Y}{X} V^{2/3}$ or  $|\gamma_1 \times L_2(\hat \beta_1)| \asymp \N$ for all $\hat \beta_1 \in \C_1$.  Here, as before, when writing $\hat \beta_1 = (u_1, v_1, w_1)$, $\gamma_1 = (w_1, v_1, u_1)$.
\end{lem}
\begin{proof}
If $\N \le \frac{Y}{X} V^{2/3}$, we are done, so we assume $\N > \frac{Y}{X} V^{2/3}$.  It now suffices to show that  $|\gamma_1 \times L_2(\hat \beta_1)| \asymp \N$ for all $\hat \beta_1 \in \C_1$.

There exists some $\hat \beta_1 = (u_1, v_1, w_1) \in \C_1$ such that $|\gamma_1 \times L_2(\hat \beta_1)| = \N$.  Now suppose for the sake of eventual contradiction that there exists some other $\hat \beta = (u, v, w) \in \C_1$ such that $\gamma \times L_2(\hat \beta)  = o(\N)$ where $\gamma = (w, v, u)$.  Without loss of generality, suppose $w = \max (u, v, w) \gg V^{1/3}$.  Then by \eqref{eqn:u1v1}, we have that
\begin{align}\label{eqn:uv}
u &= 2^{2/3}w \left(1 + o\bfrac{\N}{V^{2/3}}\right), \textup{ and} \notag \\
v &= 2^{1/3} w \left(1 + o\bfrac{\N}{V^{2/3}}\right).
\end{align}Here, the reader should note that the $o$ above results from our condition $\gamma \times L_2(\hat \beta)  = o(\N)$ in place of the condition $\gamma \times L_2(\hat \beta)  = O(\N)$ appearing before \eqref{eqn:u1v1}.

Recalling that the side of $\C_1$ is $o\bfrac{\N}{V^{1/3}}$, so that 
\begin{align}\label{eqn:gammagamma1replacement}
u_1 &= u +  o\bfrac{\N}{V^{1/3}} \\
&= 2^{2/3}w \left(1 + o\bfrac{\N}{V^{2/3}}\right)\\
&=2^{2/3}w_1 \left(1 + o\bfrac{\N}{V^{2/3}}\right), \textup{ and similarly} \notag \\
v_1 &= 2^{1/3} w_1 \left(1 + o\bfrac{\N}{V^{2/3}}\right).
\end{align}  In the above, we have noted that $|w_1| = |w| + o\bfrac{\N}{V^{1/3}} \gg V^{1/3}$ since $N \ll V^{2/3}$.  We have from \eqref{eqn:gammagamma1replacement} that $|\gamma_1 \times L_2(\hat \beta_1)| = o(\N)$, a contradiction.
\end{proof}

Now by Lemmas \ref{lem:T2discard} and \ref{lem:T2dissect}, we can write
\begin{equation}
\T_2 = \sumd_{\N, D_0} \T_3(\N, D_0) + O\left(\bfrac{Y}{X} YV (\log X)^{-C} \right),
\end{equation} for any $C >0$ and where $\sumd_{\N, D_0}$ denotes a sum over powers of two satisfying the bounds $\bfrac{Y}{X}^{1/2} V^{2/3} (\log X)^{-B} \ll \N \ll V^{2/3}$, $D_0 \ll \frac{VY}{X^2}$, and
\begin{equation}
\T_3 := \T_3(\N, D_0) := \sum_{\substack{\C_1\\ \N<\N(\C_1) \le 2\N}}S(\C_1; D_0)
\end{equation}for $\N \gg \frac{Y}{X} V^{2/3}$ and the cubes $\C_1$ has side $S_0 = \frac{D_0Y}{\N (\log X)^{C_1}}$ for a parameter $C_1>0$ to be determined, and where
\begin{equation}
S(\C_1;D_0) = \sum_{\substack{V < N(\beta_1) \le 2V\\ \hat \beta_1 \in \C_1}} \sumb_{\substack{V < N(\beta_2) \le 2V \\ \beta_1 \neq \beta_2\\ D_0<\Delta(\beta_1, \beta_2) \le 2D_0}} f_{\beta_1}f_{\beta_2}
\end{equation}
where $\sumb$ denotes a sum over $\beta_2$ which satisfies \eqref{eqn:L1L2cond2}.  Note that it now suffices to show the following Proposition.
\begin{prop}\label{prop:T3bound}
With notation as above, and with the same conditions as in Proposition \ref{prop:bilinear2}, we have that for any $C>0$,
\begin{equation}
\T_3(\N, D_0) \ll_C  \frac{Y^2V}{X(\log X)^C},
\end{equation} for $\bfrac{Y}{X}^{1/2} V^{2/3} (\log X)^{-B} \ll \N \ll V^{2/3}$ and $D_0 \ll \frac{YV}{X^2}$ and where the implied constant above depends only on $C$.
\end{prop}

\section{Proof of Proposition \ref{prop:T3bound} - reduction to cubes} \label{sec:reductiontocubes}
With notation as the last section, for $\hat \beta_1 \in \C_1$, $\N(\C_1) \asymp \N$, we have from Lemma \ref{lem:beta2region} that $\hat \beta_2$ is restricted to be in a rectangular box of dimensions $\asymp V^{1/3}, \frac{D_0Y}{\|\gamma_1 \times \gamma_2\|}$ and $\frac{D_0Y}{\N}$ respectively.  

Moreover, by Lemma \ref{lem:gamma1region}, the region for $\hat \beta_1$ is restricted to be in a rectangle with dimensions bounded by $\N/V^{1/3}, \N/V^{1/3}$, and $V^{1/3}$.  Indeed, writing $\hat \beta_1 = (u_1, v_1, w_1)$ and assuming without loss of generality that $|w_1| \gg V^{1/3}$ gives the claim immediately.

Since $\N \gg \frac{Y}{X} V^{2/3} (\log X)^{-B}$, we have $\frac{D_0Y}{\N} \ll \frac{D_0X}{V^{2/3}} (\log X)^{B} \ll \bfrac{Y}{X} V^{1/3} (\log X)^B \ll V^{1/3}$.  We therefore partition the region for $\hat \beta_2$ into cubes of side $S_0 = \frac{D_0Y}{\N \log^{C_1}X}$ as well.  We have already partitioned the sum over $\hat \beta_1$ into a sum over cubes of side $S_0$, and we further demand for convenience that we use the exact same partition for the sum over both $\hat \beta_1$ and $\hat \beta_2$. \footnote{Note that there could be more than one way to cover the same region with cubes.}  We do this in order to ensure that for $\hat \beta_1 \in \C_1$, $\hat \beta_2 \in \C_2$, $\hat \beta_1 = \hat \beta_2 \Rightarrow \C_1 = \C_2$.

The condition on $V$ from Proposition \ref{prop:bilinear2} gives $V^{2/3} \ll X^{(1-\tau)}$, so $\N \ll V^{2/3} \ll X^{1-\tau} = Y X^{\gamma - \tau}$ where recall $\tau > \gamma$.  Thus, 
\begin{equation}\label{eqn:D0S0}
D_0 \ll S_0
\end{equation} and hence results like Lemma \ref{lem:cubebdd} apply.  For the sake of clarity, note that the number of cubes required to cover all $\beta_1$ appearing in the sum $\T_3(\N)$ is of size $\frac{\N^2}{V^{1/3} S_0^3} \asymp \frac{\N^5 (\log X)^{3C_1}}{D_0^3Y^3 V^{1/3}}$, while for fixed $\beta_1$, the number of cubes required for $\beta_2$ is of size $\frac{V^{1/3} D_0^2 Y^2}{\|\gamma_1 \times \gamma_2\| \N S_0^3} = \frac{V^{1/3} D_0X}{\|\gamma_1 \times \gamma_2\| S_0^2}(\log X)^{C_1}$.  

Generically, $\|\gamma_1 \times \gamma_2\| \asymp \frac{Y}{X}V^{2/3}$, and we now discard those parts of the sum which do not satisfy this.

\begin{lem}\label{lem:T3'bdd}
	Let $\T_3'(\N, D_0)$ be that part of $\T_3(\N, D_0)$ satisfying $\|\gamma_1 \times \gamma_2\| \leq \frac{Y}{X} \frac{V^{2/3}}{H}$.  Then there exists some absolute constant $c$ such that
	\begin{equation}
		\T_3' \ll \frac{Y^2 V (\log X)^c}{XH}
	\end{equation} 
\end{lem}
\begin{proof}
	Indeed, for fixed $\beta_1$, we have already that $\hat \beta_2$ is restricted in a rectangular box of dimensions $\asymp V^{1/3}, \frac{D_0X}{\|\gamma_1 \times \gamma_2\|}$ and $\frac{D_0Y}{\N}$ respectively, where the first side of length $V^{1/3}$ is parallel to $\hat \beta_1$.  Moreover, the condition  $\|\gamma_1 \times \gamma_2\| \leq \frac{Y}{X} \frac{V^{2/3}}{H}$ implies that $\hat \beta_2$ is also restricted to be in a cylinder of height $V^{1/3}$ parallel to $\hat \beta_1$ and radius bounded by $\frac{Y}{X} \frac{V^{1/3}}{H}$.  We therefore have that for fixed $\beta_1$, $\hat{\beta_2}$ is restricted in a rectangular box of dimensions $\asymp V^{1/3}, \frac{Y}{X} \frac{V^{2/3}}{H}$ and $\frac{D_0Y}{\N}$ respectively

By \eqref{eqn:alphag1timesg2} and \eqref{eqn:alphanormbdd}, we have that 
$$\|\gamma_1\times \gamma_2\| = \|\hat{\alpha}\| |\Delta(\beta_1, \beta_2| \gg \frac{X}{V^{1/3}},$$ 
so $H \ll \frac{YV}{X^2}$, and the radius above $\frac{Y}{X} \frac{V^{1/3}}{H} \gg \frac{X}{V^{2/3} } \gg X^{\tau}$ upon recalling $V\ll X^{3/2(1-\tau)}$.  Moreover, the dimension $\frac{D_0 Y}{N} \gg \frac{Y}{V^{2/3}} \gg X^{\tau - \gamma}$, using that $Y = X^{1-\gamma}$ and again that $V \ll X^{3/2(1-\tau)}$.  Recalling that $\tau > \gamma$ as defined in \eqref{eqn:tau0}, we see that there exists $S_0' = X^{\epsilon_0}$ for some fixed $\epsilon_0 > 0$ such that $S_0' \le \min\{V^{1/3}, \frac{Y}{X} \frac{V^{2/3}}{H},\frac{D_0Y}{\N} \}$.
	
We now split the regions for $\hat \beta_1$ and $\hat \beta_2$ into cubes of side $S_0'$.  Previously, we have already noted that for fixed $\beta_1$, $\hat{\beta_2}$ is restricted to be in a rectangular region with sides of size $V^{1/3}, \frac{Y}{X} \frac{V^{2/3}}{H},\frac{D_0Y}{\N}$.  Due to the bound on $S_0'$, we see that if we fix a cube $\C_1$ of side $S_0'$, then the restriction $\hat{\beta_1} \in \C_1$ restricts $\hat{\beta_2}$ to be in a (somewhat larger) rectangular region with sides of size $V^{1/3}, \frac{Y}{X} \frac{V^{2/3}}{H},\frac{D_0Y}{\N}$.  Note here that the restriction on $\hat{\beta_2}$ now no longer depends on a fixed $\beta_1$, but merely the cube $\C_1$.

We then bound the contribution of each cube using Lemma \ref{lem:cubebdd} and see that
	\begin{align*}
	\T'_3(\N, D_0) 
	&\ll \sum_{D<\frac{Y}{X} \frac{V}{X}}\frac{\N^2}{V^{1/3} S_0^3} V^{1/3} \frac{Y V^{1/3}}{X H}\frac{DY}{\N} \frac{1}{ S_0^3} \frac{S_0^6}{D^2} (\log S_0)^{c(A)+C_1}\\
	&\ll  \frac{Y^2 V}{XH}(\log X)^{c(A)+C_1+1},
	\end{align*}upon using that $\N \ll V^{2/3}$.
	
\end{proof}

Picking $H = (\log X)^{c(A)+C_1+C+1}$, we may now assume that $\|\gamma_1 \times \gamma_2\| \gg \frac{Y}{X} \frac{V^{2/3}}{H}$ for the purpose of proving Proposition \ref{prop:T3bound}.  Moreover, note that this is equivalent to 
\begin{equation}\label{eqn:D0lowerbdd}
D_0 \gg \frac{1}{H} \frac{VY}{X^2},
\end{equation}so we will use these two conditions interchangeably. Indeed, \eqref{eqn:D0lowerbdd} follows immediately from
$$\frac{X}{V^{1/3}} \asymp \|\alpha\| \asymp \frac{\|\gamma_1 \times \gamma_2\|}{D_0},
$$which is a consequence of \eqref{eqn:alphag1timesg2} and \eqref{eqn:alphanormbdd}.

Later on, we will need to use that the side of our cubes $S_0$ is not too small.  For this purpose, we record the following Lemma.

\begin{lem}\label{lem:S0lowerbdd}
	For $D_0 \gg \frac{1}{H} \frac{VY}{X^2}$ as in \eqref{eqn:D0lowerbdd}, and with notation as above, we have that
	$$S_0 \gg V^{1/3-\epsilon}X^{-2\gamma}.
	$$
Recalling that $\gamma < \frac{5}{67}$ and $V \gg X^{1+\tau}$ for $\tau > \gamma$, we also have that
	$$S_0 \gg V^{\frac 13 - \frac{5}{36} +\epsilon}
	$$Finally, we have that
	$$S_0/D_0 \gg X^\epsilon.
	$$
\end{lem}
\begin{proof}
	We have that
	\begin{align*}
	S_0 &= \frac{D_0Y}{\N (\log X)^{C_1}} \\
	&\gg \frac{1}{H \log^{C_1}X} \frac{Y^2}{X^2} \frac{V}{\N}.
	\end{align*}
	Recalling that $H$ is a power of $\log X$, and $\N \ll V^{2/3}$, we see that
	$$S_0 \gg \bfrac{Y}{X}^2 V^{1/3-\epsilon} = V^{1/3-\epsilon} X^{-2\gamma}.
	$$The second claims follows from $\gamma < \frac{5}{72} (1+\gamma) < \frac{5}{72} (1+\tau)$ so that $X^{2 \gamma} < V^{5/36 + \epsilon}.$ 
	
	Finally, the last claim follows from
	$$\frac{S_0}{D_0} \gg \frac{Y}{\N (\log X)^{C_1}} \gg \frac{X^{1-\gamma}}{V^{2/3} (\log X)^{C_1}} \gg \frac{X^{1-\gamma}}{X^{1-\tau} (\log X)^{C_1}},
	$$which suffices since $\gamma < \tau$.

\end{proof}

Recall that
\begin{equation}
\T_3(\N, D_0) =\sum_{\substack{\C_1\\ \N<\N(\C_1) \le 2\N}}\sum_{\C_2} S(\C_1, \C_2; D_0),
\end{equation}where
\begin{equation}
S(\C_1, \C_2; D_0) =  \sum_{\substack{V < \beta_1 \le 2V\\ \beta_1 \in \C_1}} \sumb_{\substack{V < N(\beta_2) \le 2V \\ \beta_2 \in \C_2\\ \beta_1 \neq \beta_2\\ D_0< \Delta(\beta_1, \beta_2) \le 2D_0}} f_{\beta_1}f_{\beta_2}.
\end{equation}
Now let $\Cf_1(\C_1, \C_2)$ denote the condition that all $(\hat \beta_1, \hat \beta_2) \in \C_1 \times \C_2$ satisfy $V < N(\beta_1), N(\beta_2) \le 2V$, $\beta_1\neq \beta_2$, and \eqref{eqn:L1L2cond2} with $\alpha$ as in \eqref{eqn:alphag1timesg2}.  Let $\Cf_2(\C_1, \C_2)$ denote the condition that there exists some $(\hat \beta_1, \hat \beta_2) \in \C_1 \times \C_2$ satisfying $V < N(\beta_1), N(\beta_2) \le 2V,$ and \eqref{eqn:L1L2cond2}, and some $(\beta_1', \beta_2') \in \C_1\times \C_2$ not satisfying $V < N(\beta_1'), N(\beta_2') \le 2V,$ or not satisfying \eqref{eqn:L1L2cond2}.  

If neither $\Cf_1(\C_1, \C_2)$ nor $\Cf_2(\C_1, \C_2)$ hold, then $\C_1, \C_2$ does not contribute to the sum $\T_3$. We will also eliminate that part of the sum arising from $\C_1 = \C_2$.  We had previously demanded that the set of cubes $\C_1$ is the same set as the set of cubes $\C_2$, so this will imply that $\beta_1 \neq \beta_2$ for all $\beta_1 \in \C_1$ and $\beta_2 \in \C_2$, so we will drop that condition later.  The portion of the sum with $\C_1 = \C_2$ contributes to $\T_4(\N)$ a total bounded by
\begin{align*}
&\ll \sum_{D_0 < D \le 2D_0} \sum_{\C} \sum_{\substack{\beta_1, \beta_2 \in \C\\ D|\Delta(\beta_1, \beta_2)}}|f_{\beta_1} f_{\beta_2}|\\
&\ll \sum_{D_0 < D \le 2D_0} \frac{\N^2}{V^{1/3} S_0^3} \frac{S_0^6}{D^2}(\log S_0)^{c(A)}\\
&\ll D_0^2 \frac{Y^3}{V^{1/3} \N \log^{3C_1} X} (\log S_0)^{c(A)}\\
&\ll \bfrac{Y}{X}^2 \bfrac{V}{X}^2 \frac{Y^2 X}{V} (\log X)^{c(A)+B - 3C_1}\\
&\ll \bfrac{Y}{X}^2 \frac{VY^2}{X} (\log X)^{c(A) +B  - 3C_1}, 
\end{align*}where we have used that $D_0 \ll \frac{Y}{X} \frac{V}{X}$ and $\N \gg \frac{Y}{X} V^{2/3}$.  Recalling that $Y= X^{1-\gamma}$ for $\gamma >0$, the above is
$$\ll \frac{VY^2}{X (\log X)^C}
$$for any $C>0$, so this bound suffices\footnote{As one would expect, this contribution can be made smaller by taking smaller cubes.  Even if $\gamma= 0$, we may simply chose $C_1$ to be sufficiently large.}   for the bound in Proposition \ref{prop:T3bound}. In the above, we have used Lemma \ref{lem:cubebdd}, and have used that $D \asymp D_0 \ll S_0$ by \eqref{eqn:D0S0}.  Now we see that Proposition \ref{prop:T3bound} follows from the two Propositions below.

\begin{prop}\label{prop:Cf1cubes}
For all $C>0$,
\begin{equation}
\sum_{\substack{ \C_1 \neq \C_2\\ \Cf_1(\C_1, \C_2)}} S(\C_1, \C_2; D_0) \ll \frac{VY^2}{X(\log X)^C}.
\end{equation}
\end{prop}

\begin{prop}\label{prop:Cf2cubes}
There exists some constant $c$ such that,
\begin{equation}
\sum_{\substack{ \C_1 \neq \C_2\\ \Cf_2(\C_1, \C_2)}} S(\C_1, \C_2; D_0) \ll \frac{VY^2}{X(\log X)^{C_1 - c}}.
\end{equation}
\end{prop}
Note that the bound in Proposition \ref{prop:Cf2cubes} is acceptable for Proposition \ref{prop:T3bound}, as long as we choose $C_1 \ge c + C$.

\section{Proof of Proposition \ref{prop:Cf2cubes}}\label{sec:Cf2cubes}
The strategy to prove Proposition \ref{prop:Cf2cubes} is to bound $S(\C_1, \C_2; D_0)$ trivially and show that the number of $\C_1, \C_2$ satisfying $\Cf_2(\C_1, \C_2)$ is small.  We write
\begin{align}
S(\C_1, \C_2; D_0) &= \sum_{D_0 < D \le 2D_0} \sum_{\substack{V < \beta_1 \le 2V\\ \beta_1 \in \C_1}} \sumb_{\substack{V < N(\beta_2) \le 2V \\ \beta_2 \in \C_2\\ \beta_1 \neq \beta_2\\  \Delta(\beta_1, \beta_2) =D}} f_{\beta_1}f_{\beta_2}\\
&\le \sum_{D_0 < D \le 2D_0} \sum_{\substack{V < \beta_1 \le 2V\\ \beta_1 \in \C_1}} \sumb_{\substack{V < N(\beta_2) \le 2V \\ \beta_2 \in \C_2\\ \beta_1 \neq \beta_2\\ \Delta(\beta_1, \beta_2) =D}} |f_{\beta_1}f_{\beta_2}|,
\end{align}so that the quantity to be bounded in Proposition \ref{prop:Cf2cubes} is
\begin{equation}\label{eqn:prop6quantity}
\sum_{\substack{ \C_1 \neq \C_2\\ \Cf_2(\C_1, \C_2)}} S(\C_1, \C_2; D_0) \le \sum_{D_0 < D \le 2D_0} \sum_{\substack{ \C_1 \neq \C_2\\ \Cf_2(\C_1, \C_2)}} \tilde S(\C_1, \C_2; D)
\ll \sum_{D_0 < D \le 2D_0} \sum_{\substack{ \C_1 \neq \C_2\\ \Cf_2(\C_1, \C_2)}}\frac{S_0^6}{D^2} (\log S_0)^{c(A)}
\end{equation}
where
\begin{equation}\label{eqn:tildeS}
\tilde S(\C_1, \C_2; D) = \sum_{\substack{V < \beta_1 \le 2V\\ \beta_1 \in \C_1}} \sumb_{\substack{V < N(\beta_2) \le 2V \\ \beta_2 \in \C_2\\ \beta_1 \neq \beta_2\\ \Delta(\beta_1, \beta_2) =D}} |f_{\beta_1}f_{\beta_2}| \ll \frac{S_0^6}{D^2} (\log S_0)^{c(A)},
\end{equation}for some constant $c(A)$ by Lemma \ref{lem:cubebdd}.

We will need Lemma 4.9 from \cite{HB} below.  

\begin{lem}\label{lem:boundarycounting}
Let $C_i \subset \mathbb{R}^n$ be disjoint hypercubes with parallel edges of length $S_0$, and contained in a ball of radius $R$, centered on the origin. Let $F$ be a real cubic form in $n$ variables, and let $F_0$ be a real constant. Suppose that each hypercube contains a point $x$ for which $F(x)=F_0+O(R^2S_0)$ and $|\nabla F(x)|\gg R^2$ where $\nabla F$ denotes the gradient of $F$. Then the number of hypercubes $C_i$ contained in any ball of radius $R_0$ is $\ll_F 1+(R_0/S_0)^{n-1}$.
\end{lem}

In the sum \eqref{eqn:tildeS}, $D$ is fixed, and by the intermediate value theorem and since $\Cf_2(\C_1, \C_2)$ holds, then there exists $\hat \beta_1\in \C_1$ and $\hat \beta_2 \in \C_2$ such that one of the following holds for either $i=1$ or $i=2$:
\begin{align}\label{eqn:boundaryeqns}
L_2(\alpha)\cdot \hat \beta_i &= Y \textup{ or } Y(1+\eta) \notag \\
L_3(\alpha)\cdot \hat \beta_i &= X \textup{ or } X(1+\eta) \notag \\
N(\beta_i) &= V \textup{ or } 2V, \notag \\
\beta_i^3 &= N(\beta)^{1/3} \epsilon_0^{-1/2} \textup{ or } N(\beta)^{1/3} \epsilon_0^{1/2}
\end{align}where 
$$\alpha = \frac{\gamma_1 \times \gamma_2}{D}
$$and $\epsilon_0 = 1 + \sqrt[3]{2} + \sqrt[3]{4}$ is the fundamental unit of $K$.  Note here that we do not ask for such $\beta_i$ to have integer components; the function $N$ is a priori defined on $\beta$ with rational coordinates and may be uniquely extended to $\beta$ with real coordinates by continuity. 

Then since $D$ is fixed, each of the equations in \eqref{eqn:boundaryeqns} can be expressed in the form
\begin{equation}
F(\hat \beta_1, \hat \beta_2) = H,
\end{equation}where $F$ is homogeneous of degree $3$ in the components of $\hat \beta_i$, and $H$ is constant.

Let $R = V^{1/3}$.  We observe that the polynomials $F$ are non-singular in the region we consider and in fact satisfy $|\nabla F| \gg R^2$ in that region.  This is a calculation, an example of which has been done on pg. 74 in \cite{HB}.  We will apply Lemma \ref{lem:boundarycounting} to hypercubes of the form $\C_1 \times \C_2 \in \mathbb{R}^6$ with $R = V^{1/3}$.

To do this, we will need to cover the relevant region for $(\hat \beta_1, \hat \beta_2)$ where the region for $\hat{\beta_1}$ has dimensions bounded by $\frac{\N}{V^{1/3}}, \frac{\N}{V^{1/3}}$, and $V^{1/3}$ and for fixed $\hat{\beta_1}$, the region for $\hat{\beta_2}$ has dimensions bounded by $V^{1/3}, \frac{YV^{1/3}}{X},$ and $\frac{D_0Y}{\N}$.  Note that the volume of this region is
$$\mathfrak V := \frac{\N}{V^{1/3}}\frac{\N}{V^{1/3}} V^{1/3}V^{1/3}\frac{YV^{1/3}}{X}\frac{D_0Y}{\N} = \frac{Y^2 D_0 \N V^{1/3}}{X} .
$$

Now we set 
$$R_0 = \frac{D_0Y}{\N} = S_0 (\log X)^{C_1}.
$$Note that 
\begin{equation}\label{eqn:R0bdd}
R_0 \ll \min\left(V^{1/3}, \frac{\N}{V^{1/3}}, \frac{YV^{1/3}}{X},\frac{D_0Y}{\N}\right).
\end{equation}
Indeed, this follows directly from the bounds $\N \gg \bfrac{Y}{X}^{1/2} V^{2/3} (\log X)^{-B}$ for some constant $B$ and $D_0 Y \ll \bfrac{Y}{X}^2 V$.

By Lemma \ref{lem:boundarycounting}, the number of pairs of cubes $\C_1, \C_2$ which satisfies one of \eqref{eqn:boundaryeqns} inside a ball of radius $R_0$ is $\ll \bfrac{R_0}{S_0}^5$.  Thus, the contribution of these cubes to the right hand side of \eqref{eqn:prop6quantity} is
\begin{align*}
\ll \bfrac{R_0}{S_0}^5 \frac{S_0^6}{D^2}  (\log X)^{c},
\end{align*}for some absolute constant $c$.  

Moreover, the number of balls required to cover our region is $\ll \mathfrak V R_0^{-6}$.  To see this, simply construct the centers of these balls by first covering the region for $\hat{\beta_1}$ by $\ll \frac{\N}{V^{1/3}}\frac{\N}{V^{1/3}} V^{1/3} R_0^{-3}$ balls of radius $R_0/10$.  For one of these fixed balls $\B_1$ with center $\beta_{1, 0}$, the region for $\hat{\beta_2}$ still has dimensions bounded by $V^{1/3} + O(R_0), \frac{YV^{1/3}}{X}+O(R_0),$ and $\frac{D_0Y}{\N}+O(R_0)$.  Here, the $O(R_0)$ accounts for allowing $\hat{\beta_1}$ to vary within $\B_1$, but is negligible in the sense that by \eqref{eqn:R0bdd}, this region for $\hat{\beta_2}$ has dimensions $\ll V^{1/3}, \frac{YV^{1/3}}{X}$ and $\frac{D_0Y}{\N}$.  Thus we may cover the region for $\hat{\beta_2}$ by $\ll V^{1/3} \frac{YV^{1/3}}{X} \frac{D_0Y}{\N} R_0^{-3}$ balls $\B_2$ with radius $R_0/10$.  We now cover the region for $(\hat{\beta_1}, \hat{\beta_2})$ by balls of radius $R_0$ with centers $(\beta_{1, 0}, \beta_{2, 0})$ where $\beta_{1, 0}$ is the center of one of the ball $\B_1$ and $\beta_{2, 0}$ is one of the centers of the balls $\B_2$ (which is allowed to depend on $\B_1$).  

Thus, the right hand side of \eqref{eqn:prop6quantity} is bounded by
\begin{align}
\sum_{D_0 < D \le 2D_0} \mathfrak V R_0^{-6} \bfrac{R_0}{S_0}^5 \frac{S_0^6}{D^2} (\log X)^{c}
&= \frac{S_0}{R_0} \frac{\mathfrak V}{R_0^6} \frac{R_0^6}{S_0^6}\sum_{D_0< D\le 2D_0} \frac{S_0^6}{D^2} (\log X)^c \\
&\leq \frac{S_0}{R_0} \frac{\mathfrak V}{D_0} (\log X)^c\\
&\ll \frac{1}{(\log X)^{C_1-c}} \frac{\mathfrak V}{D_0}\\
&\ll \frac{1}{(\log X)^{C_1-c}} \frac{Y^2 V}{X}
\end{align}
since $\N \ll V^{2/3}$.  

\section{Proof of Proposition \ref{prop:Cf1cubes}}\label{sec:Cf1cubes}
Recall that we wish to  prove that for any $C>0$,
\begin{equation}\label{eqn:finalrequiredbdd}
\sum_{\substack{ \C_1 \neq \C_2\\ \Cf_1(C_1, C_2)}} S(\C_1, \C_2; D_0) \ll_C \frac{VY^2}{X(\log X)^C}.
\end{equation}
We write
\begin{align*}
S(\C_1, C_2; D_0) &= \sum_{D_0 < D \leq 2D_0} \sum_{\substack{\beta_1 \in \C_1, \beta_2 \in \C_2\\ \Delta(\beta_1, \beta_2) = D}}f_{\beta_1}f_{\beta_2}\\
&= \sum_{d} \mu(d) \sum_{D_0 < D \leq 2D_0} \sum_{\substack{\beta_1 \in \C_1, \beta_2 \in \C_2\\ dD| \Delta(\beta_1, \beta_2)}}f_{\beta_1}f_{\beta_2}\\
&= \sum_{d< d_0} \mu(d) \sum_{D_0 < D \leq 2D_0} \sum_{\substack{\beta_1 \in \C_1, \beta_2 \in \C_2\\ dD| \Delta(\beta_1, \beta_2)}}f_{\beta_1}f_{\beta_2} + O\left(\sum_{d\geq d_0} \sum_{D_0 < D \leq 2D_0} \frac{S_0^6}{(dD)^2} (\log X)^{c(A)} \right),
\end{align*}by an application of Lemma \ref{lem:cubebdd}.  In applying Lemma \ref{lem:cubebdd}, we may introduce the condition $dD_0 \ll S_0^A$ for some $A$ because $dD | \Delta(\beta_1, \beta_2)$ and by the lower bound for $S_0$ in Lemma \ref{lem:S0lowerbdd}. The latter term contributes a total
\begin{align*}
\ll \sum_{\C_1, \C_2} \frac{S_0^6(\log X)^{c(A)}}{d_0D_0}
&\ll \frac{(\log X)^{c(A)}}{d_0}\frac{\N^2}{V^{1/3}} \frac{X}{V^{1/3}} S_0 (\log X)^{C_1}\\
&\ll D_0 XY/d_0\\
&\ll \frac{VY^2(\log X)^{c(A)}}{Xd_0},
\end{align*}which suffices for Proposition \ref{prop:Cf1cubes} provided we pick 
\begin{equation}\label{eqn:d0}
d_0 = (\log X)^{C+c(A)}.
\end{equation}  For convenience, let
$$\T_6 = \T_6(\C_1, \C_2) := \sum_{d <d_0} \left|\sum_{D_0 < D \leq 2D_0} \sum_{\substack{\hat \beta_1 \in \C_1, \hat \beta_2 \in \C_2\\ dD| \Delta(\beta_1, \beta_2)}}f_{\beta_1}f_{\beta_2}\right|.$$

Since $\hat \beta_1$ and $\hat \beta_2$ are primitive, the condition $dD| \Delta(\beta_1, \beta_2)$ is equivalent to 
\begin{equation}
\hat \beta_1 \equiv \lambda \hat \beta_2 \bmod dD,
\end{equation}for some $\lambda \bmod dD$.  For $\bf a \in \mathbb{R}^3$, we introduce the exponential sum
\begin{equation}
\fS(\mathbf{a}) = \fS(\mathbf a, \C) = \sum_{\hat \beta \in \C} f_\beta e(\bf a \cdot \hat \beta),
\end{equation}where, as usual, $e(x) = e^{2\pi i x}$.

Then, writing as usual $\sumstar_{\lambda \bmod D}$ for a sum over reduced residues $\lambda$ modulo $D$,
\begin{align*}
\T_6 
&\le \sum_{d<d_0} \sum_{D_0 < D \leq 2D_0} \frac{1}{(dD)^3}\left|\sumstar_{\lambda \bmod dD} \sum_{\mathbf{a} \bmod dD} \fS((dD)^{-1} \lambda \mathbf{a}, \C_1) \overline{\fS((dD)^{-1} \lambda \mathbf{a}, \C_2)} \right|\\
&\le \sum_{d<d_0} \sum_{D_0 < D \leq 2D_0} \frac{1}{(dD)^3}\sumstar_{\lambda \bmod dD} \sum_{\mathbf{a} \bmod dD} |\fS((dD)^{-1} \lambda \mathbf{a}, \C)|^2,
\end{align*}by Cauchy-Schwarz, where $\C$ is either $\C_1$ or $\C_2$.  By a change of variables, we see that the sum
$$\sum_{\mathbf{a} \bmod dD} |\fS((dD)^{-1} \lambda \mathbf{a}, \C)|^2 = \sum_{\mathbf{a} \bmod dD} |\fS((dD)^{-1} \mathbf{a}, \C)|^2
$$is independent of $\lambda$, and so
\begin{equation}
\T_6 \leq \sum_{d<d_0} \sum_{D_0 < D \leq 2D_0} \frac{1}{(dD)^2} \sum_{\mathbf{a} \bmod dD} |\fS((dD)^{-1} \mathbf{a}, \C)|^2.
\end{equation}
  We intend to apply a large sieve bound for this type of sum; to prepare for this, we write $\frac{\mathbf{a}}{dD} = \frac{\mathbf{b}}{q}$ where the latter is in lowest terms.  Each vector $\frac{\mathbf{b}}{q}$ occurs with weight bounded by
\begin{align*}
\sum_{D_0 < D \le 2D_0} \sum_{\substack{d<d_0 \\ q|dD}} \frac{1}{(dD)^2}  \ll \sum_{\substack{D_0 < v \\ q|v}} \frac{\tau(v)}{v^2} \ll \frac{\tau(q)}{q} \frac{\log X}{D_0}.
\end{align*}
We then have that
\begin{align*}
\T_6 \ll \frac{\log X}{D_0} \sum_{q \ll D_0 d_0} \frac{\tau(q)}{q} \sumstar_{\mathbf{b} \bmod q} |\fS(\mathbf{b}/q)|^2,
\end{align*}where $\sumstar_{\mathbf{b}\bmod q}$ denotes a sum over vectors $\mathbf{b} = (b_1, b_2, b_3)$ such that the greatest common divisor of $b_1, b_2, b_3$ and $q$ is $1$.
We now quote the following large sieve bound from Lemma 13.1 of \cite{HB}.

\begin{lem}\label{lem:largesieveexpbdd}
With notation as above, and for $\C$ a cube of side $S_0$,
\begin{equation}
\sum_{Q < q \le 2Q} \sumstar_{\mathbf{b}\bmod q} |\fS(\mathbf{b}/q)|^2 \ll (S_0^3 + Q^2S_0^2 + Q^4) \sum_{\hat \beta \in \C} |f_\beta|^2.
\end{equation}
\end{lem}
We write
\begin{equation}
\T_6 \ll \T_7 + \T_8,
\end{equation}where
\begin{equation}
\T_7 := \frac{\log X}{D_0} \sum_{q \le  Q_0} \frac{\tau(q)}{q} \sumstar_{\mathbf{b} \bmod q} |\fS(\mathbf{b}/q)|^2,
\end{equation}and
\begin{equation}
\T_8 := \frac{\log X}{D_0} \sum_{Q_0<q\ll D_0d_0} \frac{\tau(q)}{q} \sumstar_{\mathbf{b} \bmod q} |\fS(\mathbf{b}/q)|^2.
\end{equation}In the above, we set $Q_0 = (\log X)^R$ for some large parameter $R>0$ to be chosen later.  It suffices to prove bounds for $\T_7$ and $\T_8$.  We start with the following bound for $\T_7$.

\begin{lem}\label{lem:T7}
We have
\begin{equation}
\T_7 \ll_C \frac{S_0^6}{D_0 \log^{C} X},
\end{equation}for any $C>0$.
\end{lem}
The proof of Lemma \ref{lem:T7} follows from the following Lemma.
\begin{lem}\label{lem:SWbdd}
With notation as above, $S_0 \ge V^{1/3 - \frac{5}{36} + \epsilon}$ and for $q \le (\log X)^R$, and any $\mathbf{c} \bmod q$, 
$$\sum_{\substack{\hat \beta \in \C\\ \hat \beta \equiv \mathbf{c} \bmod q}}f_\beta \ll \frac{S_0^3}{(\log X)^{C}},$$
for any $C>0$.
\end{lem}
We delay the proof of Lemma \ref{lem:SWbdd} until the next Section.  For now, let us verify Lemma \ref{lem:T7} assuming Lemma \ref{lem:SWbdd}.
\newline
\paragraph{{\it Proof of Lemma \ref{lem:T7}}}
Note that the condition $S_0 \ge V^{1/3 - 5/36 + \epsilon}$ is satisfied by our choice of $S_0$, by Lemma  \ref{lem:S0lowerbdd}.
We write
\begin{align*}
\fS(\mathbf{b}/q) 
&= \sum_{\mathbf{c} \bmod q} e(\mathbf{b}\cdot \mathbf{c}/q) \sum_{\substack{\hat \beta \in \C\\ \hat \beta \equiv \mathbf{c} \bmod q}}f_\beta\\
&\ll \sum_{\mathbf{c} \bmod {q}} \left| \sum_{\substack{\hat \beta \in \C\\ \hat \beta \equiv \mathbf{c} \bmod q}}f_\beta \right|
\\
&\ll \frac{q^3S_0^3}{\log^{C} X},
\end{align*}by Lemma \ref{lem:SWbdd}.
Thus
\begin{align*}
\T_7 \ll \frac{S_0^6}{D_0 \log^{2C-1}X} Q_0^4 = \frac{S_0^6}{D_0 \log^{2C-4R - 1}X},
\end{align*}which suffices for the Lemma, for $C$ sufficiently large.
\hfill$\square$

\begin{lem}\label{lem:T8}
There is some fixed $c>0$ such that
$$\T_8 \ll \frac{S_0^6}{D_0 (\log X)^{R-c}}.
$$
\end{lem}
\begin{proof}
Using the bound $\tau(q) \ll \exp\bfrac{c\log q}{\log \log q}$ for some $c>0$, we have by Lemma \ref{lem:largesieveexpbdd} that\begin{align*}
\sum_{Q<q\le 2Q} \frac{\tau(q)}{q} \sumstar_{b \bmod q} |\fS\bfrac{\mathbf b}{q}|^2
&\ll \frac{\exp\bfrac{c\log Q}{\log\log Q}}{Q} (S_0^3 + Q^2 S_0^2 + Q^4) \sum_{\hat \beta \in \C} |f_\beta|^2.
\end{align*}
We perform a dyadic summation over $Q_0 < Q \ll D_0d_0$ to see that
\begin{align*}
\T_8 &\ll \frac{\log X}{D_0} \left( \frac{S_0^3}{Q_0} \exp\bfrac{c\log Q_0}{\log\log Q_0} + S_0^2 (D_0d_0)\exp\bfrac{c\log (D_0d_0)}{\log\log (D_0d_0)} \right.\\
&+ \left.(D_0d_0)^3 \exp\bfrac{c\log (d_0D_0)}{\log\log (d_0D_0)}\right) \sum_{\hat \beta \in \C}|f_\beta|^2.
\end{align*}
The first term is bounded by
\begin{equation}\label{eqn:largefirstterm}
\frac{S_0^3}{D_0 (\log X)^{R-2}}\sum_{\hat \beta \in \C}|f_\beta|^2,
\end{equation}while the second is bounded by
\begin{equation}\label{eqn:largesecondterm}
S_0^2 d_0 \log X \exp \bfrac{c\log(d_0D_0)}{\log \log (d_0D_0)}\sum_{\hat \beta \in \C}|f_\beta|^2.
\end{equation}Recalling that $d_0$ is a power of $\log X$ from \eqref{eqn:d0} and $S_0/D_0 \gg X^\epsilon$ by Lemma \ref{lem:S0lowerbdd}, so that the quantity in \eqref{eqn:largesecondterm} is bounded by the quantity in \eqref{eqn:largefirstterm}.  Similarly, $S_0 \gg D_0d_0$, so the third term is bounded by \eqref{eqn:largefirstterm} as well.  Then then Lemma follows from 
$$\sum_{\hat \beta \in \C}|f_\beta|^2 \ll S_0^3 (\log S_0)^{c},
$$for some $c>0$ which is an immediate consequence of Lemma \ref{lem:cubedivisorsum}.
\end{proof}

We now sum over cubes $\C_1, \C_2$ satisfying $\mathfrak{C}_1 (\C_1, \C_2)$.  Recall that the condition  $\mathfrak{C}_1 (\C_1, \C_2)$ implies that $\C_1 \times \C_2$ is in a region with dimensions $\frac{\N}{V^{1/3}}, \frac{\N}{V^{1/3}}, V^{1/3}, V^{1/3}, \frac{YV^{1/3}}{X},$ and $\frac{D_0Y}{\N}$ (recall the discussion at the beginning of \S \ref{sec:reductiontocubes}) and recall that the volume of this region is
$$\mathfrak V = \frac{\N}{V^{1/3}}\frac{\N}{V^{1/3}} V^{1/3}V^{1/3}\frac{YV^{1/3}}{X}\frac{D_0Y}{\N} = \frac{Y^2 D_0 \N V^{1/3}}{X},
$$so that the number of pairs of cubes $\C_1, \C_2$ required to cover the region is bounded by $\frac{\mathfrak V}{S_0^6}$.

Now let $J = \min(C_1, R-c)$, so that by Lemmas \ref{lem:T7} and \ref{lem:T8}, we have 
\begin{align*}
\sum_{\substack{\C_1, \C_2\\ \Cf_1(C_1, C_2)}}\T_7 + \T_8 
&\ll \sum_{\substack{\C_1, \C_2\\ \Cf_1(C_1, C_2)}} \frac{S_0^6}{D_0 (\log X)^{J}}\\
&\ll\frac{1}{D_0 (\log X)^{J}} \frac{Y^2 D_0 \N V^{1/3}}{X}\\
&\ll \frac{Y^2V}{X(\log X)^{J}},
\end{align*}upon recalling that $\N \ll V^{2/3}$.  The bound above suffices for \eqref{eqn:finalrequiredbdd} upon picking $R = C+c$ and $C_1 = C$, so that $J \ge C$.

\section{Proof of Lemma \ref{lem:SWbdd}}\label{sec:reductiontogrossencharacters}
Recall that for some fixed $n \geq 0$ and $\mathbf{m} = (m_1,...m_{n+1}) \in \mathbb{N}^n$, we introduced intervals $J(m_i) = [X^{m_i \xi}, X^{(m_i+1)\xi}]$, and set
\begin{equation}\label{eqn:hbeta}
h_S = \prod_{i=1}^{n+1} \frac{\Lambda_K(T_i) \frak{W}_i(N(T_i))}{m_i \xi \log X},
\end{equation}where $S = T_1...T_{n+1}$ for integral ideals $T_i$, $0\le \frak{W}_i(x)\le 1$ is a smooth function which is $1$ on $J(m_i) = [X^{m_i \xi}, X^{(m_i+1)\xi}]$, is supported on $[X^{m_i \xi}(1-\mathcal{\iota}), X^{(m_i+1)\xi}(1+\mathfrak \iota)]$, where

\begin{equation}\label{eqn:iota}
\iota = \exp(- (\log X)^\epsilon).
\end{equation}  In the above, $m_1>m_2>...>m_{n+1} \ge \frac{\delta}{\xi}$ where we further recall from \eqref{eqn:xi} that 
\begin{equation}
\xi = \frac{1}{\log \log X},
\end{equation}and \eqref{eqn:delta} that
\begin{equation}
\delta \asymp 1.
\end{equation}

Further, to extract the main term from $h_S$, we introduced the coefficients
$$e_S = \frac{w'(N(S))}{\prod_{i=1}^{n+1} m_i \xi \log X} \sum_{J|S: N(J) < L} \mu(J) \log \frac{L}{N(J)}, 
$$where 
\begin{equation}\label{eqn:L}
L = X^{\xi}
\end{equation} and 
$$w(t) = \int_{\substack{x\in \mathbb{R}^{n+1}\\ \prod x_i \le t }} 1\; dx_1...dx_{n+1}.
$$

Recall $f_S = h_S - e_S$.  For all of this section, $\beta$ shall always denote an algebraic integer in $\cO_K$.  As before, we write $h_{(\beta)} = h_\beta$, $e_{(\beta)} = e_\beta$ and $f_{(\beta)} = f_\beta$.  Our definitions were motivated by the expectation that $h_{\beta}$ and $e_\beta$ behaved similarly in arithmetic progressions, and this section is devoted to verifying this for small moduli.

First, we quote Lemma 8.1 from \cite{HB} regarding $e_\beta$.  It may help the reader to recall that morally $e_\beta$ is a simple model for a product of prime ideals given by $h_\beta$ and is thus generally easier to understand, including when $\hat{\beta}$ is restricted to the small cube $\C$. 
\begin{lem}\label{lem:eb}
Let $\C \subset \mathbb{R}^3$ be a cube of side $S_0\geq L^2$ and edges parallel to the coordinate axes.  Define 
$$N(x, y, z) = x^3 + 2y^3 + 4z^3 - 6xyz
$$and
$$\I  =\int_\C w'(N(\mathbf{x})) dx dy dz.
$$For any positive integer $q\le L^{1/6}$ and any integer $\alpha \in \mathbb{Z}[\sqrt[3]{2}]$ we have
$$\sum_{\substack{\beta \equiv \alpha \bmod q\\ \hat \beta \in \C}} e_{\beta} = \frac{1}{\gamma_0 M} (\xi \log X)^{-n-1} \I \frac{\epsilon(\alpha, q)}{\phi_K(q)} + O\left(S_0^3 M^{-1} \tau(q)^c \exp(-c\sqrt{\log L}) \right),
$$where $\epsilon(\alpha, q) = 1$ if $(\alpha, q) = 1$ and $\epsilon(\alpha, q) =0$ otherwise.  Moreover, we have defined
$$M = \prod_{i=1}^{n+1} m_i,
$$and we have written $\phi_K$ for the Euler function over the field $K$.
\end{lem}
Our definition of $e_\beta$ is the same as that appearing in \cite{HB}, save that our $n \ll \frac{1}{\delta} \asymp 1$, which actually makes things slightly simpler.  We now prove the following adaptation of Lemma 9.1 from \cite{HB}.
\begin{lem}\label{lem:dbcommonfactor}
Let $\C$ be as in Lemma \ref{lem:eb} and $q \le (\log X)^R$ be as in Lemma \ref{lem:SWbdd}.  For $\alpha \in \mathbb{Z}[\sqrt[3]{2}]$, we have
$$\sum_{\substack{\beta \equiv \alpha \bmod q\\ \hat \beta \in \C}} h_\beta \ll S_0^3 X^{-\delta/2 +\epsilon} 
$$whenever $\alpha$ and $q$ have a non-trivial common factor.
\end{lem}
\begin{proof}
Suppose $\alpha$ and $q$ have a common factor $q_0$ with $P_0 = (q_0)$ a prime ideal.  Then $q_0|\beta$ whenever $\beta \equiv \alpha \bmod q$.  Recall that that $h_\beta$ is supported on $\beta$ with
$$(\beta) = T_1...T_{n+1}
$$ where $T_i$ is a power of a prime ideal for all $1\le i\le n+1$, so $P_0|T_i$ for some $1\le i\le n+1$.  Moreover, 
\begin{equation}
N(P_0) \le N(q) \leq (\log X)^{3R}, 
\end{equation}while $N(T_i) \ge X^{\delta}$, so for $X$ sufficiently large, we have that
\begin{equation}\label{eqn:P0bdd}
N(P_0)^{100} < X^{\delta}\le N(T_i),
\end{equation}
so that $P_0|T_i$ implies that 
\begin{equation}\label{eqn:TiP0}
T_i = P_0^{k}
\end{equation}
for some $k\ge 100$.

We recall that
\begin{align*}
S_0 \gg V^{1/3 - 5/36 +\epsilon} \geq X^{1/6+1/36},
\end{align*}where the first statement is from the conditions in Lemma \ref{lem:SWbdd}, and the second comes from recalling $V \gg X^{1+\tau}$.  Recalling $\delta \le 1/6$ from the definition of $\delta$ in \eqref{eqn:delta}, we have that for sufficiently large $X$, 
$$X^\delta \leq S_0.
$$

We now let $k_0(P_0) = k_0 \leq k$ be such that $X^{\delta/2} \leq N(P_0)^{k_0} \leq X^\delta \leq S_0$, possible from \eqref{eqn:P0bdd} and \eqref{eqn:TiP0}.  Note that $P_0^{k_0}|\beta$ and
\begin{align*}
\sum_{\substack {\hat{\beta} \in \C \\ P_0^{k_0}|\beta}} 1 \ll \frac{S_0^3}{N(P_0)^{k_0}}
\end{align*}since $N(P_0^{k_0}) \leq S_0$.  Here, we have used that, in general, for an ideal $I$, a cube of side $N(I)$ contains $O(N(I)^2)$ values of $\hat{\beta}$ with $I|\beta$.  

Then since $h_\beta \ll 1$,
\begin{align*}
\sum_{\substack{\beta \equiv \alpha \bmod q \\ \hat{\beta} \in \C}} h_\beta  
&\ll \sum_{P_0|q} \sum_{\substack {\hat{\beta} \in \C \\ P_0^{k_0}|\beta}} 1 \\
&\ll \sum_{P_0|q} \frac{S_0^3}{N(P_0)^{k_0}} \\
&\ll \frac{S_0^3}{X^{\delta/2 - \epsilon}},
\end{align*}using the crude bound $\sum_{P_0|q} \ll \log q \ll X^\epsilon$.
\end{proof}

\begin{rem}
In the proof of the Lemma above, it was important that we were able to produce a factor $P_0^{k_0}|\beta$ with $N(P_0^{k_0}) \leq S_0$.  The reader may check that this line of reasoning is not sufficient to estimate
$$\sum_{\hat{\beta} \in \C} (h_\beta - d_\beta)
$$as that would naturally lead to problems involving estimating quantities like  
$$\sum_{\substack{\hat{\beta} \in \C\\ I|\beta}} 1
$$ for some integral ideal $I$ with $N(I)$ significantly larger than the side of $\C$.  We chose to avoid this entirely by performing the replacement earlier in Lemma \ref{lem:hbetadbetareplace} well before the reduction to small cubes.
\end{rem}

Lemma \ref{lem:SWbdd} follows from the previous two Lemmas and the Lemma below.
\begin{lem} \label{lem:db}
Set notation as in Lemma \ref{lem:eb} and assume that $S_0 \ge V^{1/3 - 5/36 + \epsilon}.$  Then for any positive integer $q\le (\log L)^A$ and any algebraic integer $\alpha \in \mathbb{Z}[\sqrt[3]{2}]$ coprime to $q$ we have that there exists some constant $c>0$ such that
$$\sum_{\substack{\beta \equiv \alpha \bmod q\\ \hat \beta \in \C}} d_{\beta} = \frac{1}{\gamma_0 M\phi_K(q)} (\xi \log X)^{-n-1} \I + O\left(S_0^3 \exp(-c(\log V)^{1/3 - \epsilon}) \right),
$$where $M = \prod_{i=1}^{n+1} m_i$.
\end{lem}
Unfortunately Lemma \ref{lem:db} does not appear in the current literature.  The reader may compare this to Lemma 9.2 in \cite{HB}, which is a version of Lemma \ref{lem:db} when $S_0$ is not much smaller than $V^{1/3}$.  The version we require is analogous to asking for a Siegel Walfisz theorem for primes in short intervals of the form $(x, x+x^{7/12 +\epsilon})$.  

We now proceed to prove this result.  The condition $\beta \equiv \alpha \bmod q$ may be picked out using multiplicative characters for $\mathbb{Z}[\sqrt[3]{2}]$ mod $q$, so that
\begin{equation}
\sum_{\substack{\beta \equiv \alpha \bmod q\\ \hat \beta \in \C}} h_{\beta} =\frac{1}{\phi_K(q)} \sum_{\chi \bmod q} \overline{\chi}(a) \sum_{\hat \beta \in \C} h_\beta \chi(\beta).
\end{equation}
The condition $\hat \beta \in \C$ may be picked out using Hecke Grossencharacters.  To be specific, recall that $\epsilon_0$ is the fundamental unit and let $\epsilon_0'\neq \epsilon_0$ be a Galois conjugate.  For fixed $\chi$ a character of $\mathbb{Z}(\sqrt[3]{2}) \bmod q$,  we define $s, t, u, v$ by
\begin{equation}
\chi(-1) = (-1)^s, \chi(\epsilon_0) =e^{it}, \frac{\epsilon_0'}{|\epsilon_0'|} = e^{iu}, \log \epsilon_0 = v,
\end{equation} where $s = 0$ or $1$, $0\le t, u<2\pi$ and $v\in \mathbb{R}$.  Now, we define for $\beta \in \cO_K$
\begin{align*}
\nu_0(\beta) &= \chi(\beta) \bfrac{\beta}{|\beta|}^s \exp(-itv^{-1}\log |\beta|)\\
\nu_1(\beta) &= \frac{\beta\beta'}{|\beta\beta'|} \exp(-iuv^{-1}\log |\beta|)\\
\nu_2(\beta) &= \exp(-2\pi i v^{-1} \log |\beta|),
\end{align*}where for $\hat \beta =(a, b, c)$, $\hat \beta' = (a, b\omega , c\omega^2)$, for $\omega=1/2(-1+\sqrt{-3})$.  These characters are completely multiplicative and are the same on associates.  Thus, even though we defined $\nu_i$ as characters on $\cO_K$ for $0\le i\le 2$, they are also well defined as characters on ideals via $\nu_i((\beta)) = \nu_i(\beta)$, where recall $(\beta)$ is the ideal generated by $\beta$.

\begin{rem}\label{rem:characterchi}
Since $v_0$ is the same on associates, we think of $v_0$ as a Hecke Grossencharacter on ideals.  We now write 
$$v_0 = \theta v_0',
$$where $\theta$ is a Hecke Grossencharacter mod an ideal $\mathfrak q$ with $N(\mathfrak q) \asymp q$, and $v_0'$ is a Hecke Grossencharacter with conductor $1$.  Indeed, there exists a finite basis for torsion free Hecke Grossencharacters, and any Hecke Grossencharacter may be written as $\kappa v$ for $\kappa$ a Grossencharacter mod an integral ideal and $v$ torsion-free.  We refer the reader to \cite{He} for details on Hecke Grossencharacters.
\end{rem}

Vaguely, the values of $N(\beta)$, $\nu_1(\beta)$ and $\nu_2(\beta)$ determine $\beta$.  To be more precise, we follow the work of Heath-Brown to define for $\bx = (x_1, x_2, x_3) \in \mathbb{R}^3$, $\beta(\bx) = x_1+x_2\sqrt[3]{2} + x_3\sqrt[3]{4}$, $\beta'(\bx) = x_1+x_2 \omega\sqrt[3]{2} + x_3 \omega^2\sqrt[3]{4}$ and set
\begin{align*}
\nu_1(\bx) &= \frac{\beta(\bx)\beta'(\bx)}{|\beta(\bx)\beta'(\bx)|} \exp(-iuv^{-1} \log |\beta(\bx)|) \textup{ and} \\
\nu_2(\bx) &= \exp(-2\pi iv^{-1} \log |\beta(\bx)|).
\end{align*}

We may define an associate of $\bx$ as any $\bx' = \pm M^{n} \bx$ for any $n\in \mathbb{Z}$ where
\[ M = \left( \begin{array}{ccc}
1 & 2 & 2 \\
1 & 1 & 2 \\
1 & 1 & 1 \end{array} \right).\] 
This is analogous to the notion of associates on $K$; for instance, $\beta(M\bx) = \epsilon_0 \beta(\bx)$.  Further for simplicity, we write
\begin{equation}\label{eqn:Nx}
N(\bx) = N(\beta(\bx)).
\end{equation}

To fix ideas, let $h_{\frac 12}$ be a fixed smooth non-negative function on $\mathbb{R}/\mathbb{Z}$ with $h_{\frac 12}(-1/2) = h_{\frac 12}(1/2)  =0$ and $\int_0^1 h_{\frac 12}(t) dt = 1/2$.  For $\Delta <1/2$, define $h_\Delta$ on $\mathbb{R}/\mathbb{Z}$ by 
\begin{equation}\label{eqn:hDelta}
h_\Delta(t) = h_{\frac 12}\bfrac{t}{2\Delta}
\end{equation}
for $t \in [-\Delta, \Delta]$ and $h_\Delta(t) = 0$ for $t \in [1/2, 1/2) - [-\Delta, \Delta]$.

Then $h_\Delta(t)$ is a smooth non-negative function on $\mathbb{R}/\mathbb{Z}$ supported on $[-\Delta, \Delta]$.  That is, on the interval $[-1/2, 1/2)$, $h$ vanishes outside of $[-\Delta, \Delta]$, and $h$ is $1$-periodic.  Further, $ h^{(k)}(t) \ll_k \Delta^{-k}$ for all $k \geq 0$, and $\int_0^1 h(t) dt = \int_{-\Delta}^{\Delta} h(t) dt = 2\Delta \int_{-1/2}^{1/2} h_{\frac 12}(t) dt = \Delta$.  

The definition of $h_\Delta$ for $\Delta = 1/2$ is precisely our original $h_{\frac 12}$.  The function $h = h_\Delta$ depends on $\Delta$, but we will often suppress the dependence on $\Delta$ for convenience when we do not need to consider different values of $\Delta$.  We write the Fourier series of $h$ as
\begin{equation}
h(t) = \sum_{k} b(k) e(kt),
\end{equation}where partial integration yields as usual that 
\begin{equation}\label{eqn:bkbdd}
b(k) \ll_N \Delta \bfrac{1}{k\Delta}^N
\end{equation}for any $N \geq 0$, and in particular, $b(0) = \int_0^1 h(t) dt = \Delta$.  The implied constant here may depend on the choice of $h_{\frac 12}$, but we suppress the dependence since we consider $h_{\frac 12}$ fixed.

Define for integral ideals $S$
\begin{equation}
W(S, \bx, \Delta) = h\left(\arg\bfrac{v_1(S)}{v_1(\bx)}\right) h\left(\arg\bfrac{v_2(S)}{v_2(\bx)}\right).
\end{equation}We examine
\begin{equation}
\Sigma(\bx, \Delta) = \Sigma(\bx) = \sum_{N(\bx) < N(S) \le N(\bx) + \Delta V} h_S v_0(S) W(S, \bx, \Delta),
\end{equation}for fixed $\bx \in \C$ so that $N(\bx) \asymp V$.  Thus the sum above is over integral ideals $S$ with $N(S) \asymp V$.  In fact, the terms in $\Sigma$ are essentially restricted to $S = (\beta)$ for some $\beta$ where $\hat \beta$ is in a cube centered at $\bx$ of side $\Delta V$.  To be precise, we quote Lemma 9.5 of \cite{HB} below.
\begin{lem}\label{lem:restricttocube}
Let $V \ll N(S), N(\bx) \ll V$, and suppose that $W(S, \bx, \Delta) \neq 0$ and that $N(\bx) - \Delta V < N(S) \le N(\bx) + \Delta V$.  Then $S$ has a generator $\beta$ such that $\hat \beta$ satisfies 
$$\hat \beta = (1+O(\Delta)) \bx.
$$Moreover, for any associate $\beta'$ of $\beta$, $\bx$ has an associate $\bx'$ such that
$$\hat \beta' = (1+O(\Delta))\bx'.
$$
\end{lem}
This Lemma is proven in \cite{HB} for specific functions $h_{\frac 12}$ and $W$ and for the range $N(\bx) < N(\beta) \le N(\bx) + \Delta V$ rather than $N(\bx) - \Delta V < N(\beta) \le N(\bx) + \Delta V$, but the proof goes through without any changes for our situation also.  The sum $\Sigma(\bx)$ is already close to the sum we wish to understand, and indeed, an integration over $\bx \in \C$ essentially produces the sum in Lemma \ref{lem:db}.  \footnote{The sum in Lemma \ref{lem:db} has a sharp cutoff with the restriction $\hat{\beta} \in \C$, while the sum $\Sigma(\bx) = \Sigma(\bx, \Delta)$ may be considered to be a smoothed version.}  We refer the reader to \S9 of \cite{HB} for the details of this.  Here, we focus on the sum $\Sigma$ since this sum contains the essence of the problem - namely, it is over a small region.  

For the sequel, fix $\bx \in \C$ and note that $N(\bx) \asymp V$, while 
\begin{equation}\label{eqn:bxnorm}
\|\bx\| \asymp V^{1/3}
\end{equation}
the latter being inherited from \eqref{eqn:betanormbdd}.  Let
\begin{equation}\label{eqn:Delta0}
\Delta_0 = \exp( - c_0(\log L)^{1/2}),
\end{equation}for some sufficiently small constant $c_0>0$.  It is relatively easy to understand $\Sigma(\bx, \Delta_0)$, and indeed, Lemma 9.3 from the work of Heath-Brown implies that 
\begin{equation}\label{eqn:SigmaxDelta0}
\Sigma(\bx, \Delta_0) = \epsilon(\chi) m_0(\bx) \frac{\Delta_0^2}{M} (\xi\log X)^{-n-1} + O(V \exp(-c\sqrt{\log L})),
\end{equation}for some $c>0$, where $\epsilon(\chi) = 1$ or $0$ depending on whether $\chi$ is trivial or not, and $m_0(\bx) = \omega(N(\bx) + \Delta_0 V) - \omega(N(\bx))$.  The main term dominates the error term when we choose $c_0$ sufficiently small in \eqref{eqn:Delta0}.  In particular, we may demand  that $\Delta_0^{-3} \ll \exp\left( \frac c2 \sqrt{\log L}\right)$, and this is the only condition required on the constant $c_0$ from \eqref{eqn:Delta0}.  Thus, the error $V\exp(-c\sqrt{\log L})$ may be replaced by $V \Delta_0^3 \exp(-c/2\sqrt{\log L})$.

Actually, Heath-Brown proved \eqref{eqn:SigmaxDelta0} for $d_\beta$ in place of our $h_\beta$.  However, it is straightforward to replace one by the other with negligible error, since the region being summed over in $\Sigma(\bx, \Delta_0)$ is so wide.  The proof proceeds along the same lines, being an application of the work of Mitsui \cite{mitsui}, with partial summation handling the added smoothing from $\mathfrak W$.   Moreover, uniformity in $n$ for instance does not cause issues, since our $n\ll 1$.

Our main challenge is in understanding $\Sigma(\bx, \Delta)$ for smaller values of $\Delta$.  To be precise, we will prove the following Proposition.
\begin{prop}\label{prop:smallcube}
Fix notation as above.  Suppose that $V^{-5/36 + \epsilon} \le \Delta \le \Delta_0$ and $\bx \in \C$.  Then we have that
\begin{equation}
\Sigma(\bx, \Delta) = \Sigma(\bx) = \epsilon(\chi) m(\bx) \frac{\Delta^2}{M} (\xi \log X)^{-n-1} + O(\Delta^3V \exp(-c(\log V)^{1/3 - \epsilon})),
\end{equation}
for some $c>0$ and $m(\bx) = \omega(N(\bx) + \Delta V) - \omega(N(\bx))$, where $M = \prod_{i=1}^{n+1} m_i$.
\end{prop}

\begin{rem}
In the above, it is important that our modulus $q \leq (\log X)^R$.  As in Heath-Brown's work, the implied constant is ineffective due to possible exceptional zeros.  Recalling the definition of $h = h_\Delta$ as defined in \eqref{eqn:hDelta}, we note that the implied constant above may depend on $h_{\frac 12}$ but not on $\Delta$.
\end{rem}

Note that for $\bx \in \C$, 
$$N(\bx) \asymp V,$$
an estimate we shall use often in the sequel.  Roughly speaking, when interested in cubes of side $S_0$, it suffices to study $\Sigma(\bx, \Delta)$ for $\Delta = S_0/V^{1+\epsilon}$ and this is why we only consider $\Delta \ge V^{-5/36 + \epsilon}$ in Proposition \ref{prop:smallcube}.  

Proceeding to the proof of Proposition \ref{prop:smallcube}, we write
\begin{align}
\Sigma(\bx, \Delta) &= \sum_{j, k} b(j)b(k)v_1(\bx)^{-j}v_2(\bx)^{-k} \sum_{N(\bx)< N(S) \le N(\bx) + \Delta V} h_S v_0(S)v_1(S)^j v_2(S)^k\\
&= \sumsharp_{j, k} b(j)b(k)v_1(\bx)^{-j}v_2(\bx)^{-k} \Sigma_{j, k}    + O_C(V^{-C}),
\end{align}for any $C$, where $\sumsharp$ above indicates a sum $|j|, |k| \le T_0$ where
\begin{equation}\label{eqn:T_0}
V^\epsilon \ll T_0 := \frac{V^\epsilon}{\Delta} \ll V^{5/36 - \delta_0}
\end{equation}for some $\delta_0 >0 $ and where
\begin{align}
\Sigma_{j, k} &= \sum_{N(\bx)< N(S) \le N(\bx) + \Delta V} h_S v_0(S)v_1(S)^j v_2(S)^k.
\end{align}  Here, the error term bounds the contribution of those terms where $\max(j, k) > T_0$, where we have used the bound for $b(k)$ in \eqref{eqn:bkbdd}.  

Now we fix a smooth function $H$ which is identically $1$ on $[0, 1]$ and supported on $[-V^{-\epsilon}, 1+V^{-\epsilon}]$.  We further ask that $0\le H(t) \leq 1$ for all $t$.

Now define
\begin{equation}\label{eqn:Fdef}
F_\Delta(t) = F(t) = H \bfrac{t-N(\bx)}{\Delta V}.
\end{equation}

Note $F$ is a smooth function which is $1$ on $[N(\bx), N(\bx)+\Delta V]$, is always bounded by $1$, and is supported on $[N(\bx) - \Delta V^{1-\epsilon}, N(\bx)+\Delta V + \Delta V^{1-\epsilon}]$.  We further note that
\begin{equation}\label{eqn:Fkbdd}
F^{(k)} (x) \ll_k (V^{1-\epsilon}\Delta )^{-k},
\end{equation}
for all $k \geq 0$.  

We write
\begin{equation}
\Sigma_{j, k} = \sum_{S} h_S v_0(S)v_1(S)^j v_2(S)^k F(N(S)) + \Sigma'_{j, k} + \Sigma''_{j, k}.
\end{equation}Now, $\Sigma'$ and $\Sigma''$ are similar sums to $\Sigma$, but supported on intervals of length $\Delta V^{1-\epsilon}$.  More specifically, 
$$\Sigma'_{j, k} = \sum_{N(\bx) - \Delta V^{1-\epsilon} \le N(S) \le N(\bx)} d'_S,
$$where $d'_S = -F(N(S))h_S \ll 1$ and $\Sigma''$ is the rest.  We have the following bounds on $\Sigma'$ and $\Sigma ''$.  

\begin{lem}
With notation as above,
\begin{equation}
\Sigma', \Sigma'' \ll \Delta^3 V^{1-\epsilon}.
\end{equation}
\end{lem}
\begin{proof}
We bound the contribution of $\Sigma'$, the contribution of $\Sigma''$ being bounded similarly.  We have
\begin{align*}
\Sigma' &= \sumsharp_{j, k} b(j)b(k)v_1(\bx)^{-j}v_2(\bx)^{-k} \Sigma'_{j, k}\\
&=\sum_{j, k} b(j)b(k)v_1(\bx)^{-j}v_2(\bx)^{-k} \Sigma'_{j, k} + O(V^{-C})\\
&= \sum_{N(\bx) - \Delta V^{1-\epsilon} \le N(S) \le N(\bx)} d'_S v_0(S) \sum_{j, k} b(j)b(k)v_1(S)^j v_1(\bx)^{-j}v_2(S)^k v_2(\bx)^{-k} + O(V^{-C}).
\end{align*}  Collapsing the Fourier series, the first term above is
\begin{align*}
\sum_{N(\bx) - \Delta V^{1-\epsilon} \le N(S) \le N(\bx)} d'_S v_0(S) W(S, \bx, \Delta).
\end{align*}In the sum above, since $N(\bx) \asymp V$, we have that $N(S) \asymp V$.  We may thus apply Lemma \ref{lem:restricttocube} to see that the sum over $S$ must satisfy $S = (\beta)$ for some $\hat \beta$ restricted in a cube of length $\Delta V$ centered at $\bx$.  For convenience, let this cube be $\C_{\bx}$.  

Since $N(\bx) - \Delta V^{1-\epsilon} \le N(S) \le N(\bx)$ using the bound $d'_S \ll 1$, we have that
\begin{align*}
\Sigma' \ll \# \{\hat{\beta} \in \C_{\bx}: N(\beta) = N(\bx) + O(\Delta V^{1-\epsilon})\}.
\end{align*}
Writing $\bx = (x, y, z)$ and $N(\bx) = x^3 + y^3 + z^3 - 6xyz$, we have that
\begin{equation}
N(\bx) = \frac{x}{3} \frac{\partial N}{\partial x} +\frac{y}{3} \frac{\partial N}{\partial y}+\frac{z}{3} \frac{\partial N}{\partial z}.
\end{equation}
In particular, from \eqref{eqn:bxnorm}, $N(\bx) \asymp V$ and $\| \bx \| \asymp V^{1/3}$ so we have that $\nabla N \gg V^{2/3}$.  We now consider the region $\C_{\bx} - \bx$ which is contained in a ball of radius $\ll \Delta V^{1/3} := R_0$.  We apply Lemma \ref{lem:boundarycounting} with $F(u, v, w) = N(u+x, v+y, w+z)$, $F_0 = F(0, 0, 0)=  N(\bx)$ and $R = V^{1/3}$.  In so doing, we have used that $\nabla F \gg V^{2/3} = R^2$ on a ball of radius $R_0$ around the origin.  The number of cubes of side $\Delta V^{1/3 - \epsilon}$ containing some element $(u, v, w)$ satisfying $F(u, v, w) = N(\bx) + O(R^2 S_0)$ in a ball of radius $R_0$ is $\ll V^{2\epsilon}$.  Thus the total number of points to be counted is $\ll \Delta^3 V^{1-3\epsilon} V^{2\epsilon} = \Delta^3 V^{1-\epsilon}$, as desired.
\end{proof}

Now we let
$$\tilde{F}(s) = \int_0^\infty F(x) x^{s-1} dx
$$be the usual Mellin transform of $F(x)$.  Trivially, 
\begin{equation}\label{eqn:Fstrivialbdd}
|\tilde{F}(s)| \ll V^{\tRe s} \Delta,
\end{equation}
while integration by parts $k+1$ times along with the bounds \eqref{eqn:Fkbdd} imply that
\begin{equation}\label{eqn:tildeFbdd}
|\tilde{F}(s)| \ll V^{\tRe s + k} \Delta V^{1-\epsilon} (\Delta V^{1-\epsilon} |s|)^{-k-1},
\end{equation}
so that $\tilde{F}(s) \ll_C V^{-C}$ for all $|s| \ge T_0 := V^{2 \epsilon} \Delta^{-1}$ and for all $C>0$.  We will use the usual convention regarding $\epsilon$ and write
\begin{equation}\label{eqn:T0}
T_0 = V^{\epsilon} \Delta^{-1}
\end{equation} in the sequel.  Writing $F(x)$ as the Mellin inverse transform of $\tilde{F}(s)$, we see that
\begin{align*}
\sum_{S} h_S v_0(S)v_1(S)^j v_2(S)^k F(N(S))
= \frac{1}{2\pi i} \int_{(0)} \sum_{S} h_S v_0(S)v_1(S)^j v_2(S)^k (N(S))^{-s} \tilde{F}(s) ds
\end{align*}

Our bounds on $\tilde{F}(s)$ implies that we may truncate the integral at height $T_0$ with an error bounded by $V^{-C}$ for any $C>0$ so it suffices to examine
\begin{align}
\frac{1}{2\pi} \int_{-T_0}^{T_0} \sum_{S} h_S v_0(S)v_1(S)^j v_2(S)^k (N(S))^{-it} \tilde{F}(it)dt.
\end{align}
Now, if $v_0$ is trivial, then we expect a main term contribution from the case $j=k=0$, and $|t|<\tau_0$ for 
\begin{equation}\label{eqn:tau0def}
\tau_0 = \Delta_0^{-1/2}.
\end{equation}  Now, let
\begin{equation}\label{eqn:MDeltaF}
M(\Delta, F) := \frac{1}{2\pi} \int_{-\tau_0}^{\tau_0} \sum_{S} h_S v_0(S) (N(S))^{-it} \tilde{F}(it)dt.
\end{equation}  This is the contribution from the case $j = k = 0$ and $|t|<\tau_0$ for $\tau_0 = \Delta_0^{-1/2}$.

Recalling \eqref{eqn:Fdef}, we write for convenience $F_0 = F_{\Delta_0}$.  It turns out to be more convenient to compare $M(\Delta, F)$ against a quantity like $M(\Delta_0, F_0)$ rather than attempt a direct evaluation.  For brevity, let
\begin{align*}
M &= M(\Delta, F) \; \; \textup{ and}\\
M_0 &= M(\Delta_0, F_0).
\end{align*}

It now suffices to prove the following two Propositions.  
\begin{prop}\label{prop:Mcomparison}
There exists a constant $c>0$ such that 
\begin{equation}
\left| M - \frac{\Delta}{\Delta_0} M_0\right| \ll \Delta V \exp(-c\sqrt{\log L}),
\end{equation}
where the implied constant may only depend on the choices of the fixed functions $h_{\frac  12}$ and $H$ in \eqref{eqn:hDelta} and \eqref{eqn:Fdef} respectively.
\end{prop}
It may help the reader to recall that $\log L \gg (\log X)^{1-\epsilon}$ from \eqref{eqn:L}.
\begin{prop}\label{prop:errorsmallcube}
For any $\delta_0>0$, $V^\epsilon \ll T_0 \ll V^{5/36 - \delta_0}$ and $V^{-5/36+\epsilon} \le \Delta \le \Delta_0$, there exists $c>0$ such that 
\begin{equation}
\int_{-T_0}^{T_0} \sumc_{j, k} b(j)b(k) \sum_{S} h_S v_0(S)v_1(S)^j v_2(S)^k (N(S))^{-it} \tilde{F}(it)dt \ll \Delta^3 V \exp(-c(\log V)^{1/3 - \epsilon}),
\end{equation}where $\sumc$ denotes a sum over $j, k$ satisfying the conditions in $\sumsharp$ with the added condition that if $|t| < \tau_0$, then $(i, j) \neq (0, 0)$, and where recall $F = F_\Delta$ is defined as in \eqref{eqn:Fdef}.  The implied constant is ineffective and may depend only on the choices of the fixed functions $h_{\frac  12}$ and $H$ in \eqref{eqn:hDelta} and \eqref{eqn:Fdef} respectively, as well as the choice of $R$, in the restriction $q\le (\log X)^R$.
\end{prop}

Before proving these Propositions, we first explain how to verify Proposition \ref{prop:smallcube} assuming Propositions \ref{prop:Mcomparison} and \ref{prop:errorsmallcube}.  

\begin{proof}
We will compare our sum $\Sigma(\bx, \Delta)$ with the analogous sum $\Sigma_0 = \Sigma(\bx, \Delta_0)$.  For convenience, we also write 
$$h_0 = h_{\Delta_0}.$$

Recall that
\begin{equation}
h(t) = \sum_k b(k) e(kt),
\end{equation}and analogously write
\begin{equation}
h_0(t) = \sum_k b_0(k) e(kt).
\end{equation}Recall that
\begin{equation}\label{eqn:b0b00}
b_0(0) = \int_0^1 h_0(t) dt = \Delta_0.
\end{equation}
We let
\begin{equation}
\Sigma_0 =\Sigma(\bx, \Delta_0)= \sum_{j, k} b_0(j)b_0(k) v_1(\bx)^{-j} v_2(\bx)^{-k} \sum_{N(\bx) < N(S)\le N(\bx) +\Delta_0 V} h_S v_0(S) v_1(S)^j v_2(S)^k.
\end{equation}

By \eqref{eqn:SigmaxDelta0}, it suffices to show that
\begin{equation}\label{eqn:comparison1}
\left| \Sigma(\bx, \Delta) -  \frac{\Delta^3}{\Delta_0^3} \Sigma_0\right| \ll \Delta^3 V \exp(-c(\log V)^{1/3-\epsilon}).
\end{equation}
On the other hand, by Proposition \ref{prop:errorsmallcube}, we have that
$$\Sigma(\bx, \Delta) = b(0)^2M + \Delta^3 V \exp(-c(\log V)^{1/3 - \epsilon})
$$

All of our arguments above for $\Sigma(\bx, \Delta)$ are valid for $\Sigma_0 = \Sigma(\bx, \Delta_0)$ also.  In particular, the error terms may depend on the choices of $h_{\frac 12}$ in \eqref{eqn:hDelta} and $H$ in \eqref{eqn:Fdef}, but not on $\Delta$.

By Proposition \ref{prop:errorsmallcube} and recalling that $b(0) = \Delta$, $b_0(0) = \Delta_0$, we have
\begin{equation}
\left| \Sigma(\bx, \Delta) -  \frac{\Delta^3}{\Delta_0^3} \Sigma_0\right| \ll \left|b(0)^2M(\Delta) - b_0(0)^2 \frac{\Delta^3}{\Delta_0^3} M_0 \right| + \Delta^3 V \exp(-c(\log V)^{1/3 - \epsilon}). 
\end{equation}By \eqref{eqn:b0b00} and Proposition \ref{prop:Mcomparison}, 
$$\left|b(0)^2M(\Delta) - b_0(0)^2 \frac{\Delta^3}{\Delta_0^3} M_0 \right| \ll \Delta^2 \left|M(\Delta) - \frac{\Delta}{\Delta_0} M_0 \right| \ll \Delta^3 V \exp(-c\sqrt{\log L}).
$$

This proves \eqref{eqn:comparison1} and concludes the proof of Proposition \ref{prop:smallcube}.

\end{proof}

We will prove Proposition \ref{prop:Mcomparison} in the rest of this section, and delay the proof of Proposition \ref{prop:errorsmallcube} for the next section.  We start with the Lemma below.

\begin{lem}\label{lem:FhatF0hat}
For all $s$ with $\tRe s = 0$, we have that
\begin{equation}
\left|\tilde{F}(s) - \frac{\Delta}{\Delta_0} \tilde{F}_0(s)\right| \ll \Delta \Delta_0 (|s|+1).
\end{equation}
\end{lem}
\begin{proof}
Recalling the definition of $F$ from \eqref{eqn:Fdef}, by a change of variables, we have that
\begin{align*}
\tilde{F}(s) &= \int_0^\infty H\bfrac {t-N(\bx)}{\Delta V} t^{s-1} dt\\
&=\int_0^\infty H(u) \left( N(\bx) + \Delta  V u\right)^{s-1} du
\end{align*}and similarly for $\tilde{F_0}(s)$.
Thus,
\begin{equation}
\Delta_0 F(s) - \Delta F_0 = \Delta_0 \Delta V \int_0^\infty H(u) \left( \left( N(\bx) + \Delta  V u\right)^{s-1}-\left( N(\bx) + \Delta_0  V u\right)^{s-1} \right) du
\end{equation}

Further, by the mean value theorem, and for $u \ll \Delta V$,
\begin{equation}
\left( N  + \Delta Vu  \right)^{s-1} - \left( N(\bx) + \Delta_0  V u\right)^{s-1} = \left((\Delta - \Delta_0)V u\right) (s-1) \lambda^{s-2} \ll \Delta_0 \frac{|s|+1}{V}, 
\end{equation} for some $\lambda \in [N+\Delta V u, N+ \Delta_0V u]$ so that $\lambda^{-2} \ll \frac{1}{N^2} \asymp \frac{1}{V^2}$ for $N = N(\bx) \asymp V$.  Moreover $\int_0^\infty H(t) dt \asymp 1$, so
\begin{equation}
|\Delta \tilde{F}_0(s) - \Delta_0 \tilde{F}(s) | \ll \Delta_0^2 \Delta (|s|+1)
\end{equation}which suffices for the Lemma.
\end{proof}

By Lemma \ref{lem:FhatF0hat} above, 
\begin{equation}\label{eqn:integralsmallt}
\int_{-\tau_0}^{\tau_0} \sum_{N - V^{1-\epsilon} \le N(S) \le N+\Delta_0V + V^{1-\epsilon}} \left|h_S(\tilde{F}(it) -  \frac{\Delta}{\Delta_0}\tilde{F}_0(it))\right|dt.
\end{equation}
is bounded by 
$$\Delta \Delta_0 \tau_0^2 \sum_{N - \Delta_0V \le N(S) \le N+2\Delta_0 V} 1 \ll \Delta_0^2 \Delta \tau_0^2 V.
$$  Upon recalling that $\tau_0 = \Delta_0^{-1/2}$, we see that the right hand side above is $\Delta_0 \Delta V$ and this suffices for the proof of Proposition \ref{prop:Mcomparison} upon recalling $\Delta_0 = \exp(-c_0\sqrt{\log L})$ for some constant $c_0$.  


\section{Proof of Proposition \ref{prop:errorsmallcube}}\label{sec:errorsmallcube}
For convenience, we let
\begin{equation}\label{eqn:psi}
\psi(S) = \psi(j, k, t)(S) = v_0v_1^jv_2^k(S) N(S)^{-it}.
\end{equation}Further, for such a $\psi$, we let the index be
\begin{equation}\label{eqn:Ipsi}
I(\psi) = (j, k, t).
\end{equation}  Note that $I(\psi)$ is well-defined and indeed, the character $v_1^jv_2^k(S) N(S)^{-it}$ is trivial (that is it takes the value $1$ for all $S$) if and only if $i=k=t=0$.

Recall that we want to show for any $\delta_0>0$ and $T_0 \ll V^{5/36 - \delta_0}$, there exists $c>0$ such that 
\begin{equation}\label{eqn:properrorsmallcube}
\int_{-T_0}^{T_0} \sumc_{j, k} b(j)b(k)  \sum_{S} h_S \psi(j, k, t)(S)  \tilde{F}(it)dt \ll \Delta^3 V \exp(-c(\log V)^{1/3 - \epsilon}),
\end{equation}where $\sumc$ denotes a sum over $j, k$ satisfying the conditions in $\sumsharp$ with the added condition that if $|t| < \tau_0$, then $(j, k) \neq (0, 0)$.

Here, and in the sequel, we shall be claiming various quantities are bounded by the quantity on the right side of \eqref{eqn:properrorsmallcube} for some constant $c>0$.  We emphasize here that the constants $c$ appearing in these statements are not necessarily the same at each occurrence.  This is to slightly simplify our notation and avoid writing $c_1, c_2,c_3,...$ etc. throughout.

Our strategy avoids going through the zero density results route, since our result must hold for the product of many small primes, rather than pure prime ideals.  It is therefore convenient to use a technique which applies Heath-Brown's generalized combinatorial identity to our coefficients and proceed directly through large values of Dirichlet polynomials.  Our treatment here is closely analogous to Heath-Brown's work in \cite{HBidentity}.

\subsection{Large value of Dirichlet polynomials}\label{subsec:largesieve}
We first cite the following large sieve type bound which is Theorem 1.1 in Duke's work \cite{Du}.

\begin{lem}\label{lem:dukelargesieve}
Let $N\ge 1$ be a real number and let $c(\fra{a}) \in \mathbb{C}$ be arbitrary for integral ideals $\fra{a}$ satisfying $N(\fra{a}) \le N$ and write
$$\|c\|^2 = \sum_{N(\fra{a}) \le N} |c(\fra{a})|^2.
$$Then there exists some $A = A(R) >0$ such that
$$\sum_{|j|, |k| \le T_0} \int_{-T_0}^{T_0} \left|\sum_{N(\fra{a}) \le N} c(\fra{a}) \theta v_1^jv_2^k(\fra{a}) N(\fra{a})^{it}\right|^2 dt \ll (N+T_0^3) \log^{A}T_0 \|c\|^2,
$$where $\theta$ is any character with modulus bounded by $(\log T_0)^R$ for some fixed $R$.  
\end{lem}
Indeed, Theorem 1.1 in Duke's work \cite{Du} is more general, allowing for an additional average over primitive characters mod $q$, with the bound on the right depending also on $q$.  Our assumption on the modulus of $\theta$ allowed us to replace dependence on the modulus by a power of $\log T_0$.  When we apply Lemma \ref{lem:dukelargesieve}, the $\theta$ above corresponds to the $\theta$ appearing in Remark \ref{rem:characterchi}.

Lemma \ref{lem:dukelargesieve} immediately leads to Theorem 1.3 in \cite{Du} on large values of such Dirichlet polynomials, which is analogous to the one available classically.  To be precise, let $\Omega$ be a set of $\psi = \psi(j, k, t)$ with $j, k, t \ll T_0$ which is $18 \log^2 T_0$ well-spaced in the sense that if $\psi_1= \psi_1(j_1, k_1, t_1)$ and $\psi_2 = \psi(j_2, k_2, t_2)$ are distinct elements in $\Omega$, then either
\begin{align}\label{cond:wellspaced}
&|t_1 - t_2| \geq 18 \log^2 T_0 \textup{ or}\\
&(j_1, k_1) \neq (j_2, k_2).
\end{align}
Duke's Theorem 1.3 \cite{Du} is below.
\begin{lem}\label{lem:dukeupper}
With $\Omega$ as above and $c(S) \in \mathbb{C}$ arbitrary,
\begin{equation}
\sum_{\psi\in \Omega} \left| \sum_{N(S) \leq N} c(S) \psi(S)\right|^2 \ll (N+T_0^3 (\log T_0)^8) \|c\|^2
\end{equation}where as usual
$$\|c\|^2 = \sum_{\substack{N(S)\le N}} |c_S|^2,
$$where as always $\sum_{N(S) \leq N}$ denotes a sum over integral ideals $S$ with $N(S) \leq N$.
Consequently, if $R$ is the number of $\psi \in \Omega$ such that $\left|\sum_{N(S) \le N} c(S) \psi(S)\right| \geq \V$, then 
$$R \ll (N+T_0^3 (\log T_0)^8) \frac{\|c\|^2}{\V^2}.
$$
\end{lem}

For reference later, we also state the usual convexity bound for the Hecke $L$-functions  $L(s, \psi)$ which have Dirichlet series
\begin{equation}
L(s, \psi) = \sum_S \frac{\psi(S)}{N(S)^s}
\end{equation}where the sum is over integral ideals $S$ and is absolutely convergent for $\tRe(s) > 1$.  We refer the reader to \S 1.2 of \cite{Du} for basic properties of these $L$-functions.  The convexity bound we shall use is below.  

\begin{lem}\label{lem:convexity}
With notation as above and $\sigma \le 1/2$, there exists some constant $A_0$ such that
\begin{equation}
L(\sigma, \psi) \ll (q (|t|+|j|+|k|))^{3\frac{1-\sigma}{2}} \log^{A_0} (2+q (|t|+|j|+|k|)).
\end{equation}
\end{lem}

Obviously similar bounds hold for $\sigma>1/2$, with some accounting for the contribution of a possible pole when $\sigma = 1$.  In the sequel, we shall have that $|t|+|j|+|k| \ll T_0$ while $q \ll (\log T_0)^R$ for some fixed $R$, so we will often use that there exists some $A_0 >0$ (possibly different from that in the statement of Lemma \ref{lem:convexity}) such that
$$L(\sigma, \psi) \ll T_0^{3\frac{1-\sigma}{2}} (\log T_0)^{A_0}.
$$

The proof of Lemma \ref{lem:convexity} follows along standard lines through an application of the functional equation and the Phragmen-Lindelof principle.  For instance, Duke's Theorem 1.2 and the remark immediately following its proof \cite{Du} implies our Lemma \ref{lem:convexity} upon taking his $\delta_0 = \frac{1}{\log (2+q (|t|+|j|+|k|)}$ and keeping track of the dependence of the implied constant on $\delta_0$.  Similar versions of this bound are stated in other works without proofs (e.g. see (2.4) of Coleman's work \cite{Coleman2}).

We also require the following Lemma, which is essentially a result of Montgomery.
\begin{lem}\label{lem:mont}
With $\Omega$ and $c(S)$ as above, let $R$ be the number of $\psi \in \Omega$ such that
$$\left| \sum_{N(S) \le N} c(S) \psi(S) \right| \geq \V.
$$Then there exists some $A' > 0$ such that
\begin{equation}
R \V^2 \ll (N + R T_0^{3/2})\|c\|^2 (\log NT_0)^{A'}.
\end{equation}
\end{lem}
The analogous result for Dirichlet $L$-functions and for the Riemann zeta function was proven by Montgomery in \cite{Mont} (see e.g. Theorem 8.2 - 8.4 in \cite{Mont}), following work of Halasz.  Montgomery states immediately after these theorems that a generalization to number fields can likely be made but does not provide details.  The proof of Lemma \ref{lem:mont} requires little modification beyond Montgomery's work, following the same steps and putting in the convexity bound in Lemma \ref{lem:convexity} rather than the analogous convexity bound for Dirichlet $L$-functions in \cite{Mont}.  We will also use the following modification of Montgomery's result, due to Huxley.
\begin{lem}\label{lem:huxmont}
With notation as in Lemma \ref{lem:mont},
\begin{equation}
R \ll \frac{N\|c\|^2}{\V^2} (\log NT_0)^{A'} + \frac{NT_0^3\|c\|^6}{\V^6}  (\log NT_0)^{3A'}.
\end{equation}
\end{lem}
\begin{proof}
If $\V^2 \geq C_0 \|c\|^2 T_0^{3/2} (\log N T_0)^{A'}$, then by Lemma \ref{lem:mont}, 
$$R \leq K_0 \left(\frac{N \|c\|^2}{\V^2} \log^{A'}NT_0 +  \frac{R}{C_0}\right),
$$for some fixed constant $K_0 \ge 0$.  Provided that $C_0 \ge 2K_0$,
$$R \ll \frac{N \|c\|^2}{\V^2} \log^{A'}NT_0.
$$   

Otherwise, we have $\V^2 < C_0 \|c\|^2 T_0^{3/2} (\log N T_0)^{A'}$ for the fixed $C_0 = 2K_0$, and we define $\tau_1$ by setting
$$\V^2 = C_0 \|c\|^2 \tau_1^{3/2} (\log N T_0)^{A'}.$$
It is then clear that $\tau_1 < T_0$.

We now cover $[1, T_0] \times [1, T_0] \times [-T_0, T_0]$ with cubes of the form
$$C = [t_1, t_1 + \tau_1] \times [t_2, t_2 + \tau_1]\times [t_3, t_3 + \tau_1],
$$that is, cubes with side $\tau_1$.  We do not mind if the cover is not disjoint.  The relevant fact to be used later is that we need only $\ll \frac{T_0^3}{\tau_1^3}$ such cubes $C$.

Now, fix such a cube $C$.  Then, for $\psi$ such that $I(\psi) \in C$ where recall $I(\psi)$ is defined in \eqref{eqn:Ipsi}, we may write
$$\sum_{N(S)\le N} c(S) \psi(S) = \sum_{N(S)\le N} c(S) \psi(\floor{t_1}, \floor{t_2}, t_3)(S) \psi'(S),
$$for some fixed $\psi'$ satisfying $I(\psi') \in C - (\floor{t_1}, \floor{t_2}, t_3)$.  We now apply Lemma \ref{lem:mont} (with $\tau_1$ in place of $T_0$ and $\psi'$ in place of $\psi$) again to see that the number of $\psi \in C$ such that 
$$\sum_{N(S)\le N} c(S) \psi(S) \ge \V
$$ is $\ll \frac{N \|c\|^2}{\V^2} \log^{A'}N\tau_1 \le \frac{N \|c\|^2}{\V^2} \log^{A'}NT_0.$  Since there are at most $\frac{T_0^3}{\tau_1^3}$ such cubes $C$, 
\begin{align}
R &\ll \left(\bfrac{T_0}{\tau_1}^3 +1 \right) \frac{N\|c\|^2}{\V^2} (\log NT_0)^{A'}\\
&\ll \frac{N\|c\|^2}{\V^2} (\log NT_0)^{A'} + \frac{NT_0^3\|c\|^6}{\V^6}  (\log NT_0)^{3A'}.
\end{align}

\end{proof}

\subsection{Decomposition}\label{subsec:decomp}
Recall that for $\zeta_K(s)$ the Dedekind zeta function of $K = \mathbb{Q}(\sqrt[3]{2})$, we define $\Lambda_K(T)$ by 
$$-\frac{\zeta_K'}{\zeta_K}(s) = \sum_{T} \frac{\Lambda_K(T)}{N(T)^s}.
$$ 

It now remains to rewrite our sum as a sum of sums in an appropriate form to apply the bounds on large values in \S \ref{subsec:largesieve}.  For this, we recall that 
\begin{equation}\label{eqn:hbetacopy}
h_S = \prod_{i=1}^{n+1} \frac{\Lambda_K(T_i)\frak{W_i}(N(T_i))}{m_i \xi \log X},
\end{equation}where $S = T_1...T_{n+1}$ and for all $1\le i\le n+1$,
\begin{equation}\label{eqn:frakWbddcopy}
\frak{W}_i^{(k)}(x) \ll_k (\iota x)^{-k},
\end{equation}for all $k \ge 0$ where $\iota$ is defined as in \eqref{eqn:iota}.
Further recall
\begin{equation}\label{eqn:xicopy}
\xi = \frac{1}{\log \log X},
\end{equation}and
\begin{equation}\label{eqn:deltacopy}
\delta \asymp 1.
\end{equation}

We now apply Heath-Brown's identity as introduced in \cite{HBidentity} to $\Lambda_K(T_i)$.  In this context, we write for $\tRe s >1$,
\begin{equation}
\frac{1}{\zeta_K(s)} = \sum_{T} \frac{\mu_K(T)}{N(T)^s},
\end{equation}and let
\begin{equation}
M(s) = \sum_{N(T) \le U} \frac{\mu_K(T)}{N(T)^s}.
\end{equation}
Then we have the identity 
\begin{equation}\label{eqn:heathbrownidentity}
\frac{\zeta_K'(s)}{\zeta_K(s)} = \sum_{j=1}^{k} (-1)^{j-1} \binom{k}{j} \zeta_K(s)^{j-1} \zeta_K'(s) M(s)^j + \frac{\zeta_K'(s)}{\zeta_K(s)} \left(1-\zeta_K(s)M(s)\right)^k.
\end{equation}Write $\frak{W} = \frak{W_i}$ for some fixed $i$ for simplicity, where $\frak{W_i}$ is as in the definition of $h_S$ in \eqref{eqn:hbetacopy}.  Then equating coefficients in \eqref{eqn:heathbrownidentity} for ideals $T$ satisfying $N(T) \le U^k$  gives us that
\begin{equation}
\Lambda_K(T) = \sum_{J=1}^{k} (-1)^{J-1} \binom{k}{J} \prod_{\substack{N_1, N_2,...,N_J, M_1, M_2,...,M_J \\  N(M_i) \le U\\ N_1...N_JM_1...M_J = T}} \log (N(N_J)) \mu_K(M_1)\mu_K(M_2)...\mu_K(M_J) 
\end{equation}
from which it follows that
\begin{align}\label{eqn:decomp1}
\sum_{N(T) \le U^k} \Lambda_K(T)\psi(T) \frak{W}(N(T))
= \sum_{J=1}^k (-1)^{J-1} \binom{k}{J} \Scal(J)
\end{align} where
\begin{equation}\label{eqn:Scalj}
\Scal(J) = \sum_{N_1, N_2,...,N_J}\sum_{\substack{M_1, M_2,...,M_J \\ N(M_i) \le U \\ N(T) \le U^k}} \log (N(N_J)) \mu_K(M_1)\mu_K(M_2)...\mu_K(M_J) \frak{W}(N(T))\psi(T),
\end{equation}where we have continued to write $T = M_1...M_JN_1...N_J$ for brevity.
Recall that $\psi$ is the same on associates and so is well defined on (principal) ideals.  By \eqref{eqn:T_0}, we have that $V^\epsilon \ll T_0$, so we may let $U$ be the power of $2$ satisfying
\begin{equation}\label{eqn:U}
V^\epsilon \ll U < T_0^{9/5},
\end{equation}
and choose 
\begin{equation}\label{eqn:k}
k = \ceil{2/\epsilon}.
\end{equation}
Hence $U^k \gg V^2$ and so the condition $N(T) \leq U^k$ in \eqref{eqn:decomp1} and \eqref{eqn:Scalj} is extraneous since $\frak{W}(N(T))$ is supported on $T$ satisfying $N(T) \ll V$.  In the sequel, we drop the condition $N(T) \leq U^k$ entirely.

We would like to separate variables by using the inverse Mellin transform of $\frak{W}$.  Since $\frak W$ is supported on a long interval, we apply a standard partition of unity to $\frak W$ to write 
$$\frak W(x) = \sum_i \frak W(x) \frak U_i(x),
$$where each $\frak U_i(x)$ is supported on $[\mathcal M_i/2, 5/2 \mathcal M_i]$ for $\mathcal M_i = 2^i$ and satisfies that for all $k\ge 0$,
\begin{equation}\label{eqn:partitionunitybdd}
\frak U_i^{(k)} (x) \ll_k \bfrac{1}{\mathcal M_i}^k.
\end{equation}
We need only consider at most $O(\log X)$ such $i$ due to the support of $\frak W$.  Thus $\Scal(J) = \sum_{\frak V} \Scal(J, \frak V, \psi)$, where
\begin{align}\label{eqn:ScaljVpsi}
\Scal(J, \frak V, \psi) &= \sum_{N_1, N_2,...,N_J}\sum_{\substack{M_1, M_2,...,M_J \\  \\ N(M_i) \le U}} \log (N(N_J)) \mu_K(M_1)\mu_K(M_2)...\mu_K(M_J) \frak{V}(N(T))\psi(T)\\
\end{align}
  
  For clarity, let us now note in order to prove \eqref{eqn:properrorsmallcube}, it suffices to show that
\begin{equation}\label{eqn:afterdecomp}
\int_{-T_0}^{T_0} \sumc_{j, k} b(j)b(k) \prod_{i=1}^{n+1} \frac{1}{m_i \xi \log X} \Scal(J_i, \frak{V_i}, \psi(j, k, t)) \tilde{F}(it)dt \ll \Delta^3 V \exp(-c(\log V)^{1/3 - \epsilon}),
\end{equation}for some constant $c>0$.  Indeed, recall $n\ll 1$ and so the total number of quantities from \eqref{eqn:afterdecomp} required to obtain the left hand side of \eqref{eqn:properrorsmallcube} is $\ll (\log X)^{c'}$ for some constant $c'$ where $O(\log X)$ is the aforementioned bound for the number of partitions required to cover the support of $\mathfrak W$.  This is acceptable since $(\log X)^{c'} \exp(-c(\log V)^{1/3 - \epsilon}) \ll\exp\left(-\frac{c}{2}(\log V)^{1/3 - \epsilon}\right)$.

Let us fix one $\frak V(x) = \frak W(x) \frak U_i(x)$, and assume that $\frak{V}(x)$ is supported on $[\mathcal M/2, 5/2 \mathcal M]$ for some $\mathcal M \ge X^{\delta}$ and note that 
$$\frak V^{(k)} (x) \ll_k \bfrac{1}{\iota \mathcal M}^k,
$$by \eqref{eqn:frakWbddcopy} and \eqref{eqn:partitionunitybdd}.

As usual, let 
$$ \tilde {\frak{V}} (s) = \int_0^\infty \frak V(x) x^{s-1} dx
$$be the standard Mellin transform of $\frak V(x)$.  Since $\frak V(x)$ is compactly supported away from $0$, $\frak V(s)$ is entire. 

We will be separating variables by expressing $\mathfrak V(x)$ as the inverse Mellin transform of $\tilde{\mathfrak{V}}(s)$.  To this purpose, we now record the following standard bounds on $\tilde {\frak{V}} (s)$.
\begin{lem}\label{lem:frakVbdd}
For any natural number $m$
\begin{equation}\label{eqn:frakVbdd}
\tilde {\frak{V}} (s) \ll_m \frac{1}{\iota^{m} |s|^m},
\end{equation}
for $\tRe s = 0$.  Moreover
\begin{equation}\label{eqn:frakVbdd2}
\tilde {\frak{V}} (s) \ll \min\{\frac{1}{|s|}, 1\}
\end{equation}
for $\tRe s = 0$.
\end{lem}
\begin{proof}
Set $\tRe s = 0$.  Then trivially, 
$$\tilde {\frak{V}} (s) \ll \int_0^\infty \frac{\frak V(x)}{x} dx \ll \frac{\mathcal{M}}{\mathcal{M}} = 1.
$$
Integration by parts once yields that
\begin{align}\label{eqn:tildefrakVbdd1}
\tilde {\frak{V}} (s) = -\int_0^\infty \frac{\frak{V}'(x) x^s}{s} dx,
\end{align}and writing
$$\frak{V}'(x) = \frak W'(x) \frak U_i(x) + \frak U_i'(x) \frak W(x). 
$$We note that 
$$\frak U_i'(x) \frak W(x) \ll \frac{1}{\mathcal M},$$
while
$$\frak W'(x) \frak U_i(x) = 0
$$unless on an interval of length $\iota \mathcal M$ on which it trivially satisfies 
$$\frak W'(x) \frak U_i(x) \ll \frac{1}{\iota \mathcal M},
$$by \eqref{eqn:frakWbddcopy}.  Thus the bound
\begin{align*}
\tilde {\frak{V}} (s) \ll \frac{1}{|s|}
\end{align*}follows from \eqref{eqn:tildefrakVbdd1}.

By \eqref{eqn:frakVbdd}, and integration by parts $m$ times, it follows that
$$\tilde {\frak{V}} (s) \ll \int_0^\infty \frac{\frak{V}^{(m)}(x) x^{s+m-1}}{|s|^m} dx \ll \frac{1}{|s|^m} \int_{\mathcal M/2}^{5/2 \mathcal M} \frac{x^{m-1}}{(\iota \mathcal M)^m} dx \ll \frac{1}{\iota^m |s|^m},
$$ as desired.
\end{proof}

Before removing the smooth function $\frak V$, we first further partition our sums over $N_i$ and $M_i$.  

\subsubsection{Partitioning sums over $N_i$ and $M_i$}
We now proceed to partition the sums over each $N_i$ in the sum $S(J,\mathfrak{V}, \psi)$ using smooth compactly supported functions $\omega_r(t)$.  To be precise, for $r\ge 1$, let $\nu_r = 2^{r-2}$ run through powers of $2$ starting from $1/2$.  Then set 
$$\omega_r(t) = W\bfrac{t}{\nu_r}$$ 
where $W$ is a smooth function compactly supported on $[1, 3]$ satisfying 
\begin{equation}\label{eqn:Wstuff0}
W^{(k)}(t) \ll_k 1
\end{equation} for all $k\ge 0$.  

For bookkeeping convenience, we further introduce a dyadic partition to the sums over $M_i$.  Thus, we write for each $J, \mathfrak{V}, \psi$ that
\begin{align*}
S(J, \mathfrak{V}, \psi) = &\sum_{\omega_{1},...,\omega_{J}} \sumd_{L_1,...,L_J} \sum_{N_1, N_2,...,N_J} \omega_{1}(N(N_1))...\omega_{J}(N(N_J)) \\
&\sum_{\substack{M_1, M_2,...,M_J \\  \\ L_l < N(M_i) \le 2L_l \textup{ for all } 1\le l\le J}}  \log (N(N_J)) \mu_K(M_1)\mu_K(M_2)...\mu_K(M_J) \frak{V}(N(T))\psi(T),
\end{align*}where $\sumd_{L_1,...,L_J}$ denotes a sum over powers of two ranging from $1/2$ to $U/2$, recalling that $U$ is also a power of two from \eqref{eqn:U}.  To prove \eqref{eqn:afterdecomp}, it suffices to show that

\begin{equation}\label{eqn:afterdecomp2}
\int_{-T_0}^{T_0} \sumc_{j, k}b(j)b(k) \prod_{i=1}^{n+1} \frac{1}{m_i \xi \log X} T(J_i, \frak{V}_i, \psi(j, k, t))  \tilde{F}(it)dt \ll \Delta^3 V \exp(-c(\log V)^{1/3 - \epsilon}),
\end{equation}for some $c>0$,
where
\begin{align}\label{eqn:TJpsi}
T(J_i, \frak{V}_i, \psi) &= \sum_{N_1, N_2,...,N_{J_i}} \omega_{1, i}(N(N_1))...\omega_{J_i, i}(N(N_{J_i})) \notag \\
&\sum_{\substack{M_1, M_2,...,M_{J_i} \\ L_{l, i} < N(M_l) \le 2L_{l, i} \textup{ for all } 1\le l\le J_i}} \log (N(N_{J_i})) \mu_K(M_1)...\mu_K(M_{J_i}) \frak{V}_i(T) \psi(T)
\end{align}for each $J_i$, $\frak{V}_i$ and $\psi$, and for fixed smooth compactly supported functions $\omega_{l, i}(t)$ satisfying
$$\omega_{l, i}^{(k)}(t) \ll_k \frac{1}{\nu_{l, i}^k}
$$for all $k\ge 0$ and supported on $t \in [\nu_{l, i}, 3 \nu_{l, i}]$ for $1\le l\le J_i$.  

Indeed, there are $\ll (\log X)^{2J} = (\log X)^{O(1)}$ terms in the sum $\sum_{\omega_{1},...,\omega_{J}} \sumd_{L_1,...,L_J}$ and the additional product $\prod_{i=1}^{n+1}$ gives us that the expression on the left side of \eqref{eqn:afterdecomp} is a sum of $\ll (\log X)^{2(n+1)J} = (\log X)^{O(1)}$ terms of the form on the left side \eqref{eqn:afterdecomp2}.

Note that in \eqref{eqn:afterdecomp2}, we may take 
\begin{equation}\label{eqn:decompnewlength}
\prod_{i=1}^{n+1} \prod_{l = 1}^{J_i} (2L_{l, i} v_{l, i}) \asymp V.
\end{equation}
Here, we remind the reader that we are still writing $T = N_1....N_{J_i}M_1...M_{J_i}$ and the condition \eqref{eqn:decompnewlength} is inherited from $N(T) \ll V$, and we will be using this new condition \eqref{eqn:decompnewlength} after removing $\frak{V}$ through Mellin inversion.

\subsubsection{Removal of $\mathfrak{V}$ through Mellin inversion}

We now write
\begin{align}\label{eqn:inversemellinScal}
T(J, \frak V, \psi) &= \sum_{N_1, N_2,...,N_J}\omega_{1}(N(N_1))...\omega_{J}(N(N_J)) \notag \\
&\sum_{\substack{M_1, M_2,...,M_J \\  \\ L_l < N(M_i) \le 2L_l \textup{ for all } 1\le l\le J}} \log (N(N_J)) \mu_K(M_1)\mu_K(M_2)...\mu_K(M_J) \frak{V}(T)\psi(T)\\
&=\frac{1}{2\pi i} \int_{(0)} T(J, \psi)(s) \tilde {\frak{V}}(s) (1+|\iota s|)^{2J} ds,
\end{align}where 
\begin{align}\label{eqn:TJpsi}
T(J, \psi)(s) 
&=\sum_{N_1, N_2,...,N_J}\frac{\omega_{1}(N(N_1))}{(1+|\iota s|)^2}...\frac{\omega_{J}(N(N_J))}{(1+|\iota s|)^2} \log (N(N_J)) \\
&\sum_{\substack{M_1, M_2,...,M_J \\  \\ L_l < N(M_i) \le 2L_l \textup{ for all } 1\le l\le J}} \mu_K(M_1)\mu_K(M_2)...\mu_K(M_J) \frac{\psi(T)}{N(T)^s}.
\end{align}

By Lemma \ref{lem:frakVbdd}, we may truncate the integral in $s$ in \eqref{eqn:inversemellinScal} to $|s| \le V^\epsilon$, with error $\ll V^{-A}$ for any $A>0$ since $\iota V^{\epsilon} \gg V^{\epsilon/2}$ for any $\epsilon>0$.  The slightly unnatural seeming distribution of the $(1+|\iota s|)$ factors is for convenience only.  Note that $\tilde {\frak{V}}(s) (1+|\iota s|)^{2J}$ satisfies the same bounds as $\tilde {\frak{V}}(s)$ from \eqref{eqn:frakVbdd} and \eqref{eqn:frakVbdd2}.



\subsection{Notation and pruning}
We now write
\begin{equation}\label{eqn:decom}
\prod_{i=1}^{n+1} \frac{1}{m_i \xi \log X} T(J_i, \psi)(s_i) := \sum_{\substack{M_1,...M_r\\ L_l < N(M_l) \le 2L_l \textup{ for all } 1\le l \le r}} \sum_{\substack{N_1,...,N_r}} c_1(N_1)...c_r(N_r) \prod_{l=1}^r c_{i+r}(M_l),
\end{equation}where for 
\begin{equation}\label{eqn:r}
r = \prod_{i=1}^{n+1} J_i \ll k^{1/\delta} \ll 1
\end{equation} and $1\le l\le r$, either 
$$c_l(N_l) = \frac{1}{(1+|\iota s_{i(l)}|)^2} \frac{\log N(N_l)}{m_{i(l)} \xi \log X} W\bfrac{N(N_l)}{v_l} \frac{\psi(N_l)}{N(N_l)^{s_{i(l)}}}
$$or 
$$c_l(N_l) = \frac{1}{(1+|\iota s_{i(l)}|)^2}  W\bfrac{N(N_l)}{v_l}\psi(N_l) \frac{\psi(N_l)}{N(N_l)^{s_{i(l)}}},
$$where $W$ satisfies \eqref{eqn:Wstuff0}.  In the above, $i(l)$ is some index satisfying $1\le i(l) \le n+1$, and $|s_{i(l)}| \leq V^\epsilon$ with $\tRe s_{i(l)}  = 0$.  Writing $i = i(l)$, we see that both are of the form
\begin{equation}\label{eqn:ci1}
c_l(N_l) = \frac{1}{(1+|\iota s_{i}|)^2}  \omega_l\bfrac{N(N_l)}{v_l}\frac{\psi(N_l)}{N(N_l)^{s_i}},
\end{equation}
for some smooth function $\omega_l$ supported on $[1, 3]$.  We have either $\omega_l = W$, or $\omega_l(t) = \frac{1}{m_i \xi \log X} \left( W(t) \log t + \log v_l W(t)\right)$.  In both cases, by \eqref{eqn:Wstuff0}, we have the bounds
\begin{equation}\label{eqn:omegabdd}
\omega_l^{(k)}(t) \le 1,
\end{equation}and integrating by parts $k$ times gives
\begin{equation}\label{eqn:omegatildebdd}
\tilde{\omega_l}(s) \ll_k \frac{1}{(|s|+1)^k}
\end{equation}for any $k\ge 0$.

Similarly, 
\begin{equation}\label{eqn:ci2}
c_{l+r}(M_l) = \mu_K(M_l) \frac{\psi(M_l)}{N(M_l)^{s_i}} 
\end{equation}
for all $1\le l\le r$ and supported on a interval of the form $L_l < N(M_l) \le 2L_l$ for $L_l \le U$.  Thus, we write
\begin{equation}\label{eqn:frakT}
\frak T(\psi) := \prod_{i=1}^{n+1} \frac{1}{m_i \xi \log X} T(J_i, \psi)(s_i) = \prod_{l=1}^{2r} \mathcal{F}_l(\psi),
\end{equation}where 
\begin{equation}\label{eqn:Fipsi}
\mathcal{F}_l(\psi) = \sum_{M} c_l(M)
\end{equation}
where $c_l$ one of the above listed possibilities in \eqref{eqn:ci1} or \eqref{eqn:ci2}.  It should not disturb the reader that we are writing $M$ for both the $M_l$s and the $N_l$s.  

Note that the sums in $\mathcal{F}_l(\psi)$ are of length $v_l$ or $L_l$ depending on whether we are in the case \eqref{eqn:ci1} or \eqref{eqn:ci2}.  For notational convenience, let $\frak w_l$ for the length of $\mathcal{F}_l(\psi)$, so $\frak w_l = v_l$ or $\frak w_l = L_l$.  We now write
\begin{equation}\label{eqn:frakT}
\frak T(\psi) = E(\psi) K(\psi),
\end{equation}where $E(\psi)$ is the product of the sums $\mathcal{F}_l$ in \eqref{eqn:Fipsi} where $\frak w_l \leq V^{\delta_0}$ and $K(\psi)$ is the rest where $\delta_0$ is a parameter chosen to be sufficiently small.  

Recall that 
$$r \ll 1$$
from \eqref{eqn:r}.  Let $\frak u_E$ be the length of $E$ - in other words, the product of the lengths $\frak u_i$ of those $F_i$ that appear in $E$.  Similarly define $\frak u_K$.  Note that we may assume $V \asymp \frak w_1...\frak w_{2r}$ by \eqref{eqn:decompnewlength}.  

Thus, for any $\epsilon_0>0$, we may choose $\delta_0 = \frac{\epsilon_0}{2r}$ depending on $\epsilon_0$ and $k$ so that 
$$\frak u_E \ll V^{r \delta_0} = V^{\epsilon_0/2}$$ 
and thus 
\begin{equation}\label{eqn:uF}
\frak u_K \gg V^{1-\epsilon_0}
\end{equation}
We note for future reference that $\delta_0 > 0$ is a fixed positive number.

By the bound \eqref{eqn:frakVbdd}, to prove \eqref{eqn:afterdecomp2} we want to show that
\begin{align*}
\int_{-T_0}^{T_0} b(j)b(k) \sumc_{j, k} \frak T(\psi(j, k, t))\tilde{F}(it) dt \ll \Delta^3 V \exp(-c(\log V)^{1/3 - \epsilon}),
\end{align*}for some $c>0$ and fixed $s_1,...,s_{n+1}$ satisfying $\tRe s_i = 0$ and $s_i \ll V^\epsilon$, the the dependence on $s_1,...,s_{n+1}$ is as in \eqref{eqn:frakT}.  Thus it suffices to show that
\begin{equation}
 \int_{-T_0}^{T_0} \sumc_{j, k} \frak T(\psi(j, k, t))  dt\ll V \exp(-c(\log V)^{1/3 - \epsilon}),
\end{equation}for some $c>0$, upon recalling that $\tilde F(it) \ll \Delta$ and $b(j) \ll \Delta$.  Trivially $|E| \ll \frak u_E$ and $\frak u_E \frak u_K \ll V$ by \eqref{eqn:decompnewlength}, so it suffices to show that
\begin{equation}\label{eqn:afterdecomp5}
 \int_{-T_0}^{T_0} \sumc_{j, k} K(\psi(j, k, t))  dt\ll \frak u_K \exp(-c(\log V)^{1/3 - \epsilon}).
\end{equation}

For future convenience, we first deal with the case where $v_l$ corresponding to coefficients of the form \eqref{eqn:ci1} is large.  In particular, we separate the case $v_l > T_0^{9/5}$.  In this case, we write 
\begin{equation}\label{eqn:Fpsilong}
\mathcal{F}(\psi) = \sum_M c(M)
\end{equation}
where $c$ is given by 
\begin{equation}
c(M) = \frac{1}{(1+|\iota s|)^2}  \omega\bfrac{N(M)}{v}\frac{\psi(M)}{N(M)^{s}},
\end{equation}
is of the form \eqref{eqn:ci1}, with $v > T_0^{9/5}$.  For simplicity of notation, we have written $\mathcal{F}$ for $\mathcal{F}_l$, $\omega$ for $\omega_l$, $c$ for $c_l$, $M$ for $N_l$, and $v$ for $v_l$.  Note that the coefficients appearing in \eqref{eqn:ci2} are supported on intervals of length $L_l \leq U < T_0^{9/5}$, so by assumption, our coefficients must be of the form \eqref{eqn:ci1}.



We may then write
\begin{align*}
\mathcal{F}(\psi) =   \frac{1}{(1+|\iota s|)^2} \frac{1}{2\pi i}  \int_{(2)} L(w+s+it, \nu_0 \nu_1^j \nu_2^k) v^w \tilde{\omega}(w) dw.
\end{align*}Shifting contours to $\tRe w = 1/2$, we pick up a residue at $w = 1-s-it$, and so
\begin{equation}
\mathcal{F}(\psi) =    \frac{1}{(1+|\iota s|)^2} \frac{1}{2\pi i} \int_{(1/2)} L(w+s, \psi) v^w \tilde{\omega}(w) dw + O\left( \frac{|\delta_\psi v  \tilde{\omega}(1-s-it)|}{(1+|\iota s|)^2} \right).
\end{equation}
Here, $\delta_\psi$ is $0$ unless $\psi(S) = \nu_0(S) N(S)^{-it}$ where $\nu_0(S)$ is trivial in which case $\delta_\psi = 1$.

We bound the contribution of $v \tilde{\omega}(1-s-it)$ in the Lemma below.  
\begin{lem}\label{lem:polar}
With notation as above, there exists $c>0$ such that
\begin{equation}\label{eqn:polar}
\int_{\tau_0 \le |t| \le T_0} \frac{|v  \tilde{\omega}(1-s-it)|}{(1+|\iota s|)^2} dt \ll v \exp(-c(\log V)^{1/3 - \epsilon})
\end{equation}
\end{lem}
\begin{proof}
By \eqref{eqn:omegatildebdd}, we see that
\begin{equation}\label{eqn:polarbdd1}
\left|\frac{|v  \tilde{\omega}(1-s-it)|}{(1+|\iota s|)^2}\right| \ll \frac{v}{(1+|\iota s|)^2} \frac{1}{(1+|s+it|)^k}
\end{equation} 
for any $k\ge 0$.  

We examine the cases $|s| \le \tau_0/2$ and $|s| > \tau_0/2$ separately.  When $|s| \le \tau/2$, $|s+it| \gg \tau_0$ for $|t|\ge \tau_0$, and so \eqref{eqn:polarbdd1} imples that
\begin{align*}
\int_{\tau_0 \le |t| \le T_0} |v \tilde{\omega}(1-s-it)| dt 
\ll v \tau_0^{-k}
\end{align*}for any $k$, which suffices for our Lemma upon recalling $\tau_0 =  \Delta_0^{-1/2} = \exp(c_0/2 (\log L)^{1/2}) \gg \exp((\log V)^{1/3})$ from \eqref{eqn:tau0def}, \eqref{eqn:Delta0} and \eqref{eqn:L}.

When $|s| \ge \tau_0/2$, \eqref{eqn:polarbdd1} gives that
\begin{align*}
&\int_{\tau_0 \le |t| \le T_0} \left|\frac{v \tilde{\omega}(1-s-it)}{{(1+|\iota s|)^2}}\right| dt \\
&\ll \frac{v}{|\iota \tau_0|^2} \int_{\tau_0 \le |t| \le T_0}\frac{1}{(1+|s+it|)^k} dt \\
&\ll v \frac{\tau_0}{|\iota \tau_0|^2}. 
\end{align*}Recall that $\iota  = \exp(-(\log X)^\epsilon)$ from \eqref{eqn:iota} and $\tau_0 =  \Delta_0^{-1/2} = \exp(c_0/2 (\log L)^{1/2}) \gg \exp((\log V)^{1/3})$ from \eqref{eqn:tau0def}, \eqref{eqn:Delta0} and \eqref{eqn:L}.  Thus the bound above suffices for our Lemma.
\end{proof}

For ease of notation, we now write out 
$$K(\psi) = \prod_{l=1}^{r_0} \mathcal{F}_l(\psi)
$$for some $r_0 \ll r \ll 1$, which was first introduced in \eqref{eqn:frakT}.  We aim to replace $K(\psi)$ by
\begin{equation}\label{eqn:G}
G(\psi) = \prod_{l=1}^{r_0} G_l(\psi)
\end{equation} where
\begin{equation}\label{eqn:Gl1}
G_l(\psi) =  \frac{1}{(1+|\iota s|)^2} \frac{1}{2\pi i} \int_{(1/2)} L(w+s+it, \nu_0 \nu_1^j \nu_2^k) v_l^w \tilde{\omega}_l(w) dw,
\end{equation}if $\mathcal{F}_l(\psi)$ is of the form in \eqref{eqn:Fpsilong} with $\mathfrak{w}_l = v_l > T_0^{9/5}$ and otherwise
\begin{equation}\label{eqn:Gl2}
G_l(\psi) = \mathcal{F}_l(\psi).
\end{equation}

We now claim that to prove \eqref{eqn:afterdecomp5}, it suffices to show that
\begin{equation}\label{eqn:afterdecomp6}
 \int_{-T_0}^{T_0} \sumc_{j, k} \left|G(\psi(j, k, t))\right| dt\ll \frak u_F \exp(-c(\log V)^{1/3 - \epsilon}).
\end{equation}
Note that $G_l(\psi) \neq \mathcal{F}_l(\psi)$ only if $j=k=0$, $\nu_0$ is trivial, and $\mathfrak{w}_l > T_0^{9/5}$.  Further note $|G_l(\psi)| \ll \mathfrak{w}_l$ for all $l$ since $|\mathcal{F}_l(\psi)| \ll \mathfrak{w}_l$ and $G_l(\psi) = \mathcal{F}_l(\psi) + O(\mathfrak{w}_l \exp(-c(\log V)^{1/3 - \epsilon})$.  We may replace $K(\psi)$ by $G(\psi)$ by replacing each $\mathcal{F}_l(\psi)$ by $G_l(\psi)$ one by one, each time incurring an error of at most $\frak u_F \exp(-c(\log V)^{1/3 - \epsilon})$ by Lemma \ref{lem:polar}.  Since there are only $r_0 \ll 1$ of these replacements, the claim follows.


\subsection{Reduction to a discrete set}\label{subsec:reductiontodiscrete}

For any integrable function $\mathcal F$, we know that there exists $t_1,...,t_n$ such that $|t_i - t_j| \ge L >0$ for each $i\neq j$ with
$$\int_{-T_0}^{T_0} |\mathcal F(t)| dt \ll L \sum_{i=1}^n |\mathcal F(t_i)|.
$$Thus there exists a set $\Omega$ which is $18 \log^2 T_0$ well-spaced in the sense of \eqref{cond:wellspaced} such that

\begin{equation}
  \int_{-T_0}^{T_0} \sumc_{j, k} \left|G(\psi)\right| dt\ll \log^2 T_0 \sum_{\psi \in \Omega} \left|G(\psi)\right|.
\end{equation}

\begin{remark}\label{remark:psi1trivialcond}
From the conditions on $\sumc$, we also have that for $\psi \in \Omega$ with $\psi(M) = v_0v_1^jv_2^k(M) N(M)^{-it}$, that if $v_0v_1^jv_2^k(M)$ is the trivial character, then $|t| \ge \tau$ for $\tau_0  = \exp(c_0/2 \sqrt{\log L})$ from \eqref{eqn:tau0def} and \eqref{eqn:Delta0}.
\end{remark}

We thus see that in order to prove \eqref{eqn:afterdecomp6} it suffices to show that
\begin{equation}\label{eqn:afterdecomp7}
\sum_{\psi \in \Omega} \left| G(\psi)\right| \ll \frak u_F \exp(-c(\log V)^{1/3 - \epsilon}).
\end{equation}

For each $\psi \in \Omega$, we define $\sigma_j(\psi)$ to be such that
\begin{equation}\label{eqn:sigmaj}
|G_j(\psi)| = \mathfrak w_j^{\sigma_j(\psi)}.
\end{equation}
We split the range for $\sigma_j(\psi)$ into $O(\log V)$ ranges of the form $I_0 = (-\infty, 1/2]$, and 
$$I_l = \left(1/2+\frac{l-1}{\mathfrak{L}}, 1/2+ \frac {l}{\mathfrak{L}}\right]
$$for $1\le l \le \mathfrak{L} := \floor{\log V}$.
Now we have the following Lemma.

\begin{lem}\label{lem:zerofreebdd}
For any $c_i$ as defined in either \eqref{eqn:ci1} or \eqref{eqn:ci2} and for any $\delta_0>0$, $\frak w_i \gg V^{\delta_0\delta}$, we have
\begin{equation}
\sum_{N(M) \ll \frak w_i} c_i(M) \ll_\epsilon \frak w_i \exp(-c(\log \frak w_i)^{1/3-\epsilon}),
\end{equation}where $c>0$ is a constant depending on $\delta_0$.  The implied constant above is not effective due to the possible presence of Siegel zeros.
\end{lem}  

\begin{rem}
For clarity, we note that bound in Lemma \ref{lem:zerofreebdd} uses that our $q\le (\log X)^R$ for some $R$, and that $\log T_0 \ll \log V \asymp \log X$.  
\end{rem}

Note that Lemma \ref{lem:zerofreebdd} implies that there exists some $c>0$ with
\begin{equation}\label{eqn:sigmajbdd}
\sigma_j(\psi) \leq 1-\frac{c}{(\log \frak w_j)^{2/3+\epsilon}}.
\end{equation}This follows immediately from $G_j = F_j + O(\frak w_i \exp(-c(\log \frak w_i)^{1/3-\epsilon}))$ for some $c$.  We leave the standard proof of Lemma \ref{lem:zerofreebdd} until \S \ref{subsec:zerofree}.  Now we continue the proof of Proposition \ref{prop:errorsmallcube}.  

First, express $\Omega$ as a (not necessarily disjoint) union of sets $C(j, l)$ where $\psi \in C(j, l)$ if and only if $\sigma_j(\psi)$ is maximal for $1\le j\le 2r$, and $\sigma_j(\psi) \in I_l$.  To be explicit, we are setting
\begin{equation}\label{eqn:Cjl}
C(j, l) = \{\psi \in \Omega: \sigma_j(\psi) \in I_l \textup{ and } \sigma_j(\psi) \ge \sigma_i(\psi) \textup{ for all } 1\le i\le 2r\}
\end{equation}

By construction, we have that
\begin{align}\label{eqn:blah3}
\sum_{\psi \in C(j, l)} |G(\psi)| &= \sum_{\psi \in C(j, l)} \prod_{i=1}^{r_0}|G_i(\psi)| \le \sum_{\psi \in C(j, l)}\prod_{i=1}^{r_0} (\frak w_i)^{\sigma_j(\psi)}\\
& = \sum_{\psi \in C(j, l)} \frak u_F^{\sigma_j(\psi)} \ll \sum_{\psi \in C(j, l)} \frak u_F^{1/2+l/\mathfrak{L}},
\end{align}by the definition of $\sigma_i$ in \eqref{eqn:sigmaj}, since $\sigma_j(\psi) \ge \sigma_i(\psi) \textup{ for all } 1\le i\le r_0$ and since $\sigma_j(\psi) \in I_l$ for all $\psi \in C(j, l)$.

Since the number of classes $C(j, l)$ is bounded by $2\log V r_0 \ll \log V$, in order to prove \eqref{eqn:afterdecomp7} it suffices to show that
\begin{equation}\label{eqn:afterdecomp66}
\sum_{\psi \in C(j, l)} \left| G(\psi)\right| \ll \frak u_F \exp(-c(\log V)^{1/3 - \epsilon}).
\end{equation}
Let $R(j, l) = \# C(j, l)$.  By \eqref{eqn:blah3}, 
\begin{equation}\label{eqn:RuFbdd}
\sum_{\psi \in C(j, l)} \left| G(\psi)\right| \ll R(j, l)\frak u_F^{1/2+ l/\mathfrak{L}}.
\end{equation}
When $l=0$, the above is trivially bounded by $T_0^3 \frak u_F^{1/2} \ll  \frak u_F^{1-\epsilon}$, recalling $T_0 \ll V^{5/36}$ from \eqref{eqn:T_0} and $\frak u_F \gg V^{1-\epsilon_0}$ by construction from \eqref{eqn:uF}, which suffices.  Thus, to prove \eqref{eqn:afterdecomp66}, it suffices to show the following Proposition.

\begin{prop}\label{prop:afterdecomp8}
With notation as above and fixed $j, l$ with $1 \le l\le \mathfrak{L}$ and $1\le j\le 2r$, there exists some constant $c>0$ such that
\begin{equation}\label{eqn:afterdecomp8}
R(j, l)\frak u_F^{1/2+ l/\mathfrak{L}} \ll \frak u_F \exp(-c (\log V)^{1/3-\epsilon}).
\end{equation}
\end{prop}

\subsection{Proof of Proposition \ref{prop:afterdecomp8}}\label{subsec:largevalueapplication}
Since $j, l$ are fixed, for the rest of the proof, let us fix notation $C =C(j, l)$, $R = R(j, l)$, $\sigma(\psi) = \sigma_j(\psi)$, $G = G_j$, and $\frak w = \frak w_j$.  

We now split the proof into two cases, $\frak w \le T_0^{9/5}$ and $\frak w > T_0^{9/5}$.  Recall that for $\frak w \le T_0^{9/5}$ our $G = G_j = F_j = F$ as in \eqref{eqn:Gl2} for some $j$, while if $\frak w > T_0^{9/5}$, $G$ is of the form \eqref{eqn:Gl1}.

\subsubsection{Case 1: $\frak w \le T_0^{9/5}$}
Let
$$D_g = \sum_{n\leq \frak w} \tau_g(n)^2,
$$  where $\tau_g(n)$ denotes the number of ways to write $n$ as a product of $g$ natural numbers.  We will want to bound $D_g$ and the following crude bound suffices for our purposes.
\begin{lem}
With notation as above, uniformly in $g$,
\begin{equation}
D_g \ll \fw (\log \fw + \gamma)^{g^2 - 1}.
\end{equation}
\end{lem}
This is essentially identical to Lemma 2 of \cite{HBidentity} and for instance is immediately implied by the main result of Shiu in \cite{Shiu}.

We will apply Lemma \ref{lem:dukeupper} or Lemma \ref{lem:mont} to the polynomial $G^g = F^g = \sum_{M} a(M)$ say, where $|a(M)| \leq \tau_{3g}(N(M))$ by comparing the coefficients of $\zeta_K(s)^g$ with $\zeta(s)^{3g}$.  Then applying Lemma \ref{lem:dukeupper} to $G^g$ where if $\frak w \ge T_0^{9/5}$, $g = 2$ and otherwise we take $g$ to be any integer such that
$$T_0^{12/5} \le \frak w^g \le T_0^{18/5},
$$we see that
\begin{align*}
R 
&\ll \left( \fw^{g(1-2\sigma)} + T_0^3 \fw^{-2\sigma g}  (\log T_0)^8 \right) D_g\\
&\ll \left( \fw^{g(2-2\sigma)} + T_0^3 \fw^{(1-2\sigma) g}  (\log T_0)^8 \right) (\log \fw)^{3g^2 - 1}.
\end{align*}Thus, in the range $1/2 \le \sigma \le 3/4 $, the above gives that there exists some $A'>0$ such that
\begin{align}\label{eqn:Rbdd}
R \ll \begin{cases}
T_0^{\frac{36}{5}(1-\sigma)} (\log \fw+\gamma)^{3g^2 +A'} &\textup{ if } \fw\le T_0^{9/5}\\
\fw^{2(2-2\sigma)}(\log \fw+\gamma)^{3g^2 +A'} &\textup{ if } \fw \ge T_0^{9/5}.
\end{cases}
\end{align}
Indeed, the first line follows from $\fw^g \le T_0^{18/5}$ when $\fw \le T_0^{9/5}$.  The second line follows from the fact that $g = 2$ and $\fw^g > T_0^3$.  

We now proceed to prove that the same bound \eqref{eqn:Rbdd} holds for $3/4 < \sigma \le 1$ also.  We apply Lemma \ref{lem:huxmont} to the polynomial $G^g$ again choosing $g = 2$ for $\fw \ge T_0^{9/5}$ and otherwise $g$ to satisfy
$$T_0^{12/5} \le \frak w^g \le T_0^{18/5}.
$$
This gives that
\begin{equation}
R \ll (\fw^{g(2-2\sigma)} + T_0^3 \fw^{g(4-6\sigma)}) (\log \fw)^{3g^2+A'},
\end{equation}for some absolute constant $A' > 0$.  This implies that \eqref{eqn:Rbdd} holds when $3/4\le \sigma \le 1$ also.  Indeed, the case $\fw \le T_0^{9/5}$ follows from $T_0^{12/5}\le \fw^g \le T_0^{18/5}$ so that $\fw^{g(2-2\sigma)} \le T_0^{36/5}$ follows easily while $T_0^3 \fw^{g(4-6\sigma)} \le T_0^{36/5(1-\sigma)}$ follows by substituting $T_0^{12/5}$ for $\fw$ and using that $\sigma \ge 3/4$ so that $4-6\sigma \le 0$.

Now, by \eqref{eqn:RuFbdd} and \eqref{eqn:Rbdd}, when $\fw \le T_0^{9/5}$, we see that the contribution of the class $C$ is bounded by
$$T_0^{36/5(1-\sigma)} (\log \fw + \gamma)^{3g^2 + A'} \frak u_F^{\sigma} \ll \frak u_F \exp(-c \exp(\log V)^{1/3 - \epsilon})
$$since $(\log \fw + \gamma)^{3g^2 + A'} \ll \exp((\log V)^\epsilon)$ and
$$T_0^{36/5(1-\sigma)}  \ll \frak u_F^{(1-\sigma)(1-\epsilon)} \ll \frak u_F^{1-\sigma} \exp(-c(\log V)^{1/3-\epsilon})
$$using that $1-\sigma \gg \frac{1}{(\log V)^{2/3 + \epsilon}}$ by \eqref{eqn:sigmajbdd} since $\log \frak{u} \asymp \log V$ and recalling that $T_0 \ll V^{5/36 - \delta_0}$ for some $\delta_0 > 0$ from \eqref{eqn:T_0}.

\subsubsection{Case 2: $\frak w > T_0^{9/5}$}\label{subsubsec:T0big}
If $\fw > T_0^{9/5}$, recall that $G$ must be of the form \eqref{eqn:Gl1}.  We now prove the following Lemma.

\begin{lem}\label{lem:fwbigbdd}
With $G = G_l$ of the form \eqref{eqn:Gl1} and $\fw > T_0^{9/5}$, we have that there exists some constant $r>0$ such that
\begin{equation}
\sum_{\psi \in C(j, l)} \left|G(\psi)\right|^4 \ll \fw^2  T_0^3 (\log T_0)^r
\end{equation}
\end{lem}
\begin{proof}
For notational convenience, let 
$$\psi_1(M)  = \psi_1(i_1, i_2)(M) = v_0v_1^{i_1}v_2^{i_2}(M).$$  We will write $\psi(M) = \psi_1(M) N(M)^{-it}$.
Then by \eqref{eqn:Gl1} with a small change of variables and neglecting the factor $\mathfrak{w}^{-it-s}$ since it is size $1$,
\begin{align}
|G(\psi)| \le \int_{(1/2)} \left| L(w, \psi_1) \fw^{w} \tilde \omega(w-it-s) \right| dw.
\end{align}

On the line $\tRe w = 1/2$, $|\fw^{w}| = \fw^{1/2}$, so we also have
\begin{equation}\label{eqn:blah4}
\left|G(\psi)\right|^4 \le \fw^2 \left(\int_{(1/2)} \left| L(w, \psi_1) \tilde \omega(w-it-s) \right| dw\right)^4.
\end{equation}

For each $\psi_1$, let $S(\psi_1)$ be the set of well-spaced points of the form $\{t_1, t_2,...,t_n\}$ where $\psi(M) = \psi_1(M)N(M)^{-it_m} \in C(j, l)$ for each $1\le m\le n$.  Let $C_1(j, l)$ be the set of $\psi_1$ with nonempty $S(\psi_1)$.

We apply Holder's inequality and a bound for $\tilde{\omega}$ as in \eqref{eqn:omegatildebdd} to see that
\begin{align}
&\sum_{\psi_1 \in C_1(j, l)} \sum_{t\in S(\psi_1)} \fw^2 \left(\int_{(1/2)} \left| L(w, \psi_1) \tilde \omega(w-it-s) \right| dw\right)^4 \notag \\
&\leq \fw^2 \sum_{\psi_1 \in C_1(j, l)} \sum_{t\in S(\psi_1)} \left(\int_{(1/2)} \left|L(w, \psi)\right|^4 |\tilde \omega(w-it-s)| dw\right) \left(\int_{(1/2)}  |\tilde \omega(w)| dw\right)^3\\
&\ll \fw^2 \sum_{\psi_1 \in C_1(j, l)} \sumd_{T_1} I(T_1),
\end{align}where $\sumd_{T_1}$ denotes a dyadic sum over quantities $T_1 =2^k T_0$ over integers $k\ge 0$, and 
\begin{equation}
I(T_0) = \int_{1/2-iT_0}^{1/2+iT_0} \left|L(w, \psi)\right|^4   \sum_{t\in S(\psi_1)}  |\tilde{\omega}(w-it-s)| dw,
\end{equation}while for $T_1 \geq 2T_0$,
\begin{equation}
I(T_1) = \left(\int_{1/2-iT_1}^{1/2-iT_1/2} + \int_{1/2+iT_1/2}^{1/2+iT_1}\right) \left|L(w, \psi)\right|^4  \sum_{t\in S(\psi_1)}|\tilde{\omega}(w-it-s)| dw.
\end{equation} We claim
$$\sum_{t\in S(\psi_1)}|\tilde \omega(w-it-s)|  \ll_C \bfrac{T_0}{T_1}^{C}
$$for any $C>0$.  Indeed, the bound from \eqref{eqn:omegatildebdd} gives  
$$|\tilde \omega(w-it-s)| \ll_C \frac{1}{1+|w-it-s|^C}
$$for any $C>0$, so that
$$\sum_{t\in S(\psi_1)}|\tilde \omega(w-it-s)| \ll 1,
$$uniformly in $w$ and $s$ since $S(\psi_1)$ is well spaced.  To be specific, $|t_i - t_j| \ge 18\log T_0$ for $t_i \neq t_j$ elements of $S(\psi_1)$ by \eqref{cond:wellspaced}.  This gives the claimed bound when $T_1 \le 4T_0$.  

When $T_1 > 4T_0$, $|\tIm w| > 2T_0$, and since $|t| \leq T_0$ and $|s| \leq V^\epsilon \leq T_0/2 = \frac{V^\epsilon}{2\Delta}$ for sufficiently large $X$ (and $V$) upon recalling the definition of $T_0$ from \eqref{eqn:T0}, $|w-it-s| \asymp |w| \asymp T_1$, whence the stronger bound
$$\sum_{t\in S(\psi_1)}|\tilde \omega(w-it)| \ll_C \frac{1}{T_1^C},
$$holds for any $C>0$.

On the other hand, the fourth moment
$$\sum_{\psi_1 \in C_1(j, l)} \int_{1/2-iT_1}^{1/2+iT_1} \left|L(w, \psi)\right|^4 dw \ll T_1^3 (\log T_1)^r
$$for some $r>0$ by Theorem 2.2 in Duke's work \cite{Du}.  Duke states his result for fixed $\theta$ with conductor $q$ with the implied constant dependent on $q$, but following the proof, it is clear that the claimed bound above still holds.  For this, we recall that $q \le (\log X)^R$ for some $R$ as in the statement of Lemma \ref{lem:SWbdd} and $T_1 \ge T_0 \gg X^\epsilon$ so $\log X \ll \log T_1$.  Following Duke's proof, the dependence on $q$ is polynomial\footnote{This arises from an extra factor of $\sqrt{q}$ in the length of the Dirichlet polynomial approximations occuring in the approximate functional equation.}, so it may be absorbed into the power $(\log T_1)^r$.  

From this, we see that
\begin{equation}
\sum_{\psi \in C(j, l)} \left|G(\psi)\right|^4  \ll \fw^2  \sumd_{T_1} \bfrac{T_0}{T_1}^C T_1^3 (\log T_1)^r \ll \fw^2  T_0^3 (\log T_0)^r
\end{equation}
upon choosing $C = 4$ for instance.
\end{proof}

Lemma \ref{lem:fwbigbdd} immediately implies that
$$R \fw^{4\sigma} \ll \fw^2  T_0^3 (\log T_0)^r,
$$from which we conclude that
$$R \ll \fw^{2-4\sigma} T_0^3  \ll \left(T_0^{3} \right)^{1+\frac{3}{5}(2-4\sigma)} \ll T_0^{36/5(1-\sigma)},
$$where we have put in $T_0^{9/5}$ for $\fw$ valid since $2-4\sigma \le 0$.  The desired result follows as before.

\subsection{Improved zero free region type bound}\label{subsec:zerofree}
Here, we prove Lemma \ref{lem:zerofreebdd}.  The main ingredient appears in Coleman's work on an improved zero free region for Hecke L-functions in \cite{Cole}.  In our case, we are examining $L(s, \psi)$ where $\psi = v_0 v_1^jv_2^k$ for integers $j, k$.  Recall from Remark \ref{rem:characterchi} that $v_0$ involved a character $\theta$ with modulus $\asymp q \ll (\log X)^R$.  We write $\psi = \theta \lambda$ and set the analytic conductor of $L(s, \psi) = L(\sigma+it, \theta \lambda)$ to be $\fC = \fC(t) = j^2 + k^2 + t^2 + 15$, where by design, $\log \log \fC >0$.  We have not included a $q$ dependence in $\fC$ for convenience when citing Coleman's result below.  

Theorem 1 from Coleman's work \cite{Cole} implies that there exists a constant $c_2 > 0$ such that for $\sigma \ge \sigma_0 := 1-\frac{c_2}{\log^{2/3} \fC}$
\begin{equation}\label{eqn:Lcolemanbound0}
L(\sigma+it, \theta\lambda) \ll q^{1-\sigma_0} \log^{2/3} \fC.
\end{equation}

Theorem 2 from \cite{Cole} implies that $$L(\sigma+it, \theta \lambda) \neq 0$$ for all $\sigma \geq 1- \frac{c_2}{\max(\log^{2/3}\fC (\log \log \fC)^{1/3}, \log q)}$ aside for possibly an exceptional real zero $\beta$ satisfying
$$1-\beta \gg_\epsilon \frac{1}{q^\epsilon},
$$for any $\epsilon>0$ where the implied constant is ineffective, the latter bound being an earlier result of Fogels \cite{Fogels}.  

In our application, we will always have that $\log \fC \ll \log V$ and $q \ll (\log V)^A$ for some constant $A$ for the same $V$ as in Proposition \ref{prop:bilinear2}.  Thus, we now restate the above two results in terms of $V$ instead.  First \eqref{eqn:Lcolemanbound0} implies that
there exists a constant $c_2 > 0$ such that for $\sigma \ge \sigma_0 := 1-\frac{c_2}{\log^{2/3} V}$
\begin{equation}\label{eqn:Lcolemanbound}
L(\sigma+it, \theta\lambda) \ll \log^{2/3} V.
\end{equation}

Similarly Theorem 2 from \cite{Cole} and Fogel's result \cite{Fogels} implies that $$L(\sigma+it, \theta \lambda) \neq 0$$ for all $\sigma \geq 1- \frac{c_2}{\log^{2/3}V (\log \log V)^{1/3}}$ where the constant $c_2>0$ is ineffective.  


A standard argument converts these two statements into an upper bound for $\frac{1}{L(s, \theta \lambda)}$ near the $\tRe s = 1$ line.  To be precise, we have the following Lemma.
\begin{lem} \label{lem:logderivbdd}
There exists an ineffective constant $c_1>0$ such that 
$$\frac{L'}{L}(\sigma + it, \theta \lambda) \ll (\log V)^{2/3} \log\log V.
$$for all $\sigma > 1-\frac{c_1}{\log^{2/3}V}$.  
\end{lem}
\begin{proof}
Lemma $\alpha$ in \S 3.9 of \cite{titch} implies that for $s_0 = 1+(\log V)^{-2/3}$, and $r = (\log V)^{-2/3} + \frac{c_2}{\log^{2/3}V}$ that for any $s$ such that $|s-s_0| \le r/4$ that
$$\left|\frac{L'}{L}(s, \theta \lambda) \right| \ll M (\log^{2/3} V)
$$provided that
$$\frac{L(s, \theta \lambda)}{L(s_0, \theta \lambda)} < e^{M}
$$for all $|s-s_0| \le r$.  On the other hand, a comparison of $L(s_0, \theta \lambda)$ with $\zeta_K(s_0)$ implies that $L(s_0, \theta \lambda) \gg \log^{-2/3} V$, and combining this with Coleman's result above implies that
$$\frac{L(s, \theta \lambda)}{L(s_0, \theta \lambda)} < e^{M_0 \log \log V}
$$for some constant $M_0$.
\end{proof}

Now write $s = \sigma + it$ with $\sigma > 1-\frac{1}{\log^{2/3} V \log \log V}$, and let $\sigma' = 1+\frac{1}{\log^{2/3} V \log \log V}$.  Then
\begin{align*}
-\tRe \log L(s, \lambda \theta) 
&= -\tRe \log L(\sigma' + it, \lambda \theta)+ \int_{\sigma}^{\sigma'} \tRe \frac{L'}{L} (u+it, \lambda \theta) du\\
&\le \log \left(C_0 \log^{2/3} V \log \log V \right) + O(1),
\end{align*}for some $C_0 > 0$ and by Lemma \ref{lem:logderivbdd}.  We therefore have that for $\sigma> 1- \frac{1}{\log^{2/3} V \log \log V}$,
\begin{equation}\label{eqn:1Lbound}
\frac{1}{L(s, \lambda \theta)} \ll \log^{2/3} V \log \log V.
\end{equation}
Now, recall that we want to prove
\begin{equation}\label{eqn:rbound}
\sum_{N(M) \le \fw} r(M) \ll \fw \exp(-c(\log \fw)^{1/3 - \epsilon}),
\end{equation}for some $c$ where $\fw \gg V^{\delta_0 \delta}$ and either

\begin{equation}\label{eqn:r1}
r(M) = \mu_K(\beta) \lambda\theta(M) N(M)^{-it-s},
\end{equation}or
\begin{equation}\label{eqn:r2}
r(M) = \frac{1}{(1+|\iota s|)^2} W\bfrac{N(M)}{R} \lambda\theta(M) N(M)^{-it-s},
\end{equation}where $t$ and $s$ are fixed parameters satisfying $|t| \leq T_0$ and $|s| \leq V^\epsilon$.  Moreover, $W(u)$ is a smooth compactly supported function satisfying
$$W^{(k)}(u) \ll_k \frac{1}{u^k},
$$and integration by parts yields
\begin{equation}\label{eqn:tildeW}
\tilde{W}(w) \ll \frac{1}{1+|w|^k}
\end{equation}for all $k\ge 0$.
Moreover, if $\lambda \theta$ is trivial, we must have $|t| \ge \tau_0$.  

When $|s| \ge \tau_0/2$ and $r(M)$ is of the form in \eqref{eqn:r2}, the claimed bound is immediate, since then 
$$ \frac{1}{(1+|\iota s|)^2} \ll \exp(-c (\log V)^{1/3 - \epsilon}),
$$upon recalling \eqref{eqn:iota} \eqref{eqn:tau0def}, and \eqref{eqn:Delta0}.

Otherwise, the desired bound then follows by standard arguments which we sketch.  In the first case \eqref{eqn:r1}, we write
\begin{equation}
\sum_{N(\beta) \le \fw} r(\beta) = \frac{1}{2\pi i} \int_{(2)} \frac{1}{L(w+s+it, \lambda \theta)} \frac{\fw^w}{w} dw.   
\end{equation}Truncating the integral and shifting to $\tRe w = 1-\frac{1}{ \log^{2/3} V \log\log V} \ge 1 - O\bfrac{1}{\log^{2/3 + \epsilon} V}$  yields the claim \eqref{eqn:rbound} upon applying the bound \eqref{eqn:1Lbound}, and recalling that $\fw \gg V^{1/(\log V)^{\epsilon}}.$
In the second case \eqref{eqn:r2} and $|s| < \tau_0/2$, we write
\begin{equation}
\sum_{N(\beta) \le \fw} r(\beta) = \frac{1}{2\pi i} \int_{(2)} L(s+w+it, \lambda \theta) R^w \tilde W(w) dw,   
\end{equation}and shift contours to $\sigma_0 = 1-\frac{c_2}{\log^{2/3}V}$. If $\lambda \theta$ is trivial, the residue at $w = 1-it - s$ gives a contribution $\ll R \tilde{W}(1-it-s)$.  Since we have $|s| \le \tau_0/2$ while $|t| \ge \tau_0$, this is $\ll R/\tau_0$ by \eqref{eqn:tildeW}, which is acceptable as before for our claimed bound.  

On the line $\tRe w = \sigma_0 = 1-\frac{c_2}{\log^{2/3}V}$, the bound \eqref{eqn:rbound} follows from the bound \eqref{eqn:Lcolemanbound}, the rapid decay of $\tilde W$ from \eqref{eqn:tildeW} and of course $R\gg  V^{1/(\log V)^{\epsilon}}.$

\end{document}